\documentclass[12pt, reqno]{amsart}
\usepackage{amssymb}
\usepackage{aliascnt}
\usepackage{graphicx}
\usepackage{mathrsfs}
\usepackage{enumerate}
\usepackage[bookmarks=true]{hyperref}
\usepackage[centredisplay,PostScript=dvips]{diagrams}
\usepackage{upref}

\newtheorem{theorem}{Theorem}[section]
\newaliascnt{conj}{theorem}
\newaliascnt{cor}{theorem}
\newaliascnt{lemma}{theorem}
\newaliascnt{fact}{theorem}
\newaliascnt{claim}{theorem}
\newaliascnt{prop}{theorem}
\newaliascnt{definition}{theorem}
\newaliascnt{perasum}{theorem}

\newtheorem{conj}[conj]{Conjecture}
\newtheorem{cor}[cor]{Corollary}
\newtheorem{lemma}[lemma]{Lemma}

\newtheorem{prop}[prop]{Proposition}
\newtheorem{definition}[definition]{Definition}
\newtheorem{perasum}[perasum]{Perelman's Smoothness Assumption}

\aliascntresetthe{conj} \aliascntresetthe{cor}
\aliascntresetthe{lemma} \aliascntresetthe{fact}
\aliascntresetthe{claim} \aliascntresetthe{prop}
\aliascntresetthe{definition} \aliascntresetthe{perasum}

\def\sek~{\S{}}

\theoremstyle{definition}
\newaliascnt{example}{theorem}
\newtheorem{example}[example]{Example}
\aliascntresetthe{example}

\theoremstyle{remark}
\newtheorem{remark}[theorem]{Remark}
\numberwithin{equation}{section}
\newarrow {Corresponds}{<}{-}{-}{-}{>}
\newcommand{\curv}{{\rm curv}}
\newcommand{\vol}{{\rm Vol}}

\newcommand{\diam}{{\rm diam}}

\newcommand{\interior}{{\rm int}}
\newcommand{\dis}{{\rm d}}
\newcommand{\Ric}{{\rm Ric}}
\newcommand{\dbl}{{\rm dbl}}

\begin{document}

\title[Perelman's collapsing theorem for $3$-manifolds]
{A simple proof of Perelman's collapsing theorem for $3$-manifolds}
\author{Jianguo Cao}
\thanks{
This final version was refereed and accepted for publication in {\it ``The Journal of Geometric Analysis"}.  The first author was supported in part by Nanjing
University and an NSF grant.}
\author{Jian Ge}
\begin{abstract}
{\small We will simplify earlier proofs of Perelman's collapsing
theorem for $3$-manifolds given by Shioya-Yamaguchi
\cite{SY2000}-\cite{SY2005} and Morgan-Tian \cite{MT2008}. A version
of Perelman's collapsing theorem states: {\itshape ``Let $\{M^3_i\}$
be a sequence of compact Riemannian $3$-manifolds with curvature
bounded from below by $(-1)$ and $\diam(M^3_i) \ge c_0
> 0$. Suppose that all unit metric balls in $M^3_i$ have very small
volume at most $v_i \to 0$ as $i \to \infty$ and suppose that either
$M^3_i$ is closed or has possibly convex incompressible toral
boundary. Then $M^3_i $ must be a graph-manifold for sufficiently
large $i$"}. This result can be viewed as an extension of the implicit
function theorem. Among other things, we apply Perelman's critical
point theory (e.g., multiple conic singularity theory and his
fibration theory) to Alexandrov spaces to construct the desired
local Seifert fibration structure on collapsed $3$-manifolds.

The verification of Perelman's collapsing theorem is the last step
of Perelman's proof of Thurston's Geometrization Conjecture on the
classification of $3$-manifolds. A version of Geometrization Conjecture asserts
that any closed 3-manifold admits a  {\it piecewise locally homogeneous metric}.
Our proof of Perelman's collapsing
theorem is accessible to  advanced graduate students and non-experts.
}
\end{abstract}
\maketitle

\tableofcontents
\section{Introduction}\label{section0}

In this paper, we present a simple proof of Perelman's collapsing
theorem for $3$-manifolds (cf. Theorem 7.4 of \cite{Per2003a}),
which Perelman used to verify Thurston's Geometrization Conjecture
on the classification of $3$-manifolds.

\subsection{Statement of Perelman's collapsing theorem and
the main difficulties in its proof.}

\begin{theorem}[\textbf{Perelman's Collapsing Theorem}]\label{thm0.1}
Let $\{ (M^3_\alpha, g_{ij}^\alpha)\}_{\alpha \in \mathbb Z}$ be a
sequence of compact oriented Riemannian $3$-manifolds, closed or
with convex incompressible toral boundary, and let
$\{\omega_\alpha\}$ be a sequence of positive numbers with $
\omega_\alpha \to 0$. Suppose that

{\rm (1)} for each $x \in M^3_{\alpha} $ there exists a radius $\rho
= \rho_{\alpha}(x) $, $0 < \rho < 1$, not exceeding the diameter of
the manifold, such that the metric ball $B_{g^{\alpha}}(x, \rho)$ in
the metric $ g_{ij}^{\alpha} $ has volume at most $\omega_{\alpha}
\rho^3$ and sectional curvatures of $g^\alpha_{ij}$ at least $ -
\rho^{-2}$;

{\rm (2)} each component of toral boundary of $M^3_\alpha$ has
diameter at most $\omega_\alpha$, and has a topologically trivial
collar of length one and the sectional curvatures of $M^3_\alpha$
are between $(-\frac{1}{4} - \varepsilon) $ and $ (-\frac{1}{4} +
\varepsilon)$.

Then, for sufficiently large $\alpha$, $M^3_{\alpha}$ is
diffeomorphic to a graph-manifold.
\end{theorem}

If a $3$-manifold $M^3$ admits an almost free circle action, then we
say that $M^3$ admits a {\it Seifert fibration} structure or {\it
Seifert fibred}. A {\it graph-manifold} is a compact $3$-manifold
that is a connected sum of manifolds each of which is either
diffeomorphic to the solid torus or can be cut apart along a finite
collection of incompressible tori into Seifert fibred $3$-manifolds.

Perelman indeed listed an extra assumption (3) in \autoref{thm0.1}
above. However, Perelman (cf. page 20 of \cite{Per2003a}) also
pointed out that, if the proof of his stability theorem (cf.
\cite{Kap2007}) is available, then his extra smoothness assumption
(3) is in fact redundant.

The conclusion of \autoref{thm0.1} fails if the assumption of toral
boundary is removed. For instance, the product $3$-manifold $S^2
\times [0, 1]$ of a $2$-sphere and an interval can be collapsed
while keeping curvatures non-negative. However, $M^3 = S^2 \times
[0, 1]$ is not a graph-manifold.

Under an assumption on both upper and lower bound of curvatures $-
\frac{C}{\rho^{2}} \le \curv_{M^3_\alpha} \le \frac{C}{ \rho^{2}}$,
the collapsed manifold $M^3_\alpha$ above admits an $F$-structure of
positive rank, by the Cheeger-Gromov collapsing theory (cf.
\cite{CG1990}). It is well-known that a $3$-manifold $M^3_\alpha$
admits an $F$-structure of positive rank if and only if $M^3_\alpha$
is a graph-manifold, (cf. \cite{R1993}).

On page 153 of \cite{SY2005}, Shioya and Yamaguchi stated a version
of \autoref{thm0.1} above for the case of closed manifolds, but
their proof works for case of manifolds with incompressible convex
toral boundary as well. Morgan and Tian \cite{MT2008} presented a
proof of \autoref{thm0.1} without assumptions on diameters but with
Perelman's extra smoothness assumption (3) which we discuss below.

It turned out that the diameter assumption is related to the study
of diameter-collapsed $3$-manifolds with curvature bounded from
below. To see this relation, we state a local re-scaled version of
Perelman's collapsing \autoref{thm0.1}.

\medskip
\noindent\textbf{Theorem 0.1'.} (Re-scaled version of
\autoref{thm0.1}) {\itshape Let $\{ (M^3_\alpha,
g_{ij}^\alpha)\}_{\alpha \in \mathbb Z} $ be and let
$\{\omega_\alpha\}$ be a sequence of positive numbers as in
\autoref{thm0.1} above, $x_\alpha \in M^3_\alpha$ and $ \rho_\alpha=
\rho_\alpha (x_\alpha)$. Suppose that there exists a re-scaled
family of pointed Riemannian manifolds $\{ ((M^3_\alpha,
\rho^{-2}_\alpha g_{ij}^\alpha), x_\alpha)\}_{\alpha \in \mathbb Z}
$ satisfying the following conditions:
\begin{enumerate}[{\rm (i)}]
\item The re-scaled Riemannian manifold $(M^3_\alpha,
\rho^{-2}_\alpha g_{ij}^\alpha)$ has $\curv \ge -1$ on the ball
$B_{\rho^{-2}_\alpha g_{ij}^\alpha }(x_\alpha, 1)$;

\item The diameters of the re-scaled manifolds $\{( M^3_\alpha,
\rho^{-2}_\alpha g_{ij}^\alpha) \}$ are uniformly bounded from below
by $1$; i.e.;
\begin{equation}\label{eq:0.1}
\diam( M^3_\alpha, \rho^{-2}_\alpha g_{ij}^\alpha) \ge 1
\end{equation}

\item The volumes of unit metric balls collapse to zero, i.e.:
$$\vol[ B_{\rho^{-2}_\alpha g_{ij}^{\alpha}}(x_{\alpha}, 1)] \le
\omega_\alpha \to 0,$$ as $\alpha \to \infty$.
\end{enumerate}

Then, for sufficiently large $\alpha$, the collapsing $3$-manifold
$M^3_\alpha$ is a graph-manifold.}

\medskip

Without inequality \eqref{eq:0.1} above, the volume-collapsing
$3$-manifolds could be diameter-collapsing. Perelman's condition
\eqref{eq:0.1} ensures that the normalized family $\{( M^3_\alpha,
\rho^{-2}_\alpha g_{ij}^\alpha)\}$ can {\it not} collapse to a point
uniformly. By {\it collapsing to a point uniformly}, we mean that
there is an additional family of scaling constants $\{
\lambda_\alpha\} $ such that the sequence $\{ (M^3_\alpha,
\lambda_\alpha g_{ij}^\alpha)\}$ is convergent to a $3$-dimensional
(possibly singular) manifold $Y^3_\infty$ with non-negative
curvature. Professor Karsten Grove kindly pointed out  that the
study of diameter-collapsing theory for $3$-manifolds might be
related to a weak version of the Poincar\'{e} conjecture. For this
reason, Shioya and Yamaguchi made the following conjecture.

\begin{conj}[Shioya-Yamaguchi \cite{SY2000} page 4]\label{conj0.2}
Suppose that $Y^3_\infty $ is a $3$-dimensional compact,
simply-connected, non-negatively curved Alexandrov space without
boundary and that $Y^3_\infty$ is a topological manifold. Then
$Y^3_\infty$ is homeomorphic to a sphere.
\end{conj}

Shioya and Yamaguchi (\cite{SY2000}, page 4) commented that {\it if
\autoref{conj0.2} is true then the study of collapsed $3$-manifolds
with curvature bounded from below would be completely understood}.
They also observed that \autoref{conj0.2} is true for a special case
when the closed (possibly singular) manifold $Y^3_\infty$ above is a
{\it smooth} Riemannian manifold with non-negative curvature; this
is due to Hamilton's work on $3$-manifolds with non-negative Ricci
curvature (cf. \cite{H1986}).

Coincidentally, Perelman added the extra smoothness assumption (3)
in his collapsing theorem.

\begin{perasum}[cf. \cite{Per2003a}]\label{asum0.3}
For every $w' > 0$ there exist $r = r(w') > 0$ and constants $K_m =
K_m(w') < \infty$ for $m = 0, 1, 2, \cdots$, such that for all
$\alpha$ sufficiently large, and any $0 < r \le \bar{ r}$, if the
ball $B_{g_\alpha}(x, r)$ has volume at least $w'r^3$ and sectional
curvatures at least $-r^{-2}$, then the curvature and its $m$-th
order covariant derivatives at $x$ are bounded by $K_0 r^{-2}$ and
$K_mr^{-m-2}$ for $m = 1, 2,\cdots,$ respectively.
\end{perasum}

Let us explain how Perelman's Smoothness \autoref{asum0.3} is
related to the smooth case of \autoref{conj0.2} and let $\{
((M^3_{\alpha}, \rho^{-2}_\alpha g_{ij}^\alpha), x_\alpha)\}_
{\alpha \in \mathbb Z} $ be as in Theorem 0.1'. If we choose the new
scaling factor $\lambda^2_{\alpha}$ such that $\lambda^2_{\alpha} /
\rho^{-2}_{\alpha} \to +\infty$ as $\alpha \to \infty$, then the
newly re-scaled metric $\lambda^2_{\alpha} g_{ij}^\alpha$ will have
sectional curvature $\ge -
\frac{\rho^{-2}_{\alpha}}{\lambda^2_{\alpha}} \to 0$ as $\alpha \to
\infty$. Suppose that $(Y_{\infty}, y_{\infty}) $ is a pointed
Gromov-Hausdorff limit of a subsequence of $\{ ((M^3_{\alpha},
\lambda^2_\alpha g_{ij}^\alpha), x_\alpha)\}_{\alpha \in \mathbb Z}
$. Then the limiting metric space $Y_{\infty}$ will have
non-negative curvature and a possibly singular metric. When
$\dim[Y_\infty] = 3$, by Perelman's Smoothness \autoref{asum0.3},
the limiting metric space $Y_\infty$ is indeed a {\it smooth}
Riemannian manifold of non-negative curvature. In this smooth case,
\autoref{conj0.2} is known to be true, (see \cite{H1986}).

For simplicity, we let $\hat g_{ij}^\alpha = \rho^{-2}_\alpha
g_{ij}^\alpha$ be as in Theorem 0.1'. By Gromov's compactness
theorem, there is a subsequence of a pointed Riemannian manifolds
$\{ ((M^3_\alpha, \hat{ g}_{ij}^\alpha), x_\alpha) \}$ convergent to
a lower dimensional pointed space $(X^k, x_\infty)$ of dimension
either $1$ or $2$, i.e.:
\begin{equation}\label{eq:0.2}
1 \le \dim [ X^k] \le 2
\end{equation}
using \eqref{eq:0.1}. To establish \autoref{thm0.1}, it is
important to establish that, for sufficiently large $\alpha$, the
collapsed manifold $M^3_\alpha$ has a decomposition $M^3_\alpha =
\cup_i U_{\alpha, i}$ such that each $U_{\alpha, i}$ admits an
almost-free circle action:
$$ S^1 \to U_{\alpha, i } \to X^2_{\alpha, i}$$
We also need to show that these almost-free circle actions are
compatible (almost commute) on possible overlaps.

Let us first recall how Perelman's collapsing theorem for
$3$-manifolds plays an important role in his solution to Thurston's
Geometrization Conjecture on the classification of $3$-manifolds.

\subsection{Applications of collapsing theory to the
classification of $3$-manifolds.} \

\medskip

In 2002-2003, Perelman posted online three important but densely
written preprints on Ricci flows with surgery on compact
$3$-manifolds (\cite{Per2002}, \cite{Per2003a} and \cite{Per2003b}),
in order to solve both the Poincar\'{e} conjecture and Thurston's
conjecture on Geometrization of $3$-dimensional manifolds.
Thurston's Geometrization Conjecture states that {\it ``for any
closed, oriented and connected $3$-manifold $M^3$, there is a
decomposition $[M^3 -\bigcup \Sigma^2_j] = N^3_1 \cup N^3_2 \cdots
\cup N^3_{m}$ such that each $N^3_i$ admits a locally homogeneous
metric with possible incompressible boundaries $\Sigma^2_j$, where
$\Sigma^2_j$ is homeomorphic to a quotient of a $2$-sphere or a
$2$-torus".} There are exactly 8 homogeneous spaces in dimension 3.
The list of $3$-dimensional homogeneous spaces includes 8
geometries: $\mathbb{R}^3$, $\mathbb{H}^3$, $\mathbf{S}^3$,
$\mathbb{H}^2 \times \mathbb{R}$, $\mathbf{S}^2 \times \mathbb{R}$,
$\tilde{SL}(2, \mathbb{R})$, $Nil$ and $Sol$.

Thurston's Geometrization Conjecture suggests the existence of
especially nice metrics on $3$-manifolds and consequently, a more
analytic approach to the problem of classifying $3$-manifolds.
Richard Hamilton formalized one such approach by introducing the
Ricci flow equation on the space of Riemannian metrics:
\begin{equation}\label{eq:0.3}
\frac{\partial g(t)}{\partial t} = - 2\Ric(g(t))
\end{equation}
where $\Ric(g(t))$ is the Ricci curvature tensor of the metric
$g(t)$. Beginning with any Riemannian manifold $(M, g_0)$, there is
a solution $g(t)$ of this Ricci flow on $M$ for $t$ in some interval
such that $g(0) = g_0$. In dimension $3$, the fixed points (up to
re-scaling) of this equation include the Riemannian metrics of
constant Ricci curvature. For instance, they are quotients of
$\mathbb{R}^3$, $\mathbb{H}^3$ and $\mathbf{S}^3$ up to scaling
factors. It is easy to see that, on compact quotients of
$\mathbb{R}^3 $ or $\mathbb{H}^3$, the solution to Ricci flow
equation is either stable or expanding. Thus, on compact quotients
of $\mathbb{R}^3 $ and $\mathbb{H}^3$, the solution to Ricci flow
equation exists for all time $t \ge 0$.

However, on quotients of $\mathbf{S}^2 \times \mathbb{R}$ or
$\mathbf{S}^3$, the solution $\{g(t)\}$ to Ricci flow equation
exists only for finite time $t < T_0$. Hence, one knows that in
general the Ricci flow will develop singularities in finite time,
and thus a method for analyzing these singularities and continuing
the flow past them must be found.

These singularities may occur along proper subsets of the manifold,
not the entire manifold. Thus, Perelman introduced a more general
evolution process called Ricci flow with surgery (cf. \cite{Per2002}
and \cite{Per2003a}). In fact, a similar process was first
introduced by Hamilton in the context of four-manifolds. This
evolution process is still parameterized by an interval in time, so
that for each $t$ in the interval of definition there is a compact
Riemannian $3$-manifold $M_t$. However, there is a discrete set of
times at which the manifolds and metrics undergo topological and
metric discontinuities (surgeries). Perelman did surgery along
$2$-spheres rather than surfaces of higher genus, so that the change
in topology for $\{M^3_t\}$ turns out to be completely understood.
More precisely, Perelman's surgery on $3$-manifolds is {\it the
reverse process of taking connected sums:} cut a $3$-manifold along
a $2$-sphere and then cap-off by two disjoint $3$-balls. Perelman's
surgery processes produced exactly the topological operations needed
to cut the manifold into pieces on which the Ricci flow can produce
the metrics sufficiently controlled so that the topology can be
recognized. It was expected that each connected components of
resulting new manifold $M^3_t$ is either a graph-manifold or a
quotient of one of homogeneous spaces listed above. It is well-known
that any graph-manifold is a union of quotients of 7 (out of the 8
possible) homogeneous spaces described above. More precisely,
Perelman presented a very densely written proof of the following
result.

\begin{theorem}[Perelman \cite{Per2002},\cite{Per2003a}, \cite{Per2003b}]\label{thm0.4}
Let $(M^3, g_0)$ be a closed and oriented Riemannian $3$-manifold.
Then there is a Ricci flow with surgery, say $\{ (M_t, g(t))\}$,
defined for all $t \in [0, T)$ with initial metric $(M, g_0)$. The
set of discontinuity times for this Ricci flow with surgery is a
discrete subset of $[0,\infty)$. The topological change in the
$3$-manifold as one crosses a surgery time is a connected sum
decomposition together with removal of connected components, each of
which is diffeomorphic to one of $(S^2 \times S^1)/\Gamma_i$ or
$S^3/\Gamma_j$. Furthermore, there are two possibilities:

{\rm (1)} Either the Ricci flow with surgery terminates at a finite
time $T$. In this case, $M^3_0$ is diffeomorphic to the connect sum
of $(S^2 \times S^1)/\Gamma_i$ and $S^3/\Gamma_j$. In particular, if
$M^3_0$ is simply-connected, then $M^3_0$ is diffeomorphic to $S^3$.

{\rm (2)} Or the Ricci flow with surgery exists for all time, i.e.,
$T= \infty$.
\end{theorem}

The detailed proof of Perelman's theorem above can be found in
\cite{CZ2006}, \cite{KL2008} and \cite{MT2007}.
In fact, if $M^3$ is a simply-connected closed manifold, then using
a theorem of Hurwicz, one can find that $ \pi_3(M^3) = H_3(M^3,
\mathbb Z)=\mathbb Z$. Hence, there is a map $F: S^3 \to M^3$ of
degree $1$. One can view $S^3$ as a two-point suspension of $S^2$,
i.e., $S^3 \sim S^2 \times [0, 1]/\{0, 1\}$. Thus, for such a
manifold $M^3$ with a metric $g$, one can define the $2$-dimensional
width $W_2(M, g) = \inf_{deg(F) = 1} \max_{0\le s \le 1} \{Area_{g}
[F(S^2, s) ] \}$, (compare with [CalC92]). Colding and Minicozzi (cf.
\cite{CM2005},\cite{CM2008a},\cite{CM2008b}) established that
\begin{equation}\label{eq:0.4}
0< W_2(M^3_t, g(t)) \le (t + c_1)^{\frac 34} [c_2 - 16\pi
(t+c_3)^{\frac 14} )]
\end{equation}
for Perelman's solutions $\{M^3_t, g(t)\}$ to Ricci flow with
surgery, where $\{c_1, c_2, c_3\}$ are positive constants
independent of $t \in [0, T) $. Therefore, for a simply-connected
closed manifold $M^3_0$, it follows from \eqref{eq:0.4} that the
Ricci flow with surgery must end at a finite time $T$. Thus, by
\autoref{thm0.4}, the conclusion of the Poincar\'{e} conjecture
holds for such a simply-connected closed $3$-manifold $M^3_0$. Other
proofs of Perelman's finite time extinction theorem can be found in
\cite{Per2003b} and \cite{MT2007}.

It remains to discuss the long time behavior of Ricci curvature flow
with surgery. We will perform the so-called Margulis thick-thin
decomposition for a complete Riemannian manifold $(M^n, g)$. The
thick part is the {\it non-collapsing part} of $(M^n, g)$, while the
thin part is the {\it collapsing } portion of $(M^n, g)$.

Let $\rho(x, t) $ denote the radius $\rho$ of the metric ball
$B_{g(t)}(x, \rho)$, where we may choose $\rho(x, t) $ so that
$$
\inf\{ \sec_{g(t)}|_y \quad| \quad y \in B_{g(t)}(x,\rho)\}\ge -
\rho^{-2}
$$
and $ \frac 12 \rho(x,t) \le \rho(y, t)\le \rho(x, t) $ for all $y
\in B_{g(t)}(x, \rho/2)$. The re-scaled thin part of $M_t$ can be
defined as
\begin{equation}\label{eq:0.5}
M_-(\omega, t) = \{ x \in M^3_t \quad | \quad \vol[ B_{g(t)}(x,
\rho(x, t) )] < \omega \rho^3(x, t) \}
\end{equation}
and its complement is denoted by $M_+(\omega, t)$, which is called
the thick part for a fixed positive number $\omega$.

When there is no surgery occurring for all time $t \ge 0$, Hamilton
successfully classified the thick part $M_+(\omega, t)$.

\begin{theorem}[Hamilton \cite{H1999},
non-collapsing part]\label{thm0.5} Let $\{ (M^3_t, g(t)) \}$,
$M_-(\omega, t)$ and $M_+(\omega, t)$ be as above. Suppose that
$M^3_t$ is diffeomorphic to $M^3_0$ for all $t \ge 0$. Then there
are only two possibilities:

\smallskip
\noindent {\rm (1)} If there is no thin part (i.e., $M_-(\omega, t)
= \varnothing$), then either $(M^3_t, g(t))$ is convergent to a flat
$3$-manifold or $(M^3, \frac{2}{t} g(t))$ is convergent to a compact
quotient of hyperbolic space $\mathbb H^3$;

\smallskip
\noindent {\rm (2)} If both $M_+(\omega, t)$ and $M_-(\omega, t)$
are non-empty, then the thick part $M_+(\omega, t)$ is diffeomorphic
to a disjoint union of quotients of real hyperbolic space $\mathbb
H^3$ with finite volume and with cuspidal ends removed.
\end{theorem}

\begin{figure*}[ht]
\includegraphics[width=250pt]{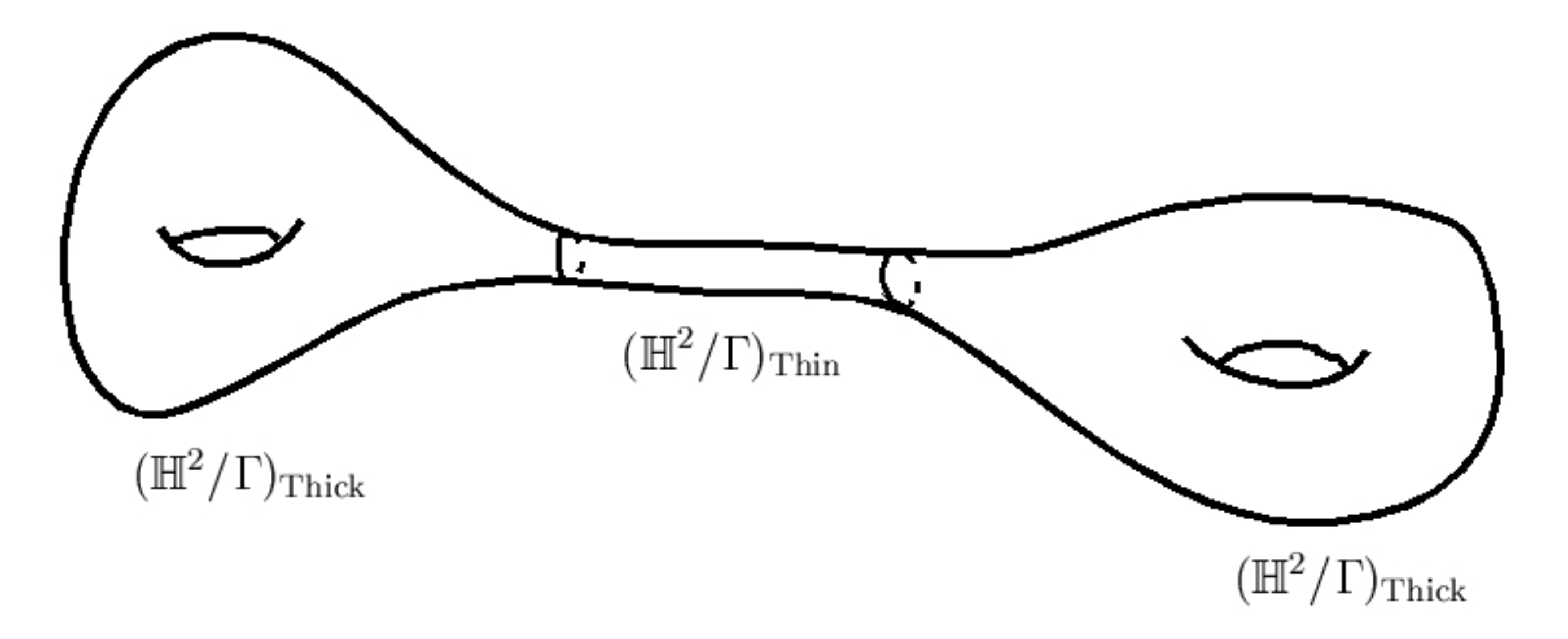}\\
\caption{2-dimensional thick-thin decomposition }\label{2dimttdec}
\end{figure*}

\begin{figure*}[ht]
\includegraphics[width=250pt]{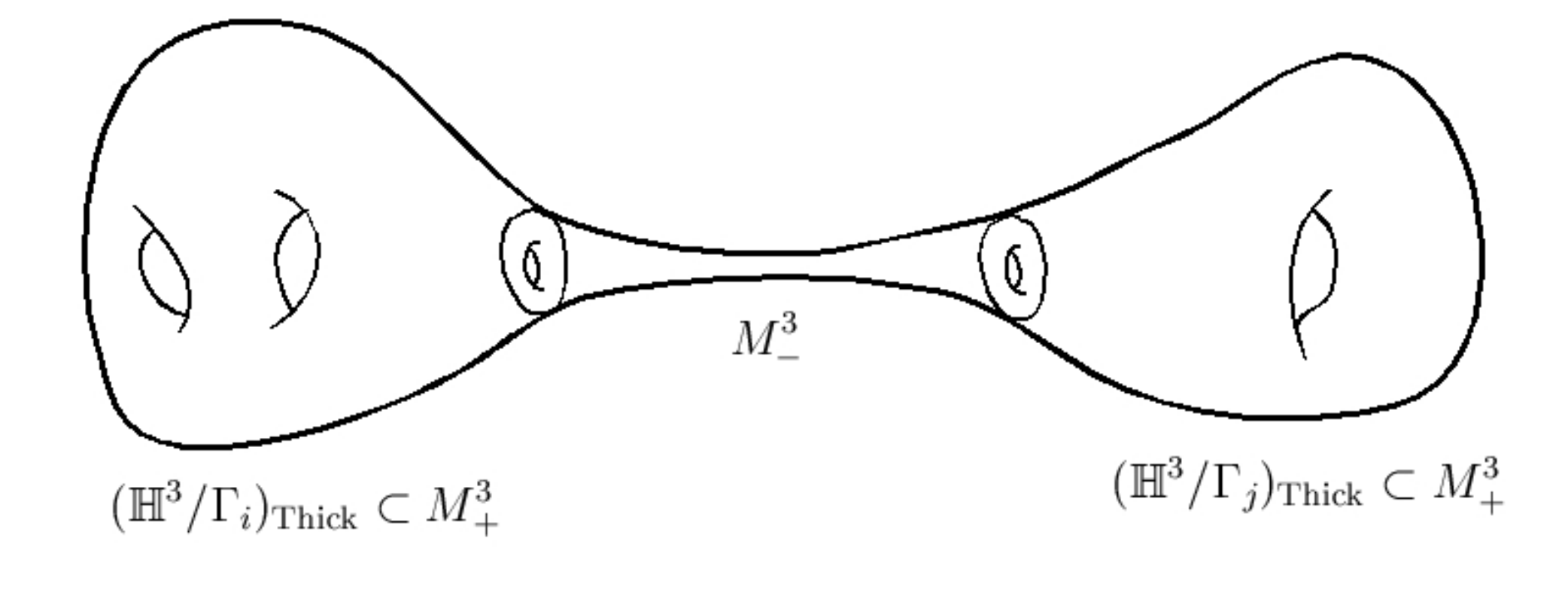}\\
\caption{3-dimensional thick-thin decomposition }\label{3dimttdec}
\end{figure*}

Perelman (cf. \cite{Per2003a}) asserted that the conclusion of
\autoref{thm0.5} holds if we replace the classical Ricci flow by
{\it the Ricci flow with surgeries}, (cf. \cite{Per2003a}).
Detailed proof of this assertion of Perelman can be found in
\cite{CZ2006}, \cite{KL2008} and \cite{MT2010}.

Suppose that $\mathbb H^3/\Gamma$ is a complete but non-compact
hyperbolic $3$-manifold with finite volume. The cuspidal ends of
$\mathbb H^3/\Gamma$ are exactly the thin parts of $M^3_\infty
=\mathbb H^3/\Gamma$. Each cuspidal end of $\mathbb H^3/\Gamma$ is
diffeomorphic to a product of a torus and half-line (i.e., $T^2 \times [0,
\infty)$). Hence, each cusp is a graph-manifold.

It should be pointed out that possibly infinitely many surgeries
took place {\it only} on thick parts of manifolds $\{M_t\}$ after
appropriate re-scalings, due to the celebrated Perelman's
$\kappa$-non-collapsing theory, (see \cite{Per2002}).

Moreover, Perelman (cf. \cite{Per2003a}) pointed out that the study
of the thin part $M_-(\omega, t)$ has nothing to do with Ricci flow,
but is related to {\it his version} of critical point theory for
distance functions. We now outline our simple proof of
\autoref{thm0.1} using Perelman's version of critical point theory
in next sub-section.

\subsection{Outline of a proof of Perelman's
collapsing theorem.} \
\medskip

In order to illustrate main strategy in the proof of Perelman's
collapsing theorem for $3$-manifolds, we make some general remarks.
Roughly speaking, Perelman's collapsing theorem can be viewed as a
generalization of the implicit function theorem. Suppose that
$\{M^3_\alpha\}$ is a sequence of collapsing $3$-manifolds with
curvature $\ge -1$ and that $\{M^3_\alpha\}$ is not a
diameter-collapsing sequence. To verify that $M^3_\alpha$ is a
graph-manifold for sufficiently large $\alpha$, it is sufficient to
construct a decomposition $M^3_\alpha = \cup^{m_\alpha}_{i=1}
U_{\alpha, i }$ and a collection of {\it regular} functions (or
maps) $F_{\alpha, i}: U_{\alpha, i } \to \mathbb R^{s_i}$, where
$s_i =1$ or $2$. We require that the collection of locally defined
functions (or maps) $\{( U_{\alpha, i }, F_{\alpha,
i})\}^{m_\alpha}_{i=1}$ satisfy two conditions:

\begin{enumerate}[{\rm (i)}]
\item Each function (or map) $F_{\alpha, i}$ is
{\it regular} enough so that Perelman's version of implicit function
theorem (cf. \autoref{thm1.2} below) is applicable;

\item The collection of locally defined {\it regular} functions
(or maps) are compatible on any possible overlaps in the sense of
Cheeger-Gromov (cf. \cite{CG1986} \cite{CG1990}). More precisely, if
$U_{\alpha, i } \cap U_{\alpha, j} \neq \varnothing$ and if $
[F_{\alpha, i}^{-1}(y) \cap F_{\alpha, j}^{-1}(z)] \neq \varnothing
$ with $\dim [ F_{\alpha, i}^{-1}(y) ] \le \dim [ F_{\alpha,
j}^{-1}(z)] $, then we require that either $F_{\alpha, i}^{-1}(y)
\subset F_{\alpha, j}^{-1}(z)$ or the union $[F_{\alpha, i}^{-1}(y)
\cup F_{\alpha, j}^{-1}(z)]$ is contained in a $2$-dimensional orbit
of an almost-free torus action.

\end{enumerate}
If the above two conditions are met, with additional efforts, we can
construct a {\it compatible} family of the locally defined Seifert
fibration structures (which is equivalent to an $F$-structure
$\mathcal F$ of positive rank in the sense of Cheeger-Gromov) on a
sufficiently collapsed $3$-manifold $M^3_\alpha$. It follows that
$M^3_\alpha$ is a graph-manifold for sufficiently large $\alpha$,
(cf. \cite{R1993}).

Perelman's choices of locally defined {\it regular} functions (or
maps) are related to distance functions $r_{A_{\alpha, i}}(x)
=\dis(x, A_{\alpha, i})$ from appropriate subsets $A_{\alpha,i}$. We
briefly illustrate the main strategy of our proof for the following
two cases.

\smallskip
\noindent
\textbf{Case 1.}{\it The metric balls $\{
B_{M^3_\alpha}(x_\alpha, r) \}$ collapse to an open interval}.

We will show that $B_{M^3_\alpha}(x_\alpha, r)$ is homeomorphic to a
slim cylinder $N^2_{\alpha}\times I$ with shrinking spherical or
toral factor $N^2_{\alpha}$ for sufficiently large $\alpha$. When a
sequence of the pointed $3$-manifolds $\{(B_{M^3_\alpha}(x_\alpha,
r), x_\alpha) \}$ with curvature $\ge -1$ are convergent to an
$1$-dimensional space $(X^1, x_\infty)$ and $x_\infty$ is an
interior point, Perelman-Yamaguchi fibration theory is applicable.
Thus, we are led to consider the fibration
$$
N^2_\alpha \to B_{M^3_\alpha}(x_\alpha, r)
\stackrel{F_{\alpha}}{\longrightarrow} (-\ell, \ell)
$$

\begin{figure*}[ht]
\includegraphics[width=250pt]{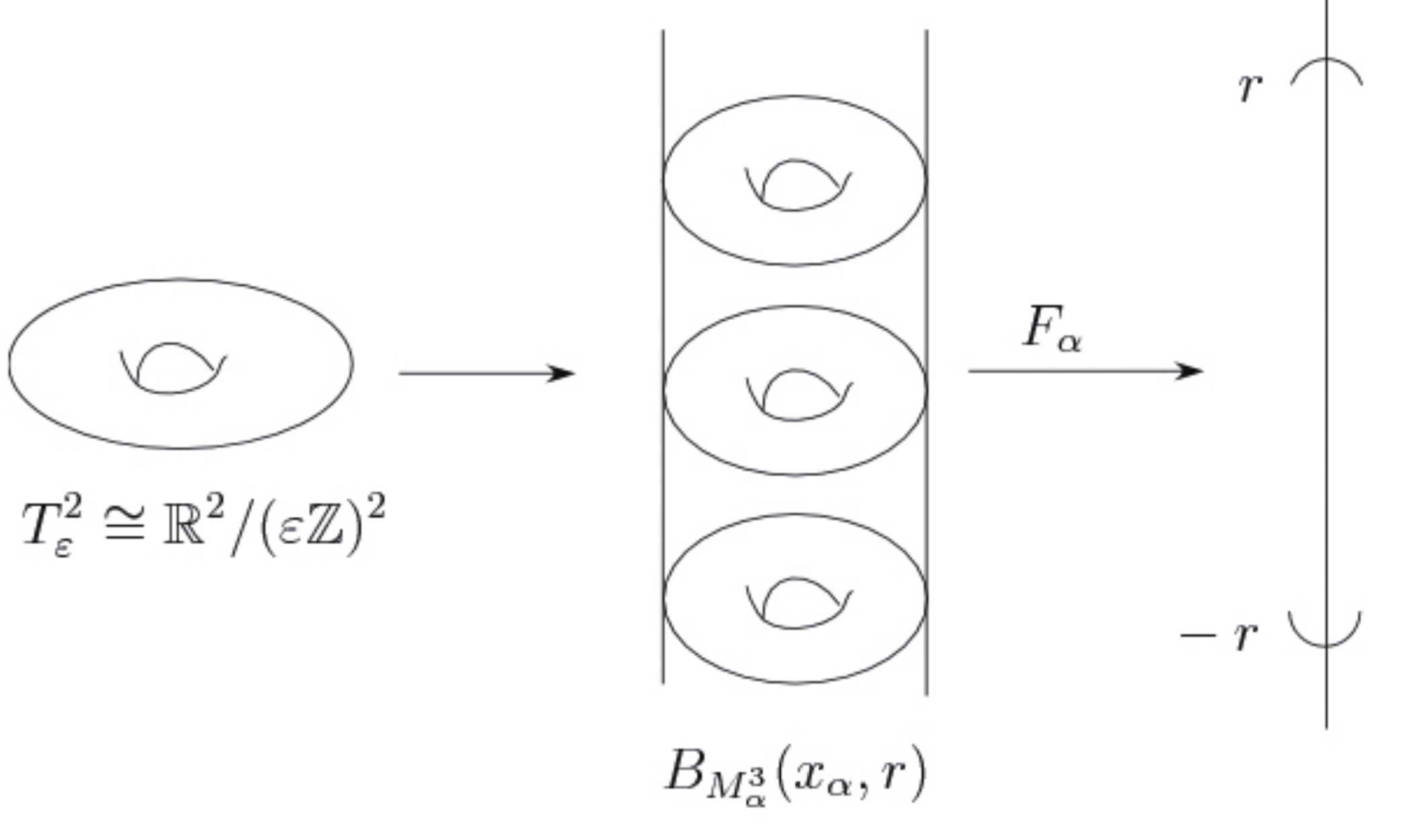}\\
\caption{Slim cylinders with collapsing fibers.}\label{fig:4.1}
\end{figure*}

We now discuss the topological type of the fiber $N^2_\alpha$. Let
$x_\infty \in (-\ell, \ell) $ and $\varepsilon_\alpha$ be the
diameter of $F_{\alpha}^{-1}(x_\infty)$ in $M^3_\alpha$. We further
consider the limiting space $Y^s_\infty$ of re-scaled spaces
$\{(B_{\varepsilon^{-1}_\alpha
M^3_\alpha}(x_\alpha,\frac{r}{\varepsilon_\alpha}), x_\alpha) \},$
as $ \varepsilon_\alpha \to 0$. There are two sub-cases: $\dim
(Y^s_\infty) = 3$ or $\dim(Y^s_\infty) = 2$. Let us consider the
subcase of $\dim (Y^s_\infty) = 3$:
$$
N^2_\infty \to Y^3_\infty \to \mathbb R
$$
where both $N^2_\infty$ and $Y^3_\infty$ are manifolds with possibly
singular metrics of non-negative curvature. To classify singular
surfaces $N^2_\infty$ with non-negative curvature, we use a
splitting theorem and the distance non-increasing property of
Perelman-Sharafutdinov retraction on the universal cover $\tilde
N^2_\infty$, when $N^2_\infty$ has non-zero genus. With some extra
efforts, we will conclude that $N^2_\infty$ must be homeomorphic to
a quotient of 2-sphere or 2-torus, (see \autoref{section2} below).
It now follows from a version of Perelman's stability theorem that
the fiber $N^2_\alpha$ is homeomorphic to $N^2_\infty$, for
sufficiently large $\alpha$. Hence, $N^2_\alpha$ is a quotient of
2-sphere or 2-torus as well, for sufficiently large $\alpha$. Our
new proof of Perelman's collapsing theorem for this subcase is much
simpler than the approach of Shioya-Yamaguchi presented in
\cite{SY2000}.

The sub-case of $\dim(Y^2_\infty) = 2$ is related to the following case:

\smallskip
\noindent \textbf{Case 2}. {\it The metric balls $\{
B_{M^3_\alpha}(x_\alpha, r)\} $ collapse to an open disk}.

We will show that $B_{M^3_\alpha}(x_\alpha, r)$ is homeomorphic to a
fat solid torus $D^2\times S^1_{\varepsilon_\alpha}$ with shrinking
core $S^1_{\varepsilon_\alpha}$ for sufficiently large $\alpha$.

\begin{figure*}[ht]
\includegraphics[width=300pt]{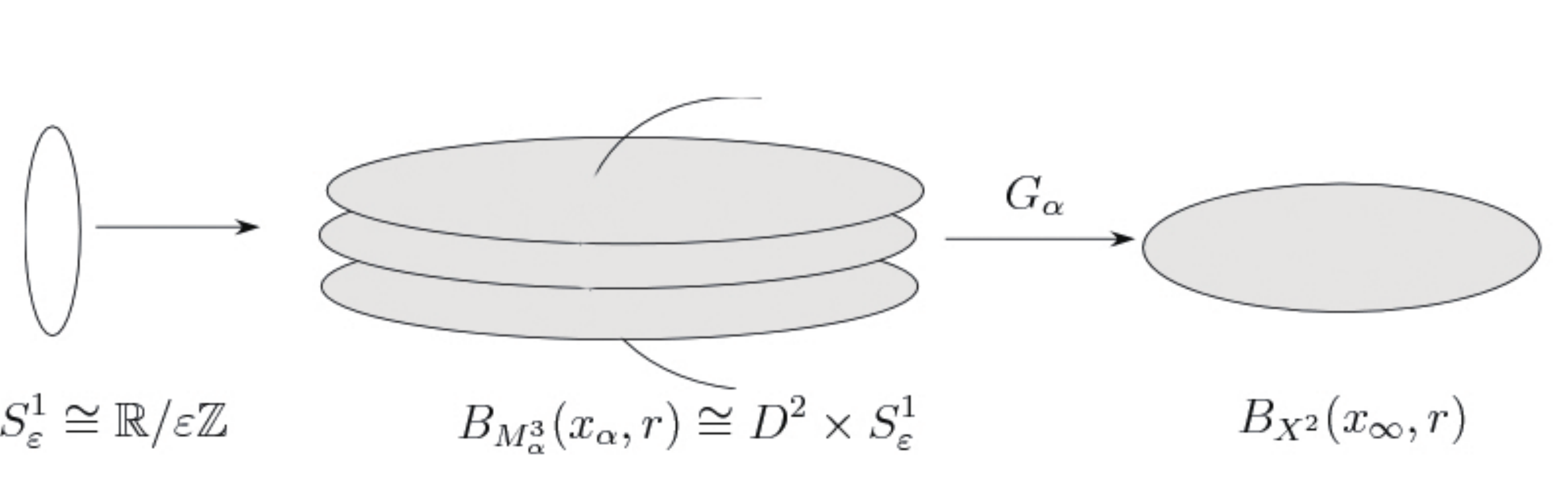}\\
\caption{Fat solid tori with shrinking cores $S^1_\varepsilon$.}\label{fig:1.1}
\end{figure*}

In this case, our strategy can be illustrated in the following
diagram
\begin{diagram}
B_{M^3_{\alpha}}(x_{\alpha},\delta)&\rTo^{F_{\alpha}}&\mathbb{R}^2\\
\dTo_{\text{G-H}}&&\dCorresponds\\
B_{X^2}(x_{\infty},\delta)&\rTo^{F_{\infty}}&\mathbb{R}^2
\end{diagram}
where the sequence of metric balls ${B_{(M^3_{\alpha},
\hat{g}^{\alpha})}(x_{\alpha}, \delta)}$ are convergent to the
metric disk $B_{X^2}(x_{\infty}, \delta)$ for $\delta \le r$. We
will construct an admissible map $F_{\infty}$ which is regular at
the punctured disk, using Perelman's multiple conic singularity
theory, (see \autoref{thm1.17} below). Among other things, we will
use the following result of Perelman to construct the desired map
$F_\infty$.

\begin{theorem}[Conic Lemma in Perelman's critical point theory,
(cf. \cite{Per1994} page 211)]\label{thm0.6} Let $X$ be an
Alexandrov space of dimension $k$, $\curv \ge -1$ and $x\in X$ be an
interior point of $X$. Then the distance function $r_x(y) = \dis(x,
y)$ has {\it no} critical points on $[B_X(x,{\delta}) -\{ x\}]$ for
sufficiently small $\delta$ depending on $x$.
\end{theorem}

We will use \autoref{thm0.6} and Perelman's semi-flow orbit
stability theorem (cf. \autoref{prop1.14} below) to conclude that
the lifting maps $F_\alpha$ is regular on the annular region
$A_{M^3_{\alpha}} (x_{\alpha}, \varepsilon, \delta)=
[B_{M^3_{\alpha}}(x_{\alpha},\delta)- B_{M^3_{\alpha}}(x_{\alpha},
\varepsilon)]$. With extra efforts, one can construct a local
Seifert fibration structure:
$$
S^1 \to A_{M^3_{\alpha}} (x_{\alpha}, \varepsilon, \delta)
\stackrel{G_{\alpha}}{\longrightarrow} A_{X^2} (x_{\infty},
\varepsilon, \delta)
$$

In summary, Perelman's collapsing theorem for 3-manifolds can be
viewed an extension of the implicit function theorem. Our proof of
Perelman's collapsing theorem benefited from {\it his version} of
critical point theory for distance functions, including his conic
singularity theory and fibration theory. Perelman's multiple conic
singularity theory and his fibration theory are designed for
possibly singular Alexandrov spaces $X^k$. Therefore, the smoothness
of metrics on $X^k$ does {\it not} play a big role in the
applications of Perelman's critical point theory, unless we run into
the so-called essential singularities (or extremal subsets). When
essential singularities do occur on surfaces, we use the MCS theory
(e.g. \autoref{thm0.6}) and the multiple-step Perelman-Sharafutdinov
flows to handle them, (see \autoref{section5.2} below).

Without using Perelman's version of critical point theory,
Shioya-Yamaguchi's proof of the collapsing theorem for $3$-manifolds
was lengthy and involved. For instance, they use their singular
version of Gauss-Bonnet theorem to classify surfaces of non-negative
curvature, (see Chapter 14 of \cite{SY2000}). The proof of the singular
version of the Gauss-Bonnet theorem was non-trivial. In
addition, Shioya-Yamaguchi extended the Cheeger-Gromoll soul theory
to 3-dimensional singular spaces with non-negative curvature, which
was rather technical and occupied the half of their first paper
\cite{SY2000}. Using Perelman's version of critical point theory, we
will provide alternative approaches to classify non-negatively
curved surfaces and open 3-manifolds with possibly singular metrics,
(e.g., the 3-dimensional soul theory). Our arguments inspired by
Perelman are considerably shorter than Shioya-Yamaguchi's proof for
the $3$-dimensional soul theory, (see \autoref{section2.2} below).

For the readers who prefer a traditional proof of the collapsing
theorem without using Perelman's version of critical point theory,
we recommend the important papers of Morgan-Tian \cite{MT2008} and
Shioya-Yamaguchi \cite{SY2000}, \cite{SY2005}. Finally, we should
also mention the recent related work of Gerard Besson et al, (cf.
\cite{BBBMP2010}). Another proof of Perelman's collapsing theorem
for 3-manifolds has been announced by Kleiner and Lott (cf.
\cite{KL2010}).

We refer the organization of this paper to the table of contents at
the beginning. In \autoref{section1}-\ref{section4} below, we mostly
discuss interior points of Alexandrov spaces, unless otherwise
specified.

\section{Brief introduction to Perelman's MCS theory and applications to local Seifert fibration}
\label{section1}

In \S 1-2, we will discuss our proof of Theorem 0.1' for a special
case. In this special case, we assume that the sequence of metric
balls $\{ B_{(M^3_{\alpha},g^{\alpha})}(x_{\alpha},r) \} $ is
convergent to a metric ball $B_{X^2}(x_{\infty},r)$, where
$x_\infty$ is an interior point of $X^2$. Using several known
results of Perelman, we will show that there is a (possibly
singular) circle fibration:
\begin{equation}\label{eq1.01}
S^1\to B_{(M^3_{\alpha},g^{\alpha})}(x_{\alpha},\varepsilon)\to
B_{X^2}(x_{\infty},\varepsilon),
\end{equation}
for some $\varepsilon < r$. In other words, we shall show that
$\hat{B}_{M^3_{\alpha}} (x_{\alpha},\varepsilon)$ looks like a {\it
fat} solid tori with a shrinking core, i.e., $
\hat{B}_{M^3_{\alpha}} (x_{\alpha},\varepsilon) \sim [D^2 \times
S^1_{\varepsilon}] \sim [(D^2 \times (\mathbb R/ \varepsilon \mathbb
Z)]$, (see \autoref{fig:1.1} above and \autoref{ex2.0} below).

In fact, using the Conic Lemma (\autoref{thm0.6} above), Kapovitch
\cite{Kap2005} already established a circle-fibration structure over
the annular region $A_{X^2}(x_{\infty},\delta,\varepsilon)$. Let
$\Sigma_x(X)$ denote the space of unit directions of an Alexandrov
space $X$ of curvature $\ge -1$ at point $x$. When $\dim(X)=2$, it
is known (cf. \cite{BGP1992}) that $X^2$ must be a $2$-dimensional
topological manifold. Thus, $\Sigma_x(X^2)$ is a circle, and hence
an $1$-dimensional manifold.

\begin{prop}[\cite{Kap2005}, page 533]\label{prop1.1}
Suppose that $M^n_{\alpha}\xrightarrow{G-H} X$, where $M^n_{\alpha}$
is a sequence of $n$-dimensional Riemannian manifolds with sectional
curvature $\ge -1$. Suppose that there exists $x_{\infty}\in X$ such
that $\Sigma=\Sigma_{x_{\infty}}(X)$ is a closed Riemannian
manifold. Then there exists $r_0=r_0(x_{\infty})$ such that for any
$M_{\alpha}\ni x_{\alpha}\to x_{\infty}$ we have: For any
sufficiently large $\alpha$, and $r \le r_0$, there exists a
topological fiber bundle
$$
S_{\alpha}\to \partial B_{M_{\alpha}} (x_{\alpha},r)\to
\Sigma_{x_{\infty}}(X)
$$
such that
\begin{enumerate}[{\rm (1)}]
\item $S_{\alpha}$ and $\partial B_{M_{\alpha}}(x_{\alpha},r)$ are
topological manifolds;
\item Both $S_{\alpha}$ and $\partial B_{M_{\alpha}}(x_{\alpha},r)$
are connected;
\item The fundamental group $\pi_1(S_{\alpha})$ of the fiber is
almost nilpotent.
\end{enumerate}
\end{prop}

We will use Perelman's fibration theorem and an multiple conic
singularity theory to establish the desired circle fibration over
the annular region $A_{X^2}(x_{\infty},\delta, \varepsilon)$ for
$\delta<\varepsilon$. Our strategy can be illustrated in the
following diagram
\begin{diagram}
B_{M^3_{\alpha}}(x_{\alpha},r)&\rTo^{F_{\alpha}}&\mathbb{R}^2\\
\dTo_{\text{G-H}}&&\dCorresponds\\
B_{X^2}(x_{\infty},r)&\rTo^{F_{\infty}}&\mathbb{R}^2
\end{diagram}
where the sequence of metric balls
${B_{(M^3_{\alpha},\hat{g}^{\alpha})}(x_{\alpha},r)}$ are convergent
to the metric disk $B_{X^2}(x_{\infty},r)$.

If $F_{\alpha}$ were a ``\emph{topological submersion\/}" to its
image, then we would be able to obtain the desired topological
fibration. For this purpose, we will recall Perelman's Fibration
Theorem for non-smooth maps.

\subsection{Brief introduction to Perelman's critical point theory}\label{section1.1}\

\smallskip

We postpone the definition of admissible maps to
\autoref{section1.2}. In \autoref{section1.2}, we will also recall
the notion of {\it regular points} for a sufficiently wide class of
``\emph{admissible mappings\/}" from an Alexandrov space $X^n$ to
Euclidean space $\mathbb R^k$.

Let $F: X^n \to \mathbb R^k$ be an admissible map. The points of an
Alexandrov space $X$ that are not regular are said to be critical
points, and their images in $\mathbb R^k$ are said to be critical
values of $F: X \to \mathbb R^k$. All other points of $\mathbb R^k$
are called regular values.

\begin{theorem}[Perelman's Fibration Theorem \cite{Per1994} page
207]\label{thm1.2}\

\begin{enumerate}[{\rm (A)}]
\item An admissible mapping is open and admits a trivialization
in a neighborhood of each of its regular points.
\item If an admissible mapping has no critical points and is
proper in some domain, then its restriction to this domain is the
projection of a locally trivial fiber bundle.
\end{enumerate}
\end{theorem}

There are several equivalent definitions of Alexandrov spaces of
$\curv \ge k$. Roughly speaking, a length space $X$ is said to have
curvature $\ge 0$ if and only if, for any geodesic triangle $\Delta$
in $X$, the corresponding triangle $\widetilde{\Delta}$ of the same
side-lengths in $\mathbb R^2$ is thinner than $\Delta$.

More precisely, let $M^2_k$ be a simply connected complete surface
of constant sectional curvature $k$. A triangle in a length space
$X$ consists of three vertices, say $\{a,b,c\}$ and three
length-minimizing geodesic segments $\{\overline{ab}, \overline{bc},
\overline{ac}\}$. Let $|ab|$ be the length of $\overline{ab}$. Given
a real number $k$, a comparison triangle $\widetilde{\Delta}^k_
{\tilde{a},\tilde{b},\tilde{c}}$ is a triangle in $M^2_k$ with the
same side lengths. Its angles are called the comparison angles and
denoted by $\widetilde{\measuredangle}^k_a(b,c)$, etc. A comparison
triangle exists and is unique whenever $k\le 0$ or $k>0$ and
$|ab|+|bc|+|ca|<\frac{2\pi}{\sqrt k}$.
\begin{definition}\label{def1.3}
A length space $X$ is called an Alexandrov space of curvature $\ge
k$ if any $x\in X$ has a neighborhood $U_x$ for any $\{a,b,c,d\}\in
U_{x}$, the following inequality
$$
\widetilde{\measuredangle}^k_a(b,c)+\widetilde{\measuredangle} ^
k_a(c,d)+\widetilde{\measuredangle}^k_a(d,b)\le 2\pi.
$$
\end{definition}

Alexandrov spaces with $\curv \ge \lambda$ have several nice
properties, (cf. \cite{BGP1992}). For instance, the dimension of an
Alexandrov space $X$ is either an integer or infinite. Moreover, for
any $x\in X$, there is a well defined tangent cone $T_x^-(X)$ along
with an ``inner product" on $T_x^-(X)$.

In fact, if $X$ is an Alexandrov space with the metric $d$, then we
denote by $\lambda X$ the space $(X,\lambda d)$. Let
$i_{\lambda}:\lambda X \to X$ be the canonical map. The
Gromov-Hausdorff limit of pointed spaces $\{(\lambda X, x)\}$ for
$\lambda \to \infty$ is the tangent cone $T_x^-(X)$ at $x$, (see
$\S$7.8.1 of \cite{BGP1992}).

For any function $f: X\to \mathbb R$, the function $d_xf:
T_x^-(X)\to \mathbb R$ such that
$$
d_xf=\lim_{\lambda\to+\infty} \frac{f\circ
i_{\lambda}-f(x)}{1/\lambda}
$$
is called the differential of $f$ at $x$.

Let us now recall the notion of regular points for distance
functions.
\begin{definition}[\cite{GS1977}, \cite{Gro1981}]\label{def1.4}
Let $A\subset X$ be a closed subset of an Alexandrov space $X$ and
$f_A(x)=\dis(x,A)$ be the corresponding distance function from $A$.
A point $x\not\in A$ is said to be a regular point of $f_A$ if there
exists a non-zero direction $\vec{\xi}\in T_x^-(X)$ such that
\begin{equation}\label{eq1.1}
d_xf(\vec{\xi})>0.
\end{equation}
\end{definition}

It is well-known that if $X$ has $\curv \ge0$ then $f(x) =
\frac{1}{2} [\dis(x,p)]^2$ has the property that $\text{Hess}(f)\le
I$, (see \cite{Petr2007}). To explain such an inequality, we recall
the notion of semi-concave functions.

\begin{definition}[\cite{Per1994} page 210]\label{def1.5}
A function $f: X\to \mathbb R$ is said to be $\lambda$-concave in an
open domain $U$ if for any length-minimizing geodesic segment
$\sigma: [a,b]\to U$ of unit speed, the function
$$
f\circ \sigma (t) - \lambda t^2 /2
$$
is concave.
\end{definition}

When $f$ is 1-concave, we say that $\text{Hess}(f)\le I$. It is
clear that if $f: U\to \mathbb R$ is a semi-concave function, then
$$d_xf: T_x^-(X)\to \mathbb R$$
is a concave function.

In order to introduce the notion of semi-gradient vector for a
semi-concave function $f$, we need to recall the notion of
``\emph{inner product\/}" on $T_x^-(X)$. For any pair of vectors
$\overrightarrow{u}$ and $\vec{v}$ in $T_x^-(X)$, we define
$$
\langle\vec{u},\vec{v}\rangle=\frac12 (|\vec{u}|^2 + |\vec{v}|^2 -
|\vec{u}\vec{v}|^2) = |\vec{u}||\vec{v}|\cos\theta
$$
where $\theta$ is the angle between $\vec{u}$ and $\vec{v}$,
$|\vec{u}\vec{v}|=\dis_{ T_x^-(X)}(\vec u, \vec v)$, $|\vec
u|=\dis_{ T_x^-(X)}(\vec u, o)$ and $o$ denotes the origin of the
tangent cone.

\begin{definition}[\cite{Petr2007}]\label{def1.6}
For any given semi-concave function $f$ on $X$, a vector
$\vec{\eta}\in T_x^-(X)$ is called a gradient of $f$ at $x$ (in
short $\vec{\eta}=\nabla f$) if
\begin{enumerate}[{\rm (i)}]
\item
$d_xf(\vec v) \le \langle\vec{\eta},\vec v\rangle$ for any $\vec v
\in T_x^-(X) $ ;
\item
$d_xf(\vec{\eta})=|\vec{\eta}|^2$.
\end{enumerate}
\end{definition}

It is easy to see that any semi-concave function has a uniquely
defined gradient vector field. Moreover, if $d_xf(\vec v)\le 0$ for
all $\vec v\in T_x^-(X)$, then $\nabla f|_x=0$. In this case, $x$ is
called a critical point of $f$. Otherwise, we set
$$
\nabla f= d_x f(\vec{\xi})\vec{\xi}
$$
where $\vec{\xi}$ is the (necessarily unique) unit vector for which
$d_xf$ attains its positive maximum on $\Sigma_x(X)$, where
$\Sigma_x(X)$ is the space of direction of $X$ at $x$.

\begin{prop}[\cite{Per1994}, \cite{PP1994}]\label{prop1.7}
Let $X^n$ be a metric space with curvature $\ge -1$ and $\hat x $ be
an interior point of $X^n$. Then there exists a strictly concave
function $h: B(\hat x, r) \to (-\infty, 0]$ such that (1) $h(\hat x)
= 0$ and $B(\hat x, \frac{s}{\lambda}) \subset h^{-1}( (-s, 0])
\subset B(\hat x, \lambda s) $ for $s \le \frac{r}{4\lambda}$; (2)
the distance function $r_{\hat x}(y)$ has no critical point in
punctured ball $[B_X(\hat x, \varepsilon) - \{\hat x\}]$, for some
$\{\varepsilon, r, \lambda\}$ depending on $\hat x$.
\end{prop}
\begin{proof}
(1) The construction of the strictly concave function $h$ described
above is available in literature (see \cite{GW1997},
\cite{Kap2002}). In fact, let $f_{\delta'}: X \to \mathbb{R}$ be
defined as on page 129 of \cite{Kap2002} for $\delta' < \delta$. We
choose $h(x) = f_{\delta'}(x) - 1$. Kapovitch showed that the
inequality
$$
1 \le \frac{\dis(x, \hat x) }{t} \le \frac{1}{\cos (3\delta)}
$$
holds for $x \in h^{-1}(-t)$ and $t < < \delta' < \delta$, (see page
132 of \cite{Kap2002}). Thus, there exists a $\lambda$ such that
$B(\hat x, \frac{s}{\lambda}) \subset h^{-1}( (-s, 0])
\subset B(\hat x, \lambda s) $ for $s \le \frac{r}{4\lambda}$.

(2) For the convenience of readers, we add the following alternative
proof of the second assertion. Let us recall that the tangent cone
$(T^-_{\hat x}(X^n), O)$ is the Gromov-Hausdorff limit of the
pointed re-scaled spaces $\{(\lambda X, \hat x)\}_{\lambda\ge 0}$ as
$\lambda \to +\infty$, i.e.
$$(\lambda X, \hat x)\to (T^-_{\hat x}X, O)$$
as $\lambda \to +\infty$, where $O$ is the apex of the tangent cone
$T^-_{\hat x}X$. Let $\dis_{O, T^-_{\hat x}X}(\eta)=\dis_{T^-_{\hat
x}X}(O, \eta)$ and $\dis_{\hat x, \lambda X}(y)=\dis_{\lambda X}(y,
\hat x)=\lambda \dis_{X}(y, \hat x)$. We consider
$f_{\lambda}(y)=\frac12 (\dis_{\hat x, \lambda X}(y, \hat x))^2$. By
an equivalent definition of curvature $\ge -1$,
$\{f_{\lambda}\}_{\lambda \ge 1}$ and $f_{\infty}$ are semi-concave
functions.

Lemma 1.3.4 of \cite{Petr2007} implies that if $p_\lambda \to
p_\infty$ as $\lambda \to p_\infty$ then $\liminf_{\lambda \to
\infty} |\nabla f_\lambda|(p_\lambda) \ge |\nabla f_{\infty}|(
p_\infty)|$.

Let $A_M(x, r, R) = [\overline{B_M(x,R)}-B_M(x,r)]$ be an annular
region. Our energy function $f_{\infty}(\eta)=\frac12 |\eta|^2$ has
property $|\nabla f_{\infty}| \ge \frac 18 $ on $A_{T_{\hat
x}(X)}(0, \frac12, 1)$. It follows from Lemma 1.3.4 of
\cite{Petr2007} that, for sufficiently large $\lambda\ge \lambda_0
>1$, the function $f_{\lambda}$ has no critical point on the annual
region
$$
A_{\lambda X}(\hat x, \frac12, 1) =A_{X}(\hat x,
\frac{1}{2\lambda},\frac {1}{\lambda}).
$$
 Since we have
$$
[B_{X}(\hat x, \frac {1}{\lambda_0})-\{\hat x\}] = \cup_{\lambda\ge
\lambda_0}A_{X} (\hat x, \frac{1}{2\lambda},\frac {1}{\lambda}),
$$
we conclude that the radial distance function has no critical point
on the punctured ball $[B_X(\hat x, \varepsilon) - \{\hat x\}]$.
\end{proof}

\subsection{Regular values of admissible maps}\label{section1.2}\

\smallskip

In this subsection, we recall explicit definitions of admissible
mappings and their regular points introduced by Perelman.

\begin{definition}[\cite{Per1994},\cite{Per1997} admissible maps]\label{def1.9}\ \,
{\rm (1)} Let $X^n$ be a complete Alexandrov space of dimension $n$
and $\curv_{X^n} \ge c$ and $U \subset X^n$. A function $f: U \to
\mathbb R$ is called admissible if $f(x)=\sum_{i=1}^m \phi_i
(\dis_{A_i}(x))$ where $A_i\subset M$ is a closed subset and
$\phi_i: \mathbb R \to \mathbb R$ is continuous.

{\rm (2)} A map $\hat F: X^n \to \mathbb R^k$ is said to be
admissible in a domain $U\subset M$ if it can be represented as
$\hat F=G\circ F$, where $G:\mathbb R^k\to \mathbb R^k$ is
bi-Lipschitz homeomorphism and each component $f_i$ of
$F=(f_1,f_2,\dots,f_k)$ is admissible.

\end{definition}

The definition of regular points for admissible maps $\hat F: X^n
\to \mathbb R^k$ on general Alexandrov spaces is rather technical.
For the purpose of this paper, we only need to consider two lower
dimensional cases of $X^n$: either $X^3$ is a smooth Riemannian
$3$-manifold or $X^2$ is a surface with curvature $\ge c$.

\begin{definition}[Regular points of admissible maps for $\dim \le
3$]\label{def1.10} \

{\rm (1)} Suppose that $\hat F: M^3 \to \mathbb R^2$ is an
admissible map from a smooth Riemannian $3$-manifold $M^3$ to
$\mathbb R^2$ on a domain $U\subset M$ and $\hat F=G\circ F$, where
$G:\mathbb R^2\to \mathbb R^2$ is bi-Lipschitz homeomorphism and
each component $f_i$ of $F=(f_1,f_2)$ is admissible. If $\{ \nabla
f_1, \nabla f_2 \}$ are linearly independent at $x \in U$, then $x$
is said to be a regular point of $\hat F$.

{\rm (2)} (\cite{Per1994} page 210). Suppose that $\hat F: X^2 \to
\mathbb R^2$ is an admissible map from an Alexandrov surface $X^2$
of curvature $\curv \ge c$ to $\mathbb R^2$ on a domain $U\subset
X^2$ and $\hat F=G\circ F$, where $G:\mathbb R^2\to \mathbb R^2$ is
bi-Lipschitz homeomorphism and each component $f_i$ of $F=(f_1,f_2)$
is admissible. Suppose that $f_1$ and $f_2$ satisfy the following
conditions:
\begin{enumerate}[{\rm (2.a)}]
\item $\langle\nabla f_1,\nabla f_2\rangle _q<-\varepsilon<0$;
\item There exits $\vec{\xi}\in T^-_q(X^2)$ such that
$\min\{d_qf_1(\vec{\xi}),d_qf_2(\vec{\xi})\}>\varepsilon>0$.
\end{enumerate}

Then $q$ is called a regular point of $\hat F |_U$.

\end{definition}

\begin{remark}
It is clear that Perelman's condition (2.a) implies that
$\diam(\Sigma_q(X^2))>\frac{\pi}{2}$. This together with (2.b)
implies that
$$
\frac{\pi}{2} < \measuredangle_q(\nabla f_1,\nabla f_2)<\pi.
$$
Conversely, we would like to point out that if
$\diam(\Sigma_q(X^2))>\frac{\pi}{2}$, then there exists an
admissible map $F=(f_1,f_2): U_q \to\mathbb R^2$ satisfying
Perelman's condition (2.a) and (2.b) mentioned above, where $U_q$ is
a small neighborhood of $p$ in $X^2$. \qed
\end{remark}

We need to single out ``\emph{bad points\/}" (i.e., essential
singularities) for which the condition
$$
\diam(\Sigma_q(X^2))>\frac{\pi}{2}
$$
fails. These bad points are related to the so-called extremal
subsets (or essential singularities) of and Alexandrov space with
curvature $\ge c$.

\begin{definition}[Extremal points of an Alexandrov surface]\label{def1.12}
Let $X^2$ be an Alexandrov surface and $z $ be an interior point of
$X^2$. If the diameter of space of unit tangent directions
$\Sigma_z(X^2)$ has diameter less than or equal to $\frac{\pi}{2}$,
i.e.
\begin{equation}\label{eq1.10}
\diam(\Sigma_z(X^2))\le\frac{\pi}{2},
\end{equation}
then $z$ is called an extremal point of the Alexandrov surface
$X^2$. If $\diam(\Sigma_z(X^2))>\frac{\pi}{2}$, then we say that $z$
is a regular point of $X^2$.
\end{definition}

A direct consequence of \autoref{thm0.6} (i.e., \autoref{prop1.7})
is the regularity of sufficiently small punctured disk in an
Alexandrov surface.

\begin{cor}\label{cor1.12}
Let $X^2$ be an Alexandrov space of curvature $\ge -1$ and $\delta$
be as in \autoref{thm0.6}. Then each point $y\in [B_{X^2}(\hat
x,\delta)-\{\hat x\}]$ in punctured disk is regular.
\end{cor}

We recall the Perelman-Sharafutdinov gradient semi-flows for
semi-concave functions.

\begin{definition}\label{def1.13}
A curve $\phi:[a,b]\to X$ is called an $f$-gradient curve if for any
$t\in [a,b]$
\begin{equation}\label{eq1.10(2)}
\frac{d^+\phi}{dt}=\nabla f|_{\phi(t)}.
\end{equation}
\end{definition}

It is known that if $f:X\to \mathbb R$ is a semi-concave function
then there exists a unique $f$-gradient curve $\phi:[a,+\infty) \to
X$ with a given initial point $\phi(0)=p$, (cf. Prop 2.3.3 of
\cite{KPT2009}). We will frequently use the following result of
Perelman (cf. \cite{Per1994}) and Perelman-Petrunin (cf.
\cite{PP1994}).

\begin{prop}[Lemma 2.4.2 of \cite{KPT2009}, Lemma 1.3.4 of \cite{Petr2007}]\label{prop1.14}
Let $X_{\alpha}\to X_{\infty}$ be a sequence of Alexandrov space of
curvature $\ge -1$ which converges to an Alexandrov space
$X_{\infty}$. Suppose that $f_{\alpha}\to f_{\infty}$ where
$f_{\alpha}: X_{\alpha}\to \mathbb R$ is a sequence of
$\lambda$-concave functions and $f_{\infty}: X_{\infty}\to \mathbb
R$. Assume that $\psi_\alpha:[0,+\infty)\to X_{\alpha}$ is a
sequence of $f_{\alpha}$-gradient curves with
$\psi_{\alpha}(0)=p_{\alpha}\to p_{\infty}$ and
$\psi_{\infty}:[0,+\infty)\to X_{\infty}$ be the $f$-gradient curve
with $\psi_{\infty}(0)=p_{\infty}$. Then the following is true

{\rm (1)} for each $t \ge 0$, we have
$$\psi_{\alpha}(t)\to \psi_{\infty}(t)$$
as $\alpha \to \infty$;

{\rm (2)} $\liminf_{\alpha \to \infty} |\nabla f_\alpha|(p_\alpha)
\ge |\nabla f_{\infty}|( p_\infty)$. Consequently, if $\{q_\alpha\}$
is a bounded sequence of critical points of $f_\alpha$, then
$\{q_\alpha\} $ has a subsequence converging to a critical point
$q_\infty$ of $f_\infty$.
\end{prop}

As we pointed out earlier, the pointed spaces $\{ (\frac{1}
{\varepsilon} X, x)\}$ converge to the tangent cone of $X$ at $x$,
i.e., $(\frac{1}{\varepsilon}X, x)\to(T_x^-(X), 0)$ as
$\varepsilon\to 0$, where $0$ is the origin of tangent cone.

When $X^2$ is an Alexandrov surface of curvature $\ge -1$, it is
known that $X^2$ is a $2$-dimensional manifold. Moreover we have the
following observation.

\begin{prop}[\cite{Per1994}]\label{prop1.15}
Let $X$ be an Alexandrov space of curvature $\ge -1$. Suppose that
$\hat x$ is an interior point of $X$. Then $B_{X}(\hat x,
\varepsilon)$ is homeomorphic to $B_{T^-_{\hat x}X}(0,\varepsilon)$,
where $\varepsilon$ is given by \autoref{prop1.7}. Furthermore,
there exits an admissible map
$$G: T^-_{\hat x}(X^2)\to \mathbb R^2$$
such that $G$ is bi-Lipschitz homeomorphism and $G$ is regular at
$\vec v \ne 0$.
\end{prop}
\begin{proof}
This is an established result of Perelman, (cf. \cite{Per1994},
\cite{Kap2007}). We provide a short proof here only for the
convenience of readers. Let us first prove that $B_{X}(\hat x,
\varepsilon)$ is homeomorphic to $B_{T^-_{\hat x}X}(0,\varepsilon)$,
where $\varepsilon$ is given by \autoref{prop1.7}. Recall that
$\{(\frac{1}{\delta} X, \hat x)\}$ is convergent to $ (T^-_{\hat
x}X, O)$, as $\delta \to 0$. By Perelman's stability theorem (cf.
Theorem 7.11 of \cite{Kap2007}), $B_{\frac{1}{\delta}X}(\hat x, 1)$
is homeomorphic to $B_{T^-_{\hat x}X}(0, 1)$ for sufficiently small
$\delta$. Thus, $B_{X}(\hat x, \delta)$ is homeomorphic to
$B_{T^-_{\hat x}X}(0,\delta)$.

By \autoref{prop1.7}, the function $r_{\hat x}(y) = \dis_X(y, \hat
x)$ has no critical point in punctured ball $[B_{X}(\hat x,
\varepsilon) -\{ x\}]$. Thus, we can apply Perelman's fibration
theorem to the following diagram:
$$
\Sigma_{\hat x}(X) \to A_{X^2}(\hat x, \delta/2, \varepsilon )
\stackrel{r_{\hat x}}{\longrightarrow} (\delta/2, \varepsilon).
$$

Consequently, we see that $A_{X^2}(\hat x, \frac{\delta}{2},
\varepsilon )$ is homeomorphic to a cylinder $C=(\partial
B)\times(\frac{\delta}{2},\varepsilon)$. Furthermore, the metric
sphere $(\partial B)$ is homeomorphic to $\Sigma_{\hat x}(X)$. It
follows that the metric ball $B_{X}(\hat x, \varepsilon) \sim [
B_{T^-_{\hat x}X}(0,\delta) \cup C]$ is homeomorphic to
$B_{T^-_{\hat x}X}(0,\varepsilon)$, where $\varepsilon$ is given by
\autoref{prop1.7}.

It remains to construct the desired map $G$. If
$\theta=\frac{1}{3}\diam(\Sigma_{\hat x}(X^2))$, then $\theta\le
\pi/3$. Let us choose six vectors $\{\vec{\xi}_1,\vec{\xi}_2,
\vec{\xi}_3, \vec{\xi}_4, \vec{\xi}_5, \vec{\xi}_6\}\subset
\Sigma_{\hat x}(X^2)$ such that $\measuredangle (\vec{\xi}_i,
\vec{\xi}_j)=\theta$ for $i=j+1$ or $\{i=1, j=6 \}$. It is easy to
construct affine map $H$ from an Euclidean sector of angle $\theta$
to an Euclidean sector of angle $\pi/3$. In fact we can first
isometrically embed an Euclidean sector as $\{(x,y)| x\ge 0, y\ge 0,
0 < \varepsilon_1 \le \arctan ( \frac{y}{x}) \le \theta +
\varepsilon_1 \} \subset \mathbb R^2 $. We could choose $h_1(x,y)=x,
h_2(x,y)=\lambda y$, (see \autoref{lifted_regular_graph}).

Each affine map $H=(h_1,h_2)$ has height functions (distance
functions from axes) up to scaling factors as its components. We can
arrange the Euclidean sectors in an appropriate order so that $H$ is
admissible. For instance, we could change the role of $h_1$ and
$h_2$ for adjunct sector such that, by gluing six Euclidean sectors
together, we can recover $T_{\hat x}(X^2)$ and construct an
admissible map $G: T_{\hat x}(X^2)\to \mathbb R^2$.
\end{proof}

\begin{figure*}[ht]
\includegraphics[width=200pt]{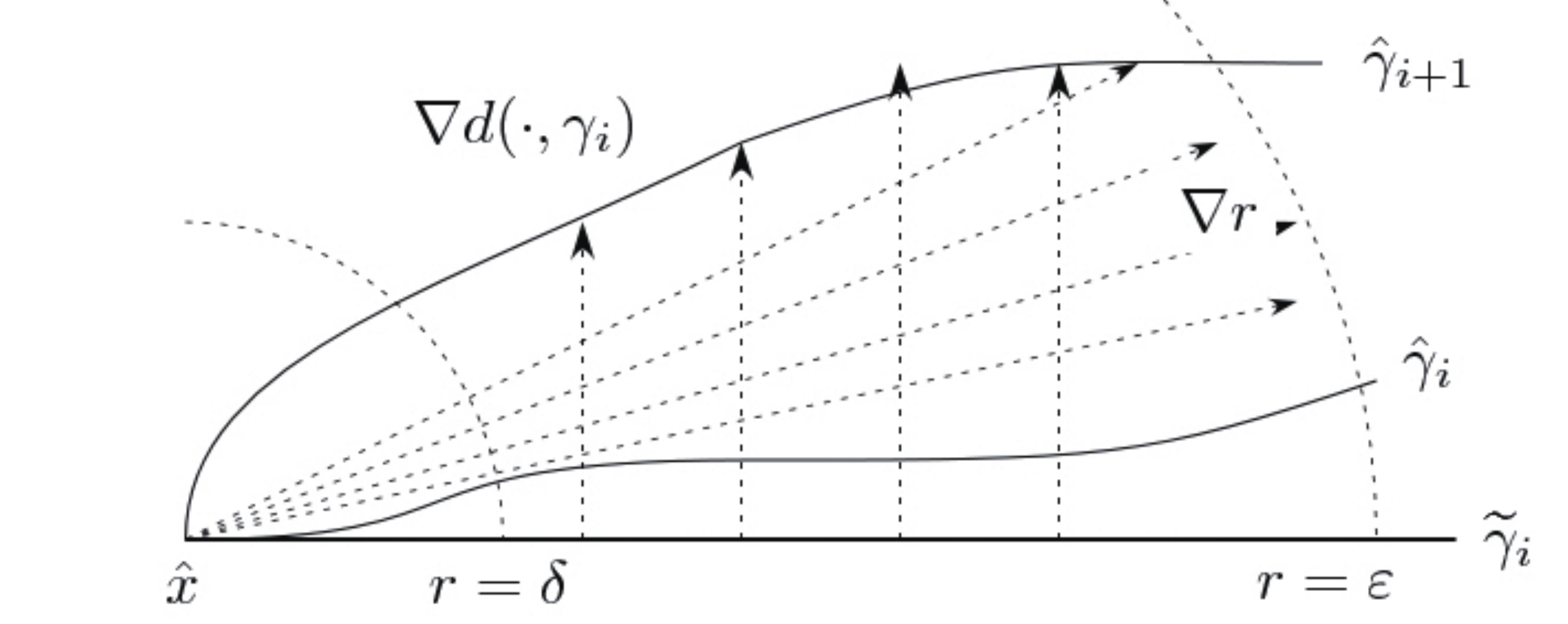}\\
\caption{The construction of a regular map form $S_i\cap
A_{X^2}(x_{\infty},\delta,\varepsilon)$.}\label{lifted_regular_graph}
\end{figure*}

Recall that $\lim_{\lambda \to \infty} (\lambda X, x) = (T_x^-(X),
O)$. By lifting the admissible map $G: T_x(X^2) \to \mathbb{R}^2 $
to $F: \lambda X^2 \to \mathbb{R}^2$, we have the following result.

\begin{cor}\label{cor1.16}
Let $X^2$ be an Alexandrov surface of curvature $\ge -1$ and $\hat
x$ be an interior point. Then there exist sufficiently small
$\varepsilon>\delta>0$ and admissible maps
$$
F_{\infty, i}: X^2 \to \mathbb R^2
$$
such that $F_{\infty, i}$ is regular on $S_i\cap A_{X^2} (\hat
x,\delta, \varepsilon)$ for $i = 1, 2, \cdots, 6$, where $S_i
\subset X^2 $ is the lift of the Euclidean sector bounded by
$\{\vec\xi_i, \vec \xi_{i+1}\}$ from $T^-_{\hat x}X^2$ to $X^2$.
\end{cor}
\begin{proof}
Let $\{h_1, h_2\}$ and $\{\vec{\xi}_1,\vec{\xi}_2, \vec{\xi}_3,
\vec{\xi}_4, \vec{\xi}_5, \vec{\xi}_6\} \subset \Sigma_{\hat
x}(X^2)$ be as in the proof of \autoref{prop1.15}. We choose an
1-parameter family of length-minimizing geodesic segments $\hat
\gamma_{i, \lambda}: [0, 2] \to \lambda X^2$ such that $\{ \hat
\gamma_{i, \lambda} \}$ are converging to $\xi_i $ in $T_{\hat
x}(X^2)$, as $\lambda \to \infty$. We also choose another
length-minimizing geodesic segment $\tilde \gamma_{i, \lambda}: [0,
2\delta] \to \lambda X^2$ outside the geodesic hinge $\{\hat
\gamma_{i, \lambda}, \hat \gamma_{i+1, \lambda}\}$ such that $ 0 <
\varepsilon_1/2 \le \angle_{\hat x} ( \tilde \gamma_{i,
\lambda}'(0), \hat \gamma_{i, \lambda}'(0)) \le \varepsilon_1$. (see
\autoref{lifted_regular_graph} above). We consider lifting distance
functions as follows:
$$r_{ \lambda X^2}(y ) =\dis_{ \lambda X^2}(\hat x, y)$$
and
$$f_{\lambda,i}(y) =\dis_{ \lambda X^2}(\tilde \gamma_{i, \lambda},y)$$
It follows from \autoref{prop1.14} that two functions $\{r_{\lambda
X^2}, f_{\lambda,i}\}$ are regular on $S_i\cap A_{\lambda X^2}(\hat
x, \frac{1}{2}, 1)$ for sufficiently large $\lambda$.

Choosing a sufficiently large $\lambda_0$, we consider the map
$$F_{\infty, i}(\cdot) =(\lambda_0 \dis_{X^2}(\hat x, \cdot),
\lambda_0 \dis_{X^2}(\tilde \gamma_i, \cdot))$$ It follows from
\autoref{prop1.14} that $F_{\infty, i}$ is regular on the curved
trapezoid-like region $S_i\cap A_{X^2}(\hat x,\delta, \varepsilon)$,
(see \autoref{lifted_regular_graph}).
\end{proof}

Recall that there is a sequence $\{M^3_{\alpha}\}$ convergent to
$X_{\infty}^2$. We conclude \autoref{section1} by the following
circle-fibration theorem.

\begin{theorem}\label{thm1.17}
Suppose that a sequence of pointed $3$-manifolds
$\{(M^3_{\alpha},x_{\alpha})\}$ is convergent to a $2$-dimensional
Alexandrov surface $(X_{\infty}, x_{\infty})$ such that $x_{\infty}$
is an interior point of $X^2_{\infty}$. Suppose that
$$ F_{\infty}:
A_{X^2}(x_{\infty},\delta,\varepsilon) \to A_{\mathbb R^2}(0,\delta,
\varepsilon)
$$
is a regular map as above. Then there exist maps
$$F_{\alpha}: A^3_{(M^3_{\alpha},\hat g^{\alpha})}(x_{\alpha},
\delta, \varepsilon) \to A_{\mathbb R^2}(0,\delta, \varepsilon)$$
 such that $F_{\alpha}$ is regular for sufficiently large $\alpha$.
Moreover, the map $F_\alpha$ gives rise to a circle fibration:
$$S^1\hookrightarrow A^3_{(M^3_{\alpha},\hat g^{\alpha})}(x_{\alpha},
\delta, \varepsilon) \to A_{X^2}(x_{\infty},\delta,\varepsilon)$$
for sufficiently large $\alpha$.
\end{theorem}
\begin{proof}
We will use the same notations as in proofs of \autoref{prop1.15}
and \autoref{cor1.16}. For each $F_{\infty, i}$ constructed in
\autoref{cor1.16}, we consider the map
$$
F_{\infty, i}=(\dis_{X^2}(\hat x, \cdot), \dis_{X^2}(\tilde
\gamma_i, \cdot))
$$
up to re-scaling factors, which is regular on $S_i\cap
A_{X^2}(x_{\infty},\delta,\varepsilon)$.  Since
$B_{M^3_\alpha}(x_\alpha, r) \to B_{X^2}(x_{\infty}, r)$ as $\alpha \to \infty$,
 we can
choose geodesic segment $\tilde \gamma_{i,\alpha}$ in $M^3_{\alpha}$
with $\tilde \gamma_{i,\alpha}\to \tilde \gamma_i$. Then we can
define a  map of $F_{i,\infty}$ by
$$F_{i,\alpha}=(\dis_{M^3_{\alpha}}(x_{\alpha},\cdot),
\dis_{M^3_{\alpha}}(\tilde \gamma_{i,\alpha}, \cdot)).$$
We further choose $S_{i,\alpha} \sim F_{i,\alpha}^{-1}(S^*_i)
\subset M^3_{\alpha}$ such that $ S_{i,\alpha}\to S_i$ as $\alpha
\to \infty$, where $S^*_i $ is an Euclidean sector of $\mathbb R^2$
described in the proof of \autoref{prop1.15} and \autoref{cor1.16},
(see \autoref{lifted_regular_graph} and \autoref{fig1_3}).
\begin{figure*}[ht]
\includegraphics[width=180pt]{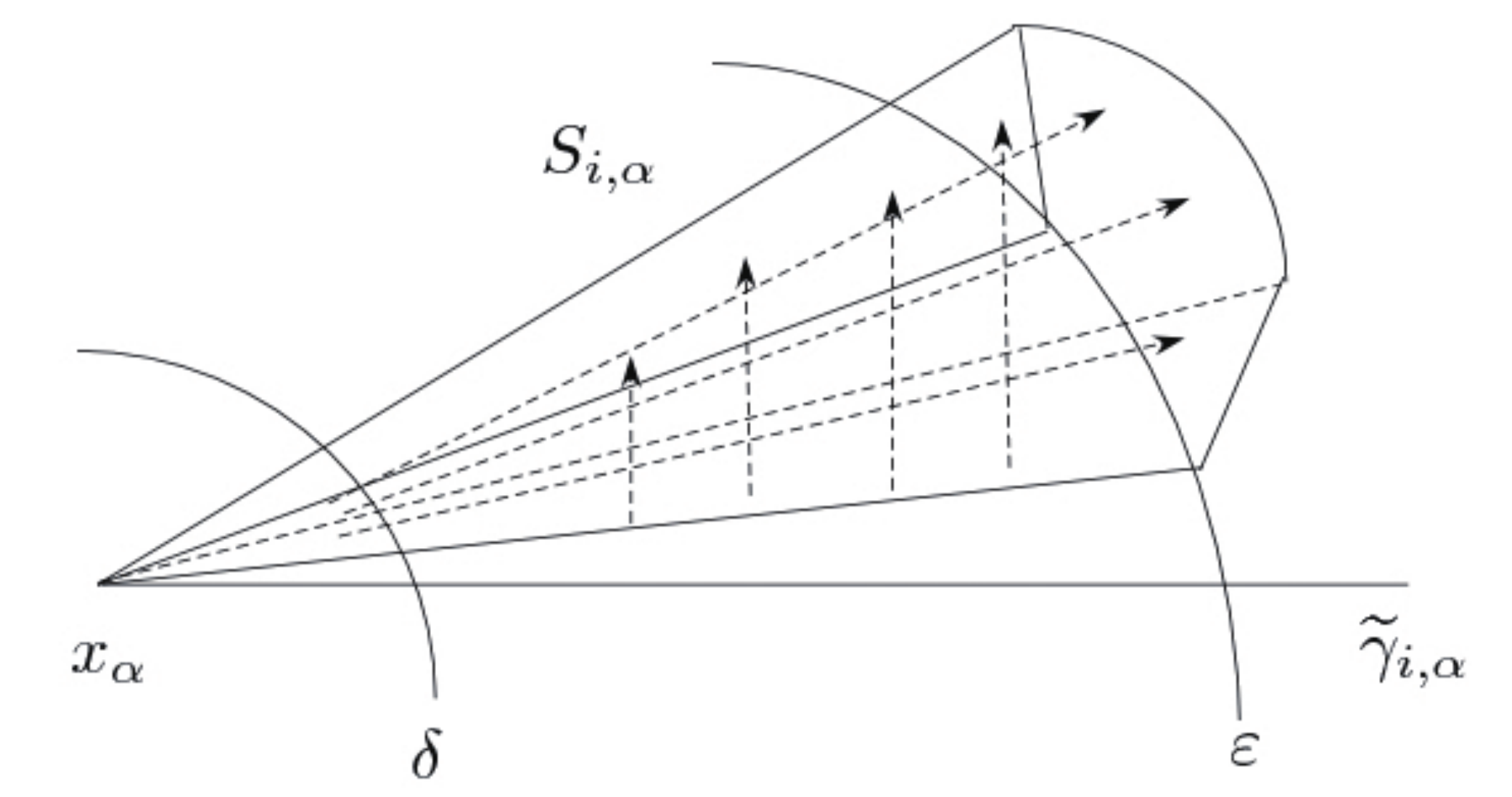}\\
\caption{A regular map from $S_{i,\alpha}\cap
A^3_{(M^3_{\alpha},\hat g^{\alpha})}(x_{\alpha}, \delta,
\varepsilon)$.}\label{fig1_3}
\end{figure*}
It follows from \autoref{prop1.14} that  $F_{i,\alpha}$ is
also regular in $S_{i,\alpha}\cap A^3_{(M^3_{\alpha},\hat
g^{\alpha})}(x_{\alpha}, \delta, \varepsilon)$, for sufficiently
large $\alpha$. Using Perelman's fibration theorem (\autoref{thm1.2}
above), we obtain that $F_{\alpha, i}$ defines a $S^1$ fibration on
$S_{i,\alpha}\cap A^3_{(M^3_{\alpha},\hat g^{\alpha})}(x_{\alpha},
\delta, \varepsilon)$. To get a global $S^1$-fibration on
$A^3_{(M^3_{\alpha},\hat g^{\alpha})}(x_{\alpha}, \delta,
\varepsilon)$, we need to glue these local fibration structures
together. We will discuss the detail of the gluing procedure in
\autoref{section6} below.
\end{proof}

\section{Exceptional orbits, Cheeger-Gromoll-Perelman soul
theory and Perelman-Sharafutdinov flows}\label{section2}
\setcounter{theorem}{-1}

We begin with an example of collapsing manifolds with exceptional
orbits of a circle action on a solid tori.

\begin{example}[\cite{CG72}]\label{ex2.0}
Let $\mathbb R\times D^2=\{(x,y,z)| y^2+z^2\le 1)\}$ be an infinite
long $3$-dimensional cylinder. We consider an isometry
$\psi_{\varepsilon,m_0}: \mathbb R^3\to \mathbb R^3$ given by
\begin{equation*}
\psi_{\varepsilon,m_0}: \quad \left(
\begin{array}{c}
    x \\
    y \\
    z \\
  \end{array}
\right) \to \left(
  \begin{array}{c}
    \varepsilon \\
    0\\
    0 \\
  \end{array}
\right)+ \left(
  \begin{array}{ccc}
    1 & 0 & 0 \\
    0 & \cos\frac{2\pi}{m_0} & -\sin\frac{2\pi}{m_0} \\
    0 & \sin\frac{2\pi}{m_0} & \cos\frac{2\pi}{m_0}\\
  \end{array}
\right) \left(
  \begin{array}{c}
    x \\
    y \\
    z \\
  \end{array}
\right)
\end{equation*}
where $m_0$ is a fixed integer $\ge 2$. Let
$\Gamma_{\varepsilon,m_0}=\langle\psi_{\varepsilon,m_0}\rangle$ be
an sub-group generated by $\psi_{\varepsilon,m_0}$. It is clear that
the following equation
\begin{equation*}
\psi_{\varepsilon,m_0}^{m_0} \left(
  \begin{array}{c}
    x \\
    y \\
    z \\
  \end{array}
\right)= \left(
  \begin{array}{c}
    x+m_0\varepsilon \\
    y \\
    z \\
  \end{array}
\right)
\end{equation*}
holds. The quotient space $M^3_{\varepsilon,m_0}= (\mathbb R\times
D)/ \Gamma_{\varepsilon, m_0}$ is a solid torus. Let
$$
G_{\varepsilon,m_0}: \mathbb R\times D^2 \to M^3_{\varepsilon,m_0}
$$
be the corresponding quotient map. The orbit $\mathscr
{O}_{0,\varepsilon}=G_{\varepsilon,m_0}(\mathbb R\times\{0\})$ is an
exceptional orbit in $M^3_{\varepsilon,m_0}$. It is clear that such
an exceptional orbit $\mathscr{O}_{0,\varepsilon}$ has non-zero
Euler number. Let $\varepsilon\to 0$, the solid tori
$M^3_{\varepsilon,m_0}$ is convergent to $X^2=D^2/\mathbb Z_{m_0}$,
where $\mathbb Z_{m_0}=\langle\hat {\psi}_{m_0}\rangle$ is a
subgroup generated by $\hat{\psi}_{m_0}: \left(
  \begin{array}{c}
    y \\
    z \\
  \end{array}
\right)\to \left(
  \begin{array}{cc}
    \cos\frac{2\pi}{m_0} & -\sin \frac{2\pi}{m_0} \\
    \sin \frac{2\pi}{m_0} & \cos\frac{2\pi}{m_0} \\
  \end{array}
\right) \left(
  \begin{array}{c}
    y \\
    z \\
  \end{array}
\right). $
\end{example}

Let us now return to the diagram constructed in the previous section:
\begin{diagram}
A^3_{(M^3_{\alpha},\hat g^{\alpha})}(x_{\alpha},
\delta, \varepsilon) &\rTo^{F_{\alpha}}&\mathbb{R}^2\\
\dTo_{\text{G-H}}&&\dCorresponds\\
A_{X^2}(x_{\infty},\delta,\varepsilon) &\rTo^{F_{\infty}}&\mathbb{R}^2
\end{diagram}
Among other things, we shall derive the following theorem.

\begin{theorem}[Solid tori around exceptional circle orbits]\label{thm2.1}
Let $F_\alpha$, $F_\infty$, $A^3_{(M^3_{\alpha},\hat
g^{\alpha})}(x_{\alpha}, \delta, \varepsilon)$ and $A_{X^2}
(x_{\infty}, \delta,\varepsilon)$ as in \autoref{section1} and the
diagram above. Then there is a $\delta^* > 0$ such that $
B_{M^3_{\alpha}}(x_\alpha, \delta^*) $ is homeomorphic to a solid
torus for sufficiently large $ \alpha$. Moreover, a finite normal
cover of $ B_{M^3_{\alpha}}(x_\alpha, \delta^*)$ admits a free
circle action.
\end{theorem}

We will establish \autoref{thm2.1} by using the
Cheeger-Gromoll-Perelman's soul theory for {\it singular metrics} on
open $3$-dimensional manifolds with non-negative curvature,
(comparing with \cite{SY2000}). It will take several steps.

\subsection{A scaling argument and critical points in collapsed
regions}\label{section2.1}\

\smallskip

We start with the following observation.
\begin{prop}[cf.  \cite{SY2000}, \cite{Yama2009}]\label{prop2.2}
Let $\{(B_{(M^3_{\alpha},\hat
g^{\alpha})}(x_{\alpha},\varepsilon),x_{\alpha})\}$ be a sequence of
metric balls convergent to $(B_{X^2}(x_{\infty}, \varepsilon),
x_{\infty})$ as above. Then, there is another sequence of points
$\{x'_\alpha\}$ such that $ x'_\alpha \to x_\infty$ as $M^3_\alpha
\to X^2$ and for sufficiently large $\alpha$, the following is true:
\begin{enumerate}[{\rm (1)}]
\item
$\partial B_{(M^3_{\alpha},\hat
g^{\alpha})}(x'_{\alpha},\varepsilon)$ is homeomorphic to a quotient of  torus
$T^2$;

\item
There exists $0<\delta_{\alpha}<\varepsilon$ such that there is no
critical point of $r_{x'_{\alpha}}(z)=\dis(z,x'_{\alpha})$ for $z\in
A_{(M^3_{\alpha},\hat g^{\alpha})}(x'_{\alpha},\delta_{\alpha},
\varepsilon)$;

\item There is the furthest critical point $z_{\alpha}$ of the
distance function $r_{x'_{\alpha}}$ in $B_{(M^3_{\alpha},\hat
g^{\alpha})}(x'_{\alpha},\delta_{\alpha})$ with
$\lambda_{\alpha}=r_{x'_{\alpha}}(z_{\alpha}) \to 0$ as $\alpha \to
\infty$.
\end{enumerate}
\end{prop}
\begin{proof}
This result can be found in Shioya-Yamaguch's paper \cite{SY2000}.
 For convenience of
readers, we reproduce a proof inspired by Perelman  and Yamaguchi (cf. \cite{Per1994}, \cite{Yama2009}) with appropriate modifications.

 Our choices of the desired points $\{ x'_{\alpha}\}$
are related to  certain averaging distance functions $\{h_\alpha\}$ described below.
Let us first construct the limit function $h_\infty$ of the sequence $\{h_\alpha\}$.
Similar constructions related to $h_\infty$ can be found in
the work of Perelman, Grove and others  (see \cite{Per1994} page 211, \cite{PP1994} page 223, \cite{GW1997} page 210, \cite{Kap2002}, \cite{Yama2009}).
 Let $\Uparrow_p^A$ denote the set of directions of geodesics from
$p$ to $A$ in $\Sigma_p$. It is clear that if we choose $Y^2 = T_{
x_\infty}(X^2 )$ and if $o= o_{ x_\infty}$ is the apex of $T_{
x_\infty}(X^2 )$, then $\Uparrow_{o}^{\partial B_{Y}(o, R)} =
\Sigma_o(Y^2)$ for any $R >0$. Recall that $(\lambda X^2, x_\infty)
\to (Y^2, o_{ x_\infty})$ as $\lambda \to \infty$.  Suppose that
$X^2$ has curvature $\ge -1$. Applying \autoref{prop1.14}, we see
that, for any $\delta >0$, there is a sufficiently small $r$ such
that

\smallskip
\noindent (2.2.4) {\it The minimal set
$\Uparrow_{x_\infty}^{\partial B_{X^2}(x_\infty, 2r)}$ is
$\delta$-dense in $\Sigma_{x_\infty}(X^2)$.}

We now choose $\delta' = \frac{\delta}{800}$ and $\delta =
\frac{2\pi }{1000}$. Let $\{ q_i\}_{1 \le i \le m}$ be a maximal
$(\delta r)$-separated subset in $ \partial B_{X^2}(x_\infty, r)$
and $ \{ q_{i, j}\}_{1 \le j \le N_i}$ be a $(\delta' r) $-net in $
B_{X^2}(q_i, \delta r) \cap \partial B_{X^2}(x_\infty, r)$.
Yamaguchi (\cite{Yama2009})  considered
$$
f_i (x) = \frac{1}{N_i}\sum_{j= 1}^{N_i} \dis(x, q_{i, j})
$$
for $i = 1, \cdots, m$. Our choice of $h_\infty$ is given by
$$
h_\infty(x) = \min_{1 \le i \le m} \{ f_i(x)\}.
$$
 We now
verify that $h_\infty$ has a unique maximal point $ x_\infty$. It is
sufficient to establish
$
h_\infty(x)  \le [r - \frac{1}{2} \dis(x, x_\infty)]
$
for all $ x \in B_{X^2}(x_\infty, \frac r2)$. This can be done as
follows. For each $ x \in B_{X^2}(x_\infty, \frac r2) - \{x_\infty\}$, we
choose $i_0$ such that $\angle_{x_\infty}(x, q_{i_0,j }) <
5\delta$ for $j = 1, \cdots, N_{i_0}$. Suppose that $\sigma: [0, d]
\to X^2$ is a length-minimizing geodesic segment of unit speed from
$x_\infty$ to $x$. By comparison triangle comparison theorems, one can show
that if $ 0 < t < \dis(x_\infty, x)$ then $
\angle_{\sigma(t)}(q_{i_0,j}, \sigma'(t)) < \frac{\pi}{4}$ for $j = 1, \cdots, N_{i_0}$. It follows
from the first variational formula that
$$
\dis(\sigma(t), q_{i_0,j }) \le r - (\cos \frac{\pi}{4}) t,
$$
for $j = 1, \cdots, N_{i_0}$. In fact, Kapovitch \cite{Kap2002} observed that $\dis(\sigma(t),
q_{i_0,j }) \le [r - t \cos (3 \delta)] $, (see \cite{Kap2002} page
129 or \cite{GW1997}). It follows that $h_\infty(x)_{B_{X^2}(x_\infty, \frac r2)}$ has the unique maximum point
$x_\infty$ with $h_\infty(x_\infty) =r$,  because $h_\infty(x) = \min_{1 \le i \le m} \{ f_i(x) \}
\le [r - \frac{1}{2} \dis(x, x_\infty)]$ for $x \in B_{X^2}(x_\infty, \frac r2)$.

Since $ B_{M^3_{\alpha}}(x_{\alpha}, r) \to B_{X^2}(x_{\infty}, r) $ as $\alpha \to \infty$,
we can construct a $ \mu'_\alpha$-approximation of $\{ q_{i, j,
\alpha}\} \subset M^3_\alpha$ of $ \{ q_{i, j} \}\subset X^2 $ with $ \mu'_\alpha \to 0$. Let
$$
f_{i, \alpha}(y) = \frac{1}{N_i}\sum_{j= 1}^{N_i} \dis(y, q_{i, j,
\alpha}) \quad \text{ and} \quad  h_\alpha(y) = \min_{1 \le i \le m}
\{ f_{i, \alpha}(y)\}.
$$
Let $A_\alpha$ be a local maximum set of $h_\alpha|_{B_{M^3_{\alpha}}(x_{\alpha}, \frac r4)}$ and $x'_\alpha
\in A_\alpha$. Applying  \autoref{prop1.14} to
the sequence $\{h_\alpha\} \to h_\infty$, we see that
$\diam(A_\alpha) \to 0$ and $x'_\alpha \to x_\infty$, as
$B_{M^3_{\alpha}}(x_{\alpha}, r) \to B_{X^2}(x_{\infty}, r) $ and $\alpha \to \infty$.

Let $r_\alpha (y) = d(y, x'_\alpha)$.
 Using \autoref{prop1.14} again, we can show
that there exists a sequence $\delta_\alpha \to 0 $ such that
neither   the function $h_\alpha$ nor  $r_\alpha$ has any
critical points in the annual region $A_{(M^3_{\alpha},\hat g^{\alpha})}(x'_{\alpha},
\delta_{\alpha}, \varepsilon)$, as $B_{M^3_{\alpha}}(x'_{\alpha}, r) \to B_{X^2}(x_{\infty}, r) $. It follows
from \autoref{thm1.17} that the boundary $\partial B_{(M^3_{\alpha},\hat
g^{\alpha})} (x'_{\alpha},t)$ is homeomorphic to $T^2$ or Klein
bottle $T^2/ \mathbb Z^2$ for $t>\delta_{\alpha}$. However,
$(M^3_{\alpha},\hat g^{\alpha})$ is a Riemannian $3$-manifold
$\partial B_{(M^3_{\alpha},\hat g^{\alpha})} (x'_{\alpha}, s)$ is
homeomorphic to a $2$-sphere for $s$ less than the injectivity
radius $ \eta_\alpha$ at $x'_{\alpha}$. Thus, we let $\mu_{\alpha} = \min\{\mu'_\alpha, \eta_\alpha \} $ and
$$
\lambda_{\alpha}=\max\{\dis(x'_{\alpha},z_{\alpha})|z_{\alpha}
\text{ is a critical point of } r_{x'_{\alpha}} \text{ in }
B_{(M^3_{\alpha}, \hat g^{\alpha})}(x'_{\alpha},\delta_{\alpha})\}.
$$
Clearly, we have $\lambda_{\alpha}\in
[\mu_{\alpha},\delta_{\alpha}]$. As $\alpha\to+\infty$, we have
$\lambda_{\alpha} \to 0$. This completes the proof of
\autoref{prop2.2}.
\end{proof}

In what follows, we re-choose $ x_\alpha = x'_\alpha$ as in the proof of
\autoref{prop2.2} for each $\alpha$. We now would like to study the
sequence of re-scaled metrics
\begin{equation}\label{eq:2.1}
\tilde{g}^{\alpha}=\frac{1}{\lambda^2_{\alpha}}\hat g^{\alpha}.
\end{equation}
Clearly, the curvature of $\tilde{g}^{\alpha}$ satisfies $
\curv_{\tilde{g}^{\alpha}} \ge -\lambda^2_{\alpha} \to 0$, as
$\alpha \to \infty$. By passing to a subsequence, we may assume that
the pointed Riemannian $3$-manifolds $\{((M^3_{\alpha},\tilde
g^{\alpha}),x'_{\alpha})\}$ converge to a pointed Alexandrov space
$(Y_{\infty}, y_{\infty})$ with non-negative curvature.

\begin{prop}[Lemma 3.6 of \cite{SY2000}, Theorem 3.2 of \cite{Yama2009}]\label{prop2.3}
Let $\{((M^3_{\alpha},\tilde g^{\alpha}),x'_{\alpha})\}$, $X^2$,
$\{\delta_{\alpha}\},\{\lambda_{\alpha}\}$ and $(Y_{\infty},
y_{\infty})$ be as above. Then $Y_{\infty}$ is a complete,
non-compact Alexandrov space of non-negative curvature. Furthermore,
we have
\begin{enumerate}[{\rm (1)}]
\item $\dim Y_{\infty}=3$;
\item $Y_{\infty}$ has no boundary.
\end{enumerate}
\end{prop}
\begin{proof}
This is an established result of \cite{SY2000} and \cite{Yama2009}.
We outline a proof here only for convenience of readers, using our proofs of Proposition 2.2 above and Theorem 2.4 and Proposition 2.5 below.   The metric $\tilde
g^{\alpha}=\frac{1} {\lambda^2_{\alpha}} \hat g^{\alpha}$ defined above has
curvature $\ge -\lambda^2_{\alpha} \to 0$, as $\alpha \to +\infty$.
By our construction,
the diameter of $Y_{\infty}$ is infinite. Moreover, our Alexandrov
space $Y_{\infty}$ has no finite boundary.

It remains to show that $\dim Y_{\infty}=3.$ Suppose contrary,
$\dim(Y^s_\infty) = 2$. Then, for each $\vec v^* \in \Sigma^1_{y_\infty}(
Y^2_\infty)$, the subset
$$
\Lambda^\perp_{\vec v^*} = \{ \vec w \in \Sigma^1_{y_\infty}(
Y^2_\infty) \, | \angle_{y_\infty}(\vec w, \vec v^*) = \frac{\pi}{2}
\}
$$
has at most two elements. We will find  a vector $\vec
v^*$ such that $\Lambda^\perp_{\vec v^*}$ has at least $N_{i_0} \ge 100$
elements for some $i_0$, a contradiction to $\dim(Y_\infty) =2$.

Our choice of $\vec v^*$ will be related to a tangent vector to the minimum set $A$ of a convex function
$f: Y^s_\infty \to \mathbb R$, which we now describe. We will retain the same notations as in
the proof of Proposition 2.2 above. Let $r_{x'_\alpha}(z) = d_{M^3_{\alpha}}(z, x'_\alpha)$,   $z_\alpha$ be a critical point of
$r_{x'_\alpha}|_{B_{M^3_{\alpha}}(x'_{\alpha}, \frac r4)}$ be as above and $\bar{z}_\alpha$ be its image in
the scaled manifold $\frac{1}{\lambda_\alpha} M^3_\alpha$. Suppose
that $q_\infty$ is the limit point of a subsequence of $\{
\bar{z}_\alpha \}$. It follows from \autoref{prop1.14} that the
limiting point $q_\infty$ must be a critical point of the distance
function $r_{y_\infty}$ with $r_{y_\infty}(q_\infty)= 1$, where
$r_{y_\infty} (z) = \dis_{Y_\infty}(z, y_\infty)$.

In addition, there  exist $(N_1 +\cdots+ N_m)$-many geodesic segments $\{\tilde
\gamma_{i, j, \alpha} \}_{1 \le j \le N_i, 1\le i \le m}$ in
$M^3_{\alpha}$ from $x'_\alpha$ to $q_{i,j, \alpha}$, where $\tilde \gamma_{i,j, \alpha}\to \tilde
\gamma_{i, j, \infty}$ as $B_{M^3_\alpha}(x'_\alpha, r)  \to B_{X^2}(x_\infty, r)$ with $\alpha \to \infty$. Let $\bar\gamma_{i,
j,  \alpha}: [0, \frac{r }{2\lambda_\alpha}] \to
\frac{1}{\lambda_\alpha} M^3_\alpha$ be the re-scaled geodesic with
starting point $\bar{x}'_\alpha$ in the re-scaled manifold, for $1
\le j \le N_i, 1\le i \le m.$ It can be shown that $\bar\gamma_{i, j, \alpha} \to \bar\gamma_{i,j,
\infty}$ as $\frac{1}{\lambda_\alpha} M^3_\alpha \to Y_\infty$,
after passing to appropriate subsequences of $\{ \alpha\}$.
Therefore, we have $(N_1 + \cdots + N_m)$-many distinct geodesic
rays starting from $y_\infty$ in $Y_\infty$. Let us now consider limiting
Busemann functions:
$$
\tilde{h}_{i, j}(y) = \lim_{t \to \infty}[\dis (y, \bar\gamma_{i,j,
\infty}(t)) - t] \quad \text{and } \quad  \hat {h}_i (y) =
\frac{1}{N_i}\sum_{j= 1}^{N_i}\tilde{h}_{i, j}(y).
$$
Since $Y_\infty$ has non-negative curvature, each Busemann function
$(-\tilde{h}_{i, j})$ is a convex function, (see \cite{CG72}, \cite{Wu79} or Theorem 2.4 below). If

$$
\hat h (y) = \min_{ 1\le i \le m}\{ \hat {h}_i \} \quad \text{ and }
\quad f(y) = - \hat h(y),
$$
then $f$ is  convex.
Choose  $\tilde{h}_{i, j, \alpha}(x) = [\bar{d}(x,
\bar{q}_{i, j, \alpha}) - \bar{d}( \bar{x}'_\alpha,
\bar{q}_{i, j, \alpha}) ] $, $\hat {h}_{i,
\alpha} (x) = \frac{1}{N_i}\sum_{j= 1}^{N_i}\tilde{h}_{i, j,
\alpha}(x)$ and  $ \hat h_\alpha (x) = \min _{ 1\le i \le m}\{  \hat
{h}_i(x)\}$ defined on $B_{\frac{1}{\lambda_\alpha} M^3_\alpha} (
\bar{x}'_\alpha, \frac{r}{4\lambda_\alpha}  )$. It follows  that
$\hat h = \lim_{\alpha \to \infty }
\hat h_\alpha $.  Because $\bar{x}'_\alpha$ is a maximum point of $\hat h_\alpha$ with $ \hat h_\alpha (\bar{x}'_\alpha ) = 0 $ and $\bar{x}'_\alpha \to
y_\infty$ as $\alpha \to \infty$, the point
$y_\infty$ is a critical  point   of the limiting function $\hat h = \lim_{\alpha \to \infty }
\hat h_\alpha $ with $\hat h(y_\infty) = 0$.

 Thus, $0 = \hat h (y_\infty)$ is a
critical value of the {\it convex } function $f = (- \hat h)$   with $\inf_{y \in Y^s_\infty}\{ f(y) \} = 0$.
 There are two cases for $A = f^{-1} (0)$.

\smallskip
{\bf Case 1.} If $A = \{ y_\infty \}$, then it is known (cf.
\autoref{prop2.5} below) that the distance function $r_{y_\infty}$
does not have any critical point in $[Y_\infty - \{ y_\infty \}]$,
which contracts to the existence of critical points $q_\infty$ of
$r_{y_\infty}$ mentioned above. Thus, this case can not happen.

\smallskip
{\bf Case 2.} $\dim(A) \ge 1$. In this case, our proof becomes more
involved. If $q_\infty \notin A = f^{-1}(0)$, then
$f(q_\infty) = a > a_0 =0$. Using the proof of \autoref{prop2.5}
below, we see that $q_\infty$ can {\it not} be a critical
point of $r_{y_\infty}(z) = \dis(z, y_\infty)$ either, a contradiction.
Thus, $q_\infty \in A$ holds.

If, for any quasi-geodesic segment $\sigma: [a, b] \to Y_\infty$
with ending points $\{ \sigma(a), \sigma(b)\} \subset \Omega$, the
inclusion relation $\sigma([a, b]) \subset \Omega$ holds,  then
$\Omega$ is called a totally convex subset of $Y_\infty$. It follows from
the proof of \autoref{prop2.5} below that  the sub-level set
$f^{-1}((-\infty, a]) $ is totally convex. Let us choose $\vec{v}^* \in \Uparrow_{y_\infty}^{q_\infty} $, where
$\Uparrow_p^x$ denotes the set of directions of geodesics from $p$
to $x$ in $\Sigma_p$. Because  $\{ y_\infty, q_\infty\}$
are contained in the totally convex  minimal set $A = f^{-1}(0)$ of
a convex function $f$, one has
$$
\angle_{y_\infty}(\vec v^*, \bar \gamma_{i,j,  \infty}'(0)) \ge
\frac{\pi}{2}
$$
holds for all $i, j$ by our construction of $f = - \hat h$, because $\bar
\gamma_{i,j,  \infty}'(0) = \vec s_{i, j}$ is a support vector of
$d_{y_\infty}(-f)$, (see the proof of \autoref{prop2.5} below).
Let $\sigma: [0, \ell] \to Y$ be a geodesic segment from $y_\infty$
to $q_\infty$. Since $A$ is totally convex and $f$ is convex, we have $\sigma( [0,
\ell]) \subset A$ and $f( \sigma (t)) = 0$ for all $t \in [0,
\ell]$. Hence, $\hat h ( \sigma (t)) = 0$ for all $t \in [0,
\ell]$.

Recall that $\hat h (z) = \min_{ 1\le i \le m}\{ \hat {h}_i
(z)\} $.  We choose $i_0 $ such that $  \hat h_{i_0} (\sigma(\frac{\ell}{2})) = \min_{ 1\le i \le m}\{ \hat {h}_i
(\sigma(\frac{\ell}{2}))\} = \hat h ( \sigma (\frac{\ell}{2})) =0 $.    Because $\hat
h_{i_0}(\sigma(t))$ is a concave function of $t$ with  $ 0 = \hat h (\sigma(0))= \hat h (\sigma(\ell))  \le  \min\{ \hat
h_{i_0} (\sigma(0)), \hat
h_{i_0} (\sigma(\ell)) \} $ and  $ \hat
h_{i_0} (\sigma(\frac{\ell}{2})) = 0$, one concludes that  $\hat h_{i_0} ( \sigma (t)) = 0$ for all $t \in [0,
\ell]$. Since $\hat h_{i_0} ( \sigma (t)) = 0$ for all $t \in [0,
\ell]$, choosing $\vec{v^*}= \sigma'(0) $ one  has
$$
0 = (d_{y_\infty} \hat h_{i_0}) (\vec{v^*} ) = \frac{1}{N_{i_0}}
\sum^{N_{i_0}}_{j=1} [ - \cos ( \angle_{y_\infty}(\vec v^*,
\bar\gamma_{i_0, j, \infty}'(0)))],
$$
  This together with inequalities
$\angle_{y_\infty}(\vec v^*, \bar\gamma_{i_0, j, \infty}'(0)) \ge
\frac{\pi}{2}$ implies that
$$
\angle_{y_\infty}(\vec v^*, \bar\gamma_{i_0, j, \infty}'(0)) =
\frac{\pi}{2}
$$
holds for $j = 1, 2,\cdots, {N_{i_0}}.$ Hence, we conclude that
$\bar\gamma_{i_0,j, \infty}'(0) \in \Lambda^\perp_{\vec v^*}$, for
$j= 1, 2, \cdots, {N_{i_0}}$, where ${N_{i_0}} \ge 100$. Therefore, we
demonstrated that $\# |\Lambda^\perp_{\vec v^*}| \ge N_{i_0} \ge
100$. This contradicts to $\# |\Lambda^\perp_{\vec v^*}| \le 2$ when
$\dim(Y_\infty) = 2$. This completes the proof of the assertion $\dim ( Y_{\infty}) =3$.
\end{proof}

\subsection{The classification of non-negatively curved
surfaces and 3-dimensional soul theory}\label{section2.2}\

\smallskip

In what follows, if $X$ is an open Alexandrov space of non-negative
curvature, then we let $X(\infty)$ be the boundary (or called the
ideal boundary) of $X$ at infinity. For more information about the
ideal boundary $X(\infty)$, one can consult with work of Shioya, (cf. \cite{Shio94}).

In this sub-section, we briefly review the soul theory for
non-negatively curved space $Y^k_{\infty}$ of dimension $\le 3$.
The soul theory and the splitting theorem are two important tools
in the study of low dimensional collapsing manifolds.

Let $X$ be an $n$-dimensional non-negatively curved Alexandrov
space. Suppose that $X$ is a non-compact complete space and that $X$
has no boundary. Fix a point $x_0\in X$, we consider the
Cheeger-Gromoll type function
$$f(x)=\lim_{t\to+\infty}[t-\dis(x,\partial B(x_0,t))].$$

Let us consider the sub-level sets $\Omega_c=f^{-1}((-\infty,c])$.
We will show that $\Omega_c$ is a totally convex subset for any $c$
in \autoref{thm2.4} below.

H.Wu \cite{Wu79} and Z. Shen \cite{Shen1996} further observed that
\begin{equation}\label{eq2.21}
f(x)=\sup_{\sigma\in \Lambda}\{h_{\sigma}(x)\}
\end{equation}
where $\Lambda=\{\sigma:[0,+\infty)\to X| \sigma(0)=x_0,
\dis(\sigma(s), \sigma(t))=|s-t|\}$ and $h_{\sigma}$ is a Busemann
function associated with a ray $\sigma$ by
\begin{equation}\label{eq2.22}
h_{\sigma}(x)=\lim_{t\to \infty}[t-\dis(x,\sigma(t))].
\end{equation}
Since $\Omega_c=f^{-1}((-\infty,c])$ is convex, by \eqref{eq2.21} we
see that $\Omega_c$ contains no geodesic ray starting from $x_0$.
Choose $\hat c = \max\{c, f(x_0) \}$. Since $\Omega_{\hat c}$ is
totally convex and contains no geodesic rays, $\Omega_{\hat c}$ must
be compact. It follows that $\Omega_c \subset \Omega_{\hat c}$ is
compact as well. Thus the Cheeger-Gromoll function $f(x)$ has a
lower bounded
$$a_0=\inf_{x\in M^n}\{f(x)\}=\inf_{x\in\Omega_0}\{f(x)\}>-\infty.$$
If $\Omega_{a_0}=f^{-1}(a_0)$ is a space without boundary,
$\Omega_{a_0}$ is called a soul of $X$. Otherwise, $\partial
\Omega_{a_0}\ne \varnothing$ we further consider
$$
\Omega_{a_0-\varepsilon}=\{x\in\Omega_{a_0}|
\dis(x,\partial\Omega_{a_0})\ge\varepsilon\}
$$
When $X$ is a smooth Riemannian manifold of non-negative
curvature, Cheeger-Gromoll \cite{CG72} showed that
$\Omega_{a_0-\varepsilon}$ remains to be convex. For more general
case when $X$ is an Alexandrov space of non-negative curvature,
Perelman \cite{Per1991} also showed that
$\Omega_{a_0-\varepsilon}$ remains to be convex, (see
\cite{Petr2007} and \cite{CMD2009} as well).

Let $l_1=\max_{x\in\Omega_{a_0}}\{\dis(x,\partial\Omega_{a_0})\}$
and $a_1=a_0-l_1$. If $\Omega_{a_1}=\Omega_{a_0-l_1}$ has no
boundary, then we call $\Omega_{a_1}$ a soul of $X$. Otherwise, we
repeat above procedure by setting
$$
\Omega_{a_1-\varepsilon}=\{x \in\Omega_{a_1}|
\dis(x,\partial\Omega_{a_1})\ge\varepsilon\}
$$
for $0 \le \varepsilon \le l_1$. Observe that
$$n=\dim(X)>\dim(\Omega_{a_0})>\dim(\Omega_{a_1})>\cdots$$
Because $X$ has finite dimension, after finitely many steps we will
eventually get a sequence
$$a_0>a_1>a_2>\cdots>a_m$$
such that
$\Omega_{a_i-s}=\{x\in\Omega_{a_i}|\dis(x,\partial\Omega_{a_i})\ge
s\}$ for $0\le s\le a_i-a_{i+1}$ and $i=0,1,2,\cdots,m-1$. Moreover,
$\Omega_{a_m}$ is a convex subset without boundary, which is called
a soul of $X$.

A subset $\Omega$ is said to be {\it totally convex} in $X$ if for
any quasi-geodesic segment $\sigma:[a,b]\to X$ with endpoints
$\{\sigma(a),\sigma(b)\}\subset \Omega$, we must have
$\sigma([a,b])\subset \Omega$. The definition of quasi-geodesic can
be found in \cite{Petr2007}.

\begin{theorem}[Cheeger-Gromoll \cite{CG72}, Perelman \cite{Per1991}, \cite{SY2000}]\label{thm2.4}
Let $X$ be an $n$-dimensional open complete Alexandrov space of
curvature $\ge 0$, $f(x)=\lim_{t\to+\infty}[t-\dis(x,\partial B(\hat
x,t))]$, $\{\Omega_s\}_{s\ge a_m}$ and $a_0>a_1>\cdots>a_m$ be as
above. Then the following is true.

{\rm (1)} For each $s\ge a_0$, $\Omega_s=f^{-1}((-\infty,s])$ is a
totally convex and compact subset of $X$;

{\rm (2)} If $a_i\le s<t\le a_{i-1}$, then
$$\Omega_s=\{x\in \Omega_t|\dis(x,\partial \Omega_t)\ge t-s\}$$
remains to be totally convex;

{\rm (3)} The soul $N^k=\Omega_{a_m}$ is a deformation retract of
$X$ via multiple-step Perelman-Sharafutdinov semi-flows, which are
distance non-increasing.
\end{theorem}
\begin{proof}
(1) For $s\ge a_0$, we would like to show that
$\Omega_s=f^{-1}((-\infty,s])$ is totally convex. Suppose contrary,
there were a quasi-geodesic $\sigma:[a,b]\to X$ with
$$\max\{f(\sigma(a)),f(\sigma(b))\}\le s$$
and $c$ with $a<c<b$ and
$$f(\sigma(c))>s\ge\max\{f(\sigma(a)),f(\sigma(b))\}.$$
\begin{figure*}[ht]
\includegraphics[width=180pt]{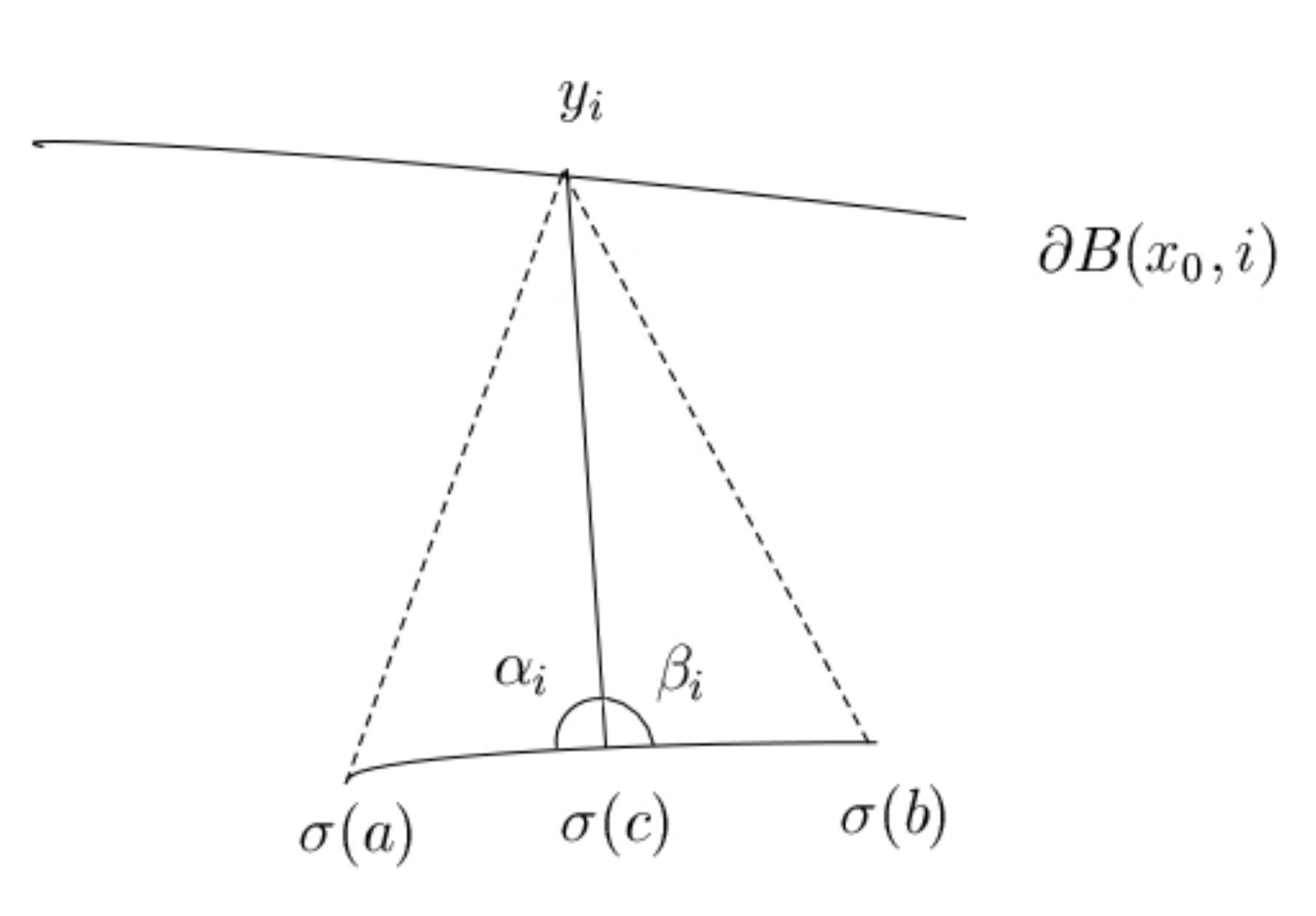}\\
\end{figure*}
For each integer $i\gg 1$, we choose $y_i\in\partial B(x_i,i)$ such
that
$$\dis(y_i,\sigma(c))=\dis(\partial B(x_0,i),\sigma(c)).$$
Let $\alpha_i=\measuredangle_{\sigma(c)}(y_i,\sigma(a))$ and
$\beta_i=\measuredangle_{\sigma(c)}(y_i,\sigma(b))$. Since $X$ has
non-negative curvature and $\sigma:[a,b]\to X$ is a quasi-geodesic,
it is well-known (\cite{Petr2007}) that
$$\cos\alpha_i + \cos\beta_i\ge 0.$$
It follows that
$$\min\{\alpha_i,\beta_i\}\le\frac{\pi}{2}.$$
After passing a sub-sequence and re-indexing, we may assume that
$$\beta_{i_j}\le\frac{\pi}{2}$$
for all $j\ge 1$. By law of cosine, we have
$$[\dis(y_{i_j},\sigma(b))]^2\le[\dis(\sigma(c),y_{i_j})]^2+|b-c|^2.$$
Therefore, we have
\begin{align*}
f(\sigma(b))&\ge\lim_{j\to+\infty}[i_j-\dis(\sigma(b),y_{i_j})]\\
&\ge \lim_{j\to+\infty}[i_j- \sqrt{[\dis(\sigma(c),y_{i_j})]^2+|b-c|^2}]\\
&=\lim_{j\to+\infty}\frac{i_j^2-[\dis(\sigma(c),y_{i_j})]^2
-|b-c|^2}{i_j+ \sqrt{[\dis(\sigma(c),y_{i_j})]^2+|b-c|^2}}\\
& = \lim_{j\to+\infty} [i_j - \dis( \sigma(c),y_{i_j} )]+ 0 \\
& = \lim_{j\to+\infty} [i_j - \dis( \sigma(c), \partial B(\hat x, i_j))] \\
&=f(\sigma(c))
\end{align*}
which is contracting to
$$f(\sigma(c))>f(\sigma(b)).$$
Hence, $\Omega_c$ is a totally convex subset of $X$.

(2) Perelman \cite{Per1991} showed that if $\Omega_c$ is a convex
subset of $X$ with non-empty boundary, then the distance function
$$r_{\partial\Omega_c}(x)=\dis(x,\partial\Omega_c)$$
is concave for $x\in \Omega_c$, (see \cite{Petr2007} and
\cite{CMD2009} as well).

(3) Because our function $(-f(x))$ and $r_{\partial\Omega_{a_i}}$
are concave in $ \Omega_{a_i}$, the corresponding semi-flows are
distance non-increasing, (see Chapter 6 of \cite{Per1991}, section 2
of \cite{KPT2009} or \cite{Petr2007}). Using the
Perelman-Sharafutdinov flow $\frac{d^+ \psi}{dt} = \frac{\nabla
(-f)}{|\nabla (-f)|^2}|_{\psi (t)} $, Perelman (cf. \cite{Per1991})
showed that $X$ is contractible to $\Omega_{a_0}$. Let $r_i:
\Omega_i \to \mathbb R$ be the distance function $r_i(z) = \dis(z,
\partial \Omega_{a_i})$ if $\partial \Omega_{a_i} \neq \emptyset$
for $i = 0, 1, \cdots , m-1$. For the same reasons, $\Omega_{a_0}$
is contractible to $\Omega_{a_1}$ via the Perelman-Sharafutdinov
flow $ \frac{d^+ \psi}{dt} = \frac{\nabla r_0}{|\nabla r_0
|^2}|_{\psi (t)} $. Using the $m$-step Perelman-Sharafutdinov flows,
we can see that the soul $N^k = \Omega_{a_m}$ is a deformation
retract of $X$.
\end{proof}

\begin{prop}\label{prop2.5}
Let $f(y) $ be a {\it convex} function on $Y$ with
$a_0=\inf_{w\in Y}\{ f(w)\} > - \infty $ and $A = f^{-1}( a_0)$ as
in the proofs of \autoref{prop2.3} and \autoref{thm2.4} above.
Suppose that $Y$ is an open and complete Alexandrov space with
non-negative curvature and $A' \subset A$ is a closed subset of $A$.
Then the distance function $r_{A'}(y) = \dis(y, A')$ from $A'$ has
no critical points in the complement $[Y - A]$ of $A$.
\end{prop}
\begin{proof}
 For each $y \notin A$
and $a = f(y) > a_0$, we observe that $A \subset \interior
(\Omega_a) = f^{-1}( ( -\infty, a))$. Let $\sigma: [0, \ell] \to Y$ be a length-minimizing
geodesic segment of unit speed from $A'$ to $y$ with $\sigma(0) \in A'$ and $\sigma(\ell) =
y$. Since $Y$ has no boundary, any geodesic $\sigma$ can be extended
to a longer quasi-geodesic of unit speed $\tilde \sigma: [0, \ell +
\varepsilon] \to Y$, (see \cite{Petr2007}). Since $f$ is convex, the
composition of function $t \to f( \tilde \sigma (t))$ remains
convex for any quasi-geodesics $ \tilde \sigma (t) $ (see \cite{Petr2007}). It follows that
$$
     \frac{d^+ (f \circ \tilde \sigma)}{dt} (\ell)
     \ge \frac{a - a_0}{\ell} > 0.
$$
Let us consider a minimum direction $\vec \xi_{min} \in \Sigma_y(Y)$
of $d_y (-f)$ and $ \vec s = [ - d_y(-f)( \vec \xi_{min})] \vec
\xi_{min} $, where we used the fact that
$$
[ - d_y(-f)( \vec \xi_{min})] \ge  \frac{d^+ (f \circ \tilde \sigma)}{dt} (\ell) > 0.
$$
Hence we have $ \vec s = [ - d_y(-f)( \vec \xi_{min})] \vec\xi_{min}
\neq 0 $. The vector $\vec s$ is called a support vector of $d_y(-f)$.
 For any support vector $\vec s$,
one has (cf. \cite{Petr2007} page 143) that inequality
$$
  d_y( -f) ( \vec u) \le - \langle \vec s, \vec u \rangle
$$
holds for all $\vec u \in \Sigma_y(Y)$, where $(-f)$ is a semi-concave
function.
Let $\Omega_c = f^{-1}((-\infty, c])$. When $f$ is convex, one has
$$
f( \tilde \sigma(t) ) \le \max \{f( \tilde \sigma(a) ), f( \tilde
\sigma(b) ) \}
$$
for any quasi-geodesic $\tilde \sigma: [ a, b] \to Y$ and $t \in [a,
b]$. Thus, $\Omega_c$ is totally convex. It follows that, for any
direction $\vec w \in \Uparrow_y^{A'}$, we have $\vec w \in T_y(
\Omega_c)$. Moreover, we have
$$
d_y(-f) (\vec w) \ge \frac{ - a_0 - (-a) }{ \ell } > 0
$$
for all $\vec w \in \Uparrow_y^{A'}$, since $(-f)$ is a concave
function. Because $\vec s$ is a support vector of $d_y(-f)$, one also has
$$
   0 < d_y(-f) (\vec w) \le - \langle \vec s, \vec w \rangle
$$
holds for $\vec w \in \Uparrow_y^{A'}$, (cf. \cite{Petr2007} page 143).
Thus, we have
$$
\angle_y( \vec w, \vec s) > \frac{\pi}{2}
$$
for all $\vec w \in \Uparrow_y^{A'}$. It follows that $d_y(r_{A'})(
\vec s) > 0$ and $y $ is not a critical point of the distance
function $r_{A'}(.) $, where $y \notin A = f^{-1}(a_0)$, (compare
with \cite{Grv1993}). This completes the proof of \autoref{prop2.5}.
\end{proof}

\begin{figure*}[ht]
\includegraphics[width=100pt]{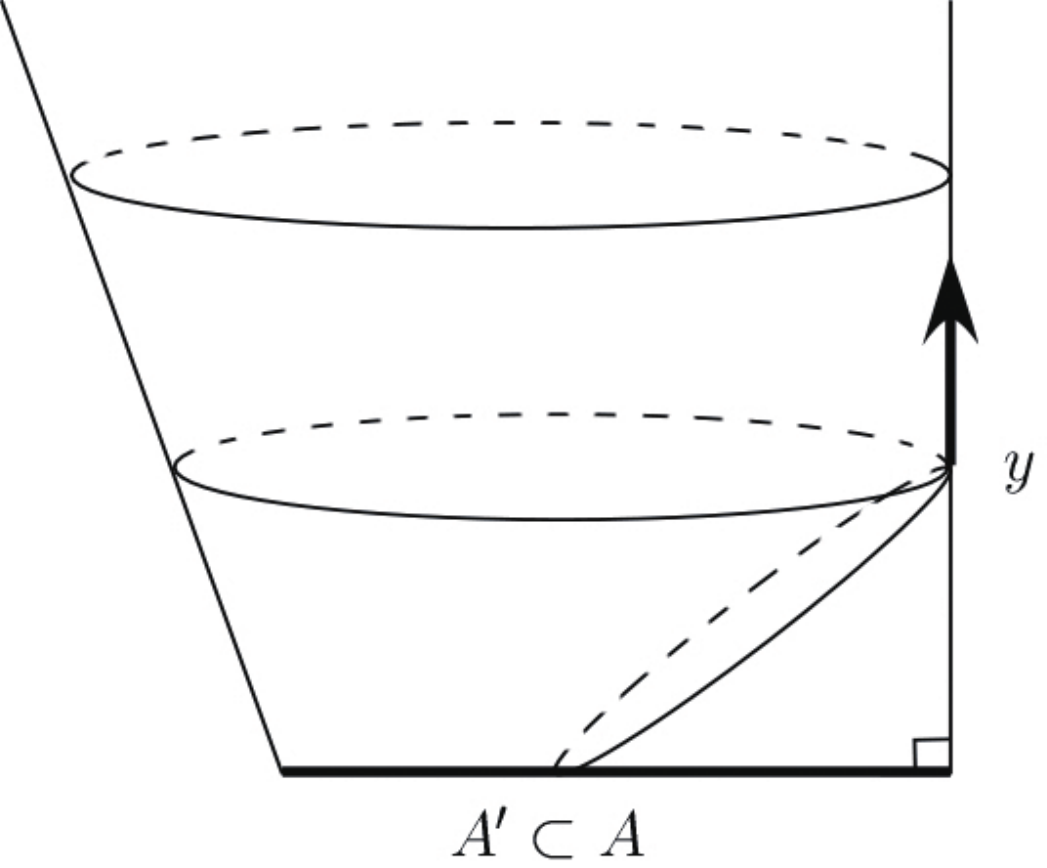}\\
\caption{The minimum set $A$ of a convex function.}\label{fig2_4}
\end{figure*}

Using soul theory and splitting theorem, we can classify
non-negatively curved surfaces with possibly singular metrics.

\begin{theorem}\label{thm2.6}
Let $X^2$ be an oriented, complete and open surface of non-negative
curvature. Then $X^2$ is either homeomorphic to $\mathbb R^2$ or
isometric to a flat cylinder.
\end{theorem}
\begin{proof}
It is known that $X^2$ is a manifold. Let $N^s = \Omega_{a_m}$ be a
soul of $X^2$. If the soul $N^s = \Omega_{a_m}$ is a single point,
then $X^2$ is homeomorphic to $\mathbb R^2$. When $N^s =
\Omega_{a_m}$ has dimension $1$, then $N^1 = \Omega_{a_m}$ is
isometric to embedded closed geodesic $\sigma: S^1 \to X^2$, (i.e.,
$N^1=\sigma(S^1)$).

Let $\tilde X^2$ be the universal cover of $X^2$ with lifted metric
and $\tilde{\sigma}: \mathbb R\to \tilde X^2$ be a lift of
$N^1 = \Omega_{a_m}$ in $X^2$. We observe that
\begin{diagram}
\tilde X^2 &\rTo^{\tilde P}&\tilde \sigma\\
\dTo &&\dTo\\
X^2&\rTo^{P}& N^1
\end{diagram}
Suppose that $P: X^2\to N^1$ is the Perelman-Sharafutdinov distance
non-increasing projection from open space $X^2$ to its soul $N^1$.
Such a distance non-increasing map $P: X^2\to N^1$ can be lifted to
a distance non-increasing map $\tilde P: \tilde X^2\to \tilde
\sigma$. Thus $\tilde \sigma: \mathbb R\to \tilde X^2$ is a line in
an open surface $\tilde X^2$ of non-negative curvature. Applying the
splitting theorem, we see that $\tilde X^2$ is isometric to $\mathbb
R^2$. It follows that $X^2$ is a flat cylinder.
\end{proof}

Let us now turn our attention to closed surfaces of curvature.

\begin{cor}\label{cor2.7}
Let $X^2$ be a closed $2$-dimensional Alexandrov space of
non-negative curvature. Then the following holds:

{\rm (1)} If the fundamental group $\pi_1(X^2)$ is finite, then
$X^2$ is homeomorphic to $S^2$ or $\mathbb{RP}^2$.

{\rm (2)} If the fundamental group $\pi_1(X^2)$ is an infinite
group, then $X^2$ is isometric to a flat tours or flat Klein bottle.
\end{cor}
\begin{proof}
After passing through to its double when needed, we may assume that
$X^2$ is oriented.

When $|\pi_1(X^2)|<\infty$, $X^2$ is covered by $S^2$.

When $|\pi_1(X^2)|=+\infty$ and $X^2$ is oriented, for a non-trivial
free homotopy class of a closed curve $[\hat{\sigma}]\ne 0$ in
$\pi_1(X^2)$ with $[\hat{\sigma}^n]\ne 0$ for all $n\ne 0$, we
choose a length minimizing closed geodesic $\sigma: S^1\to X^2$.
Suppose that $\tilde X^2$ is a universal cover of $X^2$ and $\tilde
\sigma: \mathbb R\to \tilde X^2$ is a lift of $\sigma$ in $X^2$.
Then we can check that $\tilde \sigma$ is a geodesic line of $\tilde
X^2$. Thus, $\tilde X^2$ is isometric to $\mathbb R^2$. It follows
that $X^2$ is isometric to a flat torus, whenever $X^2$ is oriented
with $|\pi_1(X^2)|=+\infty$.
\end{proof}

\begin{example}\label{ex2.9}
When $X^2$ is an open surface of non-negative curvature, it might
happen that $\Omega_{a_0}$ is an interval. For instance, let $\hat
Y^2=[0,1]\times [0,+\infty)$ be a flat half-strip in $\mathbb R^2$.
If we take two copies of $\hat Y^2$ and glue them along the
boundary, the resulting surface $X^2= {\rm dbl}(\hat Y^2)$ is
homeomorphic to $\mathbb R^2$. A result of Petrunin implies that
$X^2= {\rm dbl}(\hat Y^2)$ still have non-negative curvature (e.g.,
\cite{Petr2007} or \cite{BBI2001}). In this case, we have
$\Omega_{a_0}$ is an interval. Of course, the soul $N^s =
\Omega_{a_1}$ of $X^2$ is a single point.
\end{example}

We now say a few words for non-negatively curved surfaces $X^2$ with
non-empty convex boundary. By definition of surface $X^2$ with
curvature $\ge k$, its possibly non-empty boundary $\partial X^2$
must be convex.

\begin{cor}\label{cor2.9}
Let $X^2$ be a surface with non-negative curvature and non-empty
boundary. Then

{\rm (1)} If $X^2$ is compact, then $X^2$ is either homeomorphic to
$D^2$ or isometric to $S^1\times [0,l]$ or a flat M\"{o}bius band;

{\rm (2)} If $X^2$ is non-compact and oriented, then $X^2$ is either
homeomorphic to $[0,+\infty)\times \mathbb R$ or isometric to one of
three types: $S^1\times [0,+\infty)$, a half flat strip or
$[0,l]\times (-\infty,+\infty)$.
\end{cor}
\begin{proof}
If we take two copies of $X^2$ and glue them together along their
boundaries, the resulting surface ${\rm dbl}(X^2)$ still has
curvature $\ge 0$, due to a result of Petrunin \cite{Petr2007}.
Clearly, ${\rm dbl}(X^2)$ has no boundary.

(1) When ${\rm dbl}(X^2)$ is compact and oriented, then ${\rm
dbl}(X^2)$ is homeomorphic to the unit $2$-sphere or is isometric to
a flat strip. Hence, $X^2$ is either homeomorphic to $D^2$ or
isometric to $S^1\times [0,l]$ or a flat M\"{o}bius band.

(2) When ${\rm dbl}(X^2)$ is non-compact, then ${\rm dbl}(X^2)$ is
homeomorphic to $\mathbb R^2$ or isometric to $S^1\times \mathbb R$
or $X^2$ is isometric to $[0, \ell] \times [0, \infty)$.

To verify this assertion, we consider the soul $N^s$ of ${\rm
dbl}(X^2)$. If $N^s$ is a circle, then ${\rm dbl}(X^2)$ is isometric
to an infinite flat cylinder: $ S^1(r) \times \mathbb R$. If the
soul $N^s$ is a point, then ${\rm dbl}(X^2)$ is homeomorphic to
$\mathbb R^2$.

There is a special case which we need to single out: $X^2$ is
isometric to $[0, \ell] \times [0, \infty)$. We will elaborate this
special case in \autoref{section5} below.
\end{proof}

\begin{remark}\label{remark2.11} In \autoref{section5} below, we will
estimate the number of extremal points, i.e. essential
singularities, on surfaces with non-negative curvature, using
multi-step Perelman-Sharafutdinov flows associated with the
Cheeger-Gromoll convex exhaustion.
\end{remark}

Finally, we would like to classify all non-negatively curved open
$3$-manifolds with possibly singular metrics.

\begin{theorem}[\cite{SY2000}]\label{thm2.11}
Let $Y^3_{\infty}$ be an open complete $3$-manifold with a possibly
singular metric of non-negative curvature. Suppose that
$Y^3_{\infty}$ is oriented and $N^s$ is a soul of $Y^3_\infty$. Then
the following is true.
\begin{enumerate}[{\rm (1)}]
\item
When $\dim(N^s)= 1$, then the soul of $Y^3_{\infty}$ is isometric to
a circle. Moreover, its universal cover $\tilde Y^3_{\infty}$ is
isometric to $\tilde X^2\times \mathbb R$, where $\tilde X^2$ is
homeomorphic to $\Bbb R^2$;
\item
When $\dim(N^s) = 2$, then the soul of $Y^3_{\infty}$ is
homeomorphic to $S^2/\Gamma$ or $T^2/\Gamma$. Furthermore,
$Y^3_{\infty}$ is isometric to one of four spaces: $S^2 \times
\mathbb R$, $\mathbb R P^2 \ltimes \mathbb R = (S^2 \times \mathbb
R)/ \mathbb Z_2$, $T^2 \times \mathbb R$ or $K^2 \ltimes \mathbb R =
(T^2 \times \mathbb R)/\mathbb Z_2$, where $K^2$ is the flat Klein
bottle and $\mathbb R P^2 \ltimes \mathbb R$ is homeomorphic to $
[\mathbb {RP}^3 - \bar{B}^3(x_0, \varepsilon)]$;

\item
When $\dim(N^s) = 0$, then the soul of $Y^3_{\infty}$ is a single
point and $Y^3_{\infty}$ must be homeomorphic to $\mathbb R^3$.

\end{enumerate}
\end{theorem}
\begin{proof}
This theorem is entirely due to Shioya-Yamaguchi \cite{SY2000}. A
special case of \autoref{thm2.11} for smooth open $3$-manifold with
non-negative curvature was stated as Theorem 8.1 in Cheeger-Gromoll's
paper \cite{CG72}.

Shioya-Yamaguchi's proof is quiet technical, which occupied the half
of their paper \cite{SY2000}. For convenience of readers, we present
an alternative shorter proof of Shioya-Yamaguchi's soul theorem for
3-manifolds with possible singular metrics.

\smallskip
\noindent {\bf Case 1.} When the soul $N^1$ of $Y^3_\infty$ is a
closed geodesic $\sigma^1$, there are distance non-increasing
multi-step Perelman-Sharafutdinov retractions from $Y^3_\infty$ to
$\sigma^1$. Thus, $\sigma^1$ is length-minimizing in its free
homotopy class. It follows that the lifting geodesic $\tilde
\sigma^1$ is a geodesic line in the universal covering space $\tilde
Y^3_\infty$ of $Y^3_\infty$. Using the splitting theorem (cf.
\cite{BBI2001}) for non-negatively curved space $\tilde Y^3_\infty$,
we see that $\tilde Y^3_\infty$ is isometric to $\tilde X^2\times
\mathbb R$, where $\tilde X^2$ is a contractible surface with
non-negative curvature. Hence, $\tilde{X}^2$ is homeomorphic to
$\Bbb R^2$.

\smallskip
\noindent {\bf Case 2.} When the soul $N^2$ of $Y^3_\infty$ is a
surface $X^2$, we observe that $X^2 = f^{-1}(a_0)$ is a convex
subspace of $Y^3_\infty$, where $f(x) = \lim_{t \to \infty} [ t -
\dis(x, \partial B_{Y^3_\infty} (\hat x, t))]$ and $a_0 = \inf_{x
\in Y^3_\infty}\{f(x)\}$. Since $f$ is convex and $Y^3$ has
non-negative curvature, $X^2$ has non-negative curvature as well. By
\autoref{cor2.9}, we see that $X^2$ is either homeomorphic to a
quotient of $S^2$ or isometric to a quotient of a flat torus.

For this case, our strategy goes as follows. We will show that there
is a {\it ``normal line bundle"} over the soul $X^2$. After passing
its double cover if needed, we may assume that such a {\it ``normal
line bundle"} is topologically trivial in $Y^3_\infty$. In this
case, with some extra efforts, one can show that there is a geodesic
line $\hat \sigma^1$ orthogonal to $X^2$ in $Y^3_\infty$. Thus, the
space $Y^3_\infty$ (or its double cover) splits isometrically to
$X^2\times \mathbb R$.

Here is the detail of our {\it ``normal line bundle"} argument. For
each point $x$ in the soul $X^2$, its unit tangent space
$\Sigma_x^1(X^2)$ is homeomorphic to $S^1$. Recall that the space of
unit tangent directions $\Sigma^2_{y_\infty}(Y^3_\infty)$ of
$Y^3_\infty$ at $y$ is homeomorphic to the sphere $S^2$, because
$Y^3_\infty$ is a $3$-manifold. Observe that $\Sigma_x^1(X^2)$ is a
convex subset of $\Sigma^2_{y_\infty}(Y^3_\infty)$. Moreover we see
that $\Sigma_x^1(X^2)$ divides $\Sigma^2_{y_\infty}(Y^3_\infty)$
into exactly two parts:
$$
[\Sigma^2_{y_\infty}(Y^3_\infty) - \Sigma_x^1(X^2)] = \Omega^2_{x,
+} \cup \Omega^2_{x, +}
$$
Since the curvature of $ \Sigma^2_{y_\infty}(Y^3_\infty)$ is greater
than or equal to 1, using Theorem 6.1 of \cite{Per1991} (cf.
\cite{Petr2007}), we obtain that there is a unique unit vector $\xi_\pm
\in \Omega^2_{x, \pm}$ such that
$$
\ell_\pm = \angle_{x}(\xi_\pm, \Sigma_x^1(X^2)) = \max\{
\angle_{x}(w_\pm, \Sigma_x^1(X^2)) \, | \, w_\pm \in \Omega^2_{x,
\pm}\}.
$$
We claim that $\ell_\pm \le \frac{\pi}{2}$. Suppose contrary,
$\ell_\pm > \frac{\pi}{2}$ were true. We derive a contradiction as
follows. Let $\psi: [0, \ell] \to \Sigma^2_{y_\infty}(Y^3_\infty)$
be a length-minimizing geodesic segment of unit speed from
$\Sigma_x^1(X^2)$ of length $\ell \ge \frac{\pi}{2}$ with $\psi(0) =
u \in \Sigma_x^1(X^2)$ and $\dis(\psi(\ell), \Sigma_x^1(X^2)) =
\ell$. We now choose another geodesic segment $\eta: [0, \delta] \to
\Sigma_x^1(X^2)$ be a geodesic segment of unit speed with $\eta(0) =
u$. Since $\curv_{\Sigma^2} \ge 1$, applying the triangle comparison
theorem (cf. \cite{BBI2001}) to the our geodesic hinge at $u$ in
$\Sigma^2_{y_\infty}(Y^3_\infty)$, we see that $\dis_{\Sigma^2}
(\psi(\frac{\pi}{2}), \eta(\delta) ) \le \frac{\pi}{2}$. Thus, for
the point $w = \psi(\frac{\pi}{2})$, there are at least two points
$\{ u, \eta(\delta)\} \subset \Sigma_x^1(X^2)$ with angular distance
$\dis_{\Sigma^2}(w, u) = \dis_{\Sigma^2}(w, \eta(\delta) ) =
\frac{\pi}{2}$. In another words, there were at least two distinct
length-minimizing geodesic segments from $w = \psi(\frac{\pi}{2})$
to $\Sigma_x^1(X^2)$. Hence, $\psi|_{[0, \frac{\pi}{2} +
\varepsilon] }$ is no longer length-minimizing for any $\varepsilon
>0$, a contradiction. It follows that $\ell_\pm \le \frac{\pi}{2}$.
Moreover, the equality $\ell_\pm = \frac{\pi}{2}$ holds if and only
if $\Sigma^2_{y_\infty}(Y^3_\infty)$ is isometric to the two point
spherical suspension of $\Sigma_x^1(X^2)$. In this case,
$T_x(Y^3_\infty) $ is isometric to $T_x(X^2) \times \mathbb R$.

Recall that $X^2 = f^{-1}(a_0)$ is a level set of the Busemann
function. We can write $f(x) = [c - \dis(x, f^{-1}(c))]$ for $x \in
f^{-1}(-\infty, c])$. By the first variational formula (cf.
\cite{BBI2001} page 125), we see that
$$
\ell_\pm \ge \frac{\pi}{2}.
$$
Combining our earlier inequality $\ell_\pm \le \frac{\pi}{2}$, we
see that $\ell_\pm = \frac{\pi}{2}$. Therefore, we conclude that
$T_x(Y^3_\infty)$ is isometric to $T_x(X^2) \times \mathbb R$.
Hence, there is a {\it ``normal line bundle"} over the soul $X^2$.
After passing its double cover if necessary, such a {\it ``normal
line bundle"} of $X^2$ in $Y^3_\infty$ is topologically trivial.
Thus, we assume that $[Y^3_\infty - X^2] = \Omega^3_+ \cup
\Omega^3_-$ has exactly two ends, where we replace $Y^3_\infty$ by
its double cover $\hat Y^3$ if needed. For each end and each $x \in
X^2$, there exists a ray $\sigma_{x, \pm}: (0, \infty) \to
\Omega^3_\pm$ with starting point $x$. One can verify that
$\sigma_{x, -} \cup \sigma_{x, +}$ is a geodesic line in
$Y^3_\infty$ (or in its double cover). By the splitting theorem, we
conclude that $Y^3_\infty$ (or its double cover) is isometric to
$X^2 \times \mathbb R$.

\smallskip

(3) When the soul $N^k$ of $Y^3_\infty$ is a single point $\{
y_\infty\}$, our proof becomes more involved. Let $f(x) = \lim_{t
\to \infty}[ t - \dis(x, \partial B_{Y^3_\infty}(\hat x, t))]$ and
$a_0 = \inf_{x \in Y^3_\infty}\{f(x)\}$ be as above. There are three
possibilities for $\dim[ f^{-1}(a_0)] = 0, 1, 2$.

\smallskip
\noindent {\bf Subcase 3.0.} {\it $\dim[ f^{-1}(a_0)] = 0$ and $A =
f^{-1}(a_0) = \{ y_\infty\}$}.

In this subcase, the space of unit tangent directions
$\Sigma^2_{y_\infty} (Y^3_\infty)$ at $y$ is homeomorphic to the
sphere $S^2$ and its tangent cone $T_{y_\infty}(Y^3_\infty)$ is
homeomorphic to $\mathbb R^3$. Recall that the pointed spaces
$(\lambda Y^3_\infty, y_\infty) $ is convergent to the tangent cone
$(T_{y_\infty}(Y^3_\infty), O)$ as $\lambda \to \infty$, where $O$
is the origin of $T_{y_\infty}(Y^3_\infty)$.

By the pointed version of Perelman's stability theorem (cf. Theorem 7.11
of \cite{Kap2007}), we see that for sufficiently small
$\varepsilon$, $(B_{\frac{1}{\varepsilon} Y^3_\infty}(y_\infty, 1),
y_\infty) $ is homeomorphic to $(B_{T_{y_\infty}(Y^3_\infty)}(O, 1),
O)$. It follows that $B_{Y^3_\infty}(y_\infty, \varepsilon)$ is
homeomorphic to the unit ball $D^3$ for sufficiently small
$\varepsilon >0$, because $Y^3_\infty$ is a $3$-manifold.

We now use Perelman's fibration theorem to complete our proof for
this subcase. It follows from \autoref{prop2.5} that $r_A$ has no
critical value in $(0, \infty)$. Perelman's fibration theorem (our
\autoref{thm1.2} above) implies that there is a fibration structure
$$
(\partial D^3) \to [Y^3_\infty - U_{\frac{\varepsilon}{2}}(A) ]
\stackrel{r_A}{\longrightarrow} (\frac{\varepsilon}{2}, \infty)
$$
It follows that $ [Y^3_\infty - U_{\frac{\varepsilon}{2}}(A) ] $ is
homeomorphic to $S^2 \times (\frac{\varepsilon}{2}, \infty)$ and
that $Y^3_\infty $ is homeomorphic to $D^3 \cup [S^2 \times \mathbb
R] $. Thus, $Y^3_\infty $ is homeomorphic to $\mathbb R^3$, for this
subcase.

\smallskip
\noindent {\bf Subcase 3.1.} {\it $\dim[ f^{-1}(a_0)] = 1$ and $A =
f^{-1}(a_0) = \sigma$ is a geodesic segment.}

It follows from \autoref{prop2.5} that $r_A$ has no critical value
in $(0, \infty)$. For the same reasons as above, it remains
important to verify that $U_\varepsilon(A)$ is homeomorphic to
$D^3$.

Let $\sigma: [0, \ell] \to Y^3$ be as above and $\sigma([0, \ell]) $
be the minimal set of $f$. We denote a $\varepsilon$-neighborhood of
$A$ by $U_\varepsilon(A)$. Let $A_s = \sigma([s, \ell-s])$ for some
$s > 0$. We observe that
$$
U_\varepsilon(A_0) = B_\varepsilon(\sigma(0)) \cup U_\varepsilon(A_{
s} ) \cup B_\varepsilon(\sigma( \ell))
$$
for $s >0$. For the same reason as in Subcase 3.0 above, both
$B_\varepsilon(\sigma(0))$ and $B_\varepsilon(\sigma( \ell))$ are
homeomorphic to $D^3$, because $Y^3_\infty$ is a $3$-manifold. It is
sufficient to show that $U_\varepsilon(A_{ s} )$ is homeomorphic to
a finite cylinder $C = [s, \ell-s] \times D^2$ for sufficient small
$s$.

Let $p= \sigma(0)$. We consider the distance function $r_p(y) =
\dis_{Y^3_\infty}(y, p)$. We observe that the distance function has
no critical point on geodesic sub-segment $\sigma([s, \ell-s])$. A
result of Petrunin (cf. \cite{Petr2007} page 142) asserts that if
$x_n \to x$ as $n \to \infty$, then $\liminf_{n\to \infty} |
\nabla_{x_n} r_p | \ge | \nabla_{x} r_p|$. Hence, there exists a
sufficiently small $\varepsilon >0$ such that $r_p$ has no critical
point in $U_{\varepsilon}(A_{ s} )$. For the same reason as in
Subcase 3.0, we can apply Perelman's fibration theorem to our case:
$$
D^2 \to U_{\varepsilon}(A_{ s} ) \stackrel{r_p}{\longrightarrow} (s,
\ell-s),
$$
where we used the fact that $ [\partial B_\varepsilon(\sigma(0))]
\cap U_{\varepsilon}(A_{ s/2} )$ is homeomorphic to $D^2$. It
follows that $U_\varepsilon(A_{ s} )$ is homeomorphic to a finite
cylinder $C = (s, \ell-s) \times D^2$. Therefore, $U_\varepsilon(A_{
0})$ is homeomorphic to $D^3$. It follows that $Y^3_\infty \sim [
D^3 \cup (S^2 \times \mathbb R)]$ is homeomorphic to $\mathbb R^3$.

\smallskip
\noindent {\bf Subcase 3.2.} {\it $\dim[ f^{-1}(a_0)] = 2$ and $A =
f^{-1}(a_0) = \Omega^2_0 \sim D^2$ is a totally convex surface with
boundary.}

For the same reason as the two subcases above, it is sufficient to
establish that $ U_\varepsilon(A)$ is homeomorphic to the unit
$3$-ball $D^3$:
$$
U_\varepsilon(A) \sim D^3.
$$

Let $A_s = \{ x \in \Omega^2_0 \, | \, \dis(x, \partial \Omega^2_0)
\ge s \}$. By our discussion in Case 2 above, we see that for each
interior point $x \in \Omega^2_0$, there is a unique {\it normal
line} orthogonal to $\Omega^2_0$ at $x$. Thus, the interior
$\interior(\Omega^2_0)$ has a normal line bundle $\mathbb R$.
Because $\interior(\Omega^2_0)$ is contractible to a soul point
$y_0$, any line bundle over $\interior(\Omega^2_0)$ is topologically
trivial.

In this subcase, our technical goals are to show the following:

\smallskip
\noindent (3.2a) $U_\varepsilon(A_s)$ is homeomorphic to $D^2 \times
(-\varepsilon, \varepsilon)$;

\smallskip
\noindent (3.2b) $U_\varepsilon(\partial A_0)$ is homeomorphic to a
solid tori $(\partial D^2) \times D^2_\varepsilon = S^1 \times
D^2_\varepsilon$.

\smallskip

To establish (3.2a), we use a theorem of Perelman (cf. Theorem 6.1
of \cite{Per1991}) to show that there is a product metric on a
subset $U_\varepsilon(A_s)$ of $Y^3_\infty$. Inspired by Perelman,
we consider the distance function $r_{A_0} (y) = \dis_{Y^3_\infty}
(y, A_0)$. Since $Y^3_\infty$ has non-negative curvature and
$\interior(A_s)$ is {\it weakly concave} towards its complement $[
U_{s/4}(A_s) - A_s]$, Perelman observed that $r_{A_0}$ is {\it
concave} on $[ U_{s/4}(A_s) - A_s]$, (see the proof of Theorem 6.1
in \cite{Per1991}, \cite{Petr2007} or \cite{CMD2009}). We already
showed that for each interior point $x \in A_0$, there is a unique
{\it normal line} orthogonal to $A_0$ at $x$. With extra efforts, we
can show that, for each interior point $x \in \interior(A_0)$ and
each unit normal direction $\xi_\pm \perp_x (\interior A_0)$, there
is a unique ray $\sigma_{x, \pm}: [0, \infty) \to Y^3_\infty$ with
$\sigma_{x, \pm}(0) = x$ and $\sigma_{x, \pm}'(0) = \xi_\pm$.
Moreover, we have
$$
f(\sigma_{x, \pm}(t)) = a_0 + t.
$$
Therefore each $y \in [ U_{s/4}(A_s) - A_s]$ with $s>0$, we have
$$
\nabla (-f) |_y = - \nabla r_{A_0} |_y.
$$
Hence, our Busemann function $f$ is both convex and concave on the
subset $[ U_{s/4}(A_s) - A_s]$. Thus, for any geodesic segment
$\varphi: [a, b] \to [ U_{s/4}(A_s) - A_s]$, the function
$f(\varphi(t))$ is a linear function in $t$. Using the fact that
$f(\varphi(t))$ is a linear function in $t$ and the sharp version of
triangle comparison theorem (cf. \cite{BGP1992}), we can show that
there is a sub-domain $V$ of $Y^3_\infty$ such that the metric of
$Y^3_\infty$ on $V$ splits isometrically as
$$
V = \interior(A_0) \times \mathbb R.
$$
Since $\interior(A_0)$ is homeomorphic to $D^2$, we conclude that
$U_\varepsilon(A_s)$ is homeomorphic to $D^2 \times (-\varepsilon,
\varepsilon)$ (compare with the proof of \autoref{thm5.4} below).
Hence, our Assertion (3.2a) holds.

It remains to verify (3.2b). We consider the doubling surface
$\dbl(A_0) = A_0 \cup_{\partial A_0} A_0$. It follows from a result
of Petrunin that $\dbl(A_0)$ has non-negative curvature. By
\autoref{prop1.7}, we see that the essential singularities (extremal
points) in $\dbl(A_0)$ are isolated. Thus, there are only finitely
many points $\{ x_1, \cdots, x_k\}$ on $\partial A_0$ such that
$$
\diam[\Sigma_{x_j}(A_0) ] \le \frac{\pi}{2}
$$
for $j =1,\cdots, k$. We
can divide our boundary curve $\partial A_0$ into $k$-many arcs, say
$[
\partial A_0 - \{ x_1, \cdots, x_k\}] = \cup \gamma_j$. Using a similar
argument as in Subcase 3.1, we can show that, for each $\gamma_j$,
its $\varepsilon$-neighborhood $U_\varepsilon(\gamma_j)$ is
homeomorphic to a finite cylinder $C_j \sim [D^2 \times (0,
\ell_j)]$. Since $Y^3_\infty$ is a $3$-manifold, by the proof of
\autoref{prop1.15}, we know that $B_{Y^3_\infty}(x_j, \varepsilon)$
is homeomorphic to $D^3$. Consequently, we have
$$
U_\varepsilon(\partial A_0) = [\cup C_j] \bigcup [\cup
B_{Y^3_\infty}(x_j, \varepsilon)],
$$
which is homeomorphic to a solid tori $D^2 \times S^1$. This
completes our proof of the assertion that $ U_\varepsilon(A_0) \sim
\{U_\varepsilon(\partial A_0) \cup [ A_0 \times ( - \varepsilon,
\varepsilon)] \} $ is homeomorphic to $D^3$. Therefore, $Y^3_\infty
\sim [D^3 \cup (S^2 \times \mathbb R)]$ is homeomorphic to $\mathbb
R^3$.

We now finished the proof of our soul theorem for all cases.
\end{proof}

\subsection{Applications of the soul theory to proof of \autoref{thm2.1}.}\label{section2.3} \

\smallskip

\smallskip

Using \autoref{thm2.11}, we can complete the proof of
\autoref{thm2.1}.

\begin{proof}[Proof of \autoref{thm2.1}]
Let $\lambda_\alpha$ and $\tilde g^\alpha$ be defined by
\eqref{eq:2.1}. We may assume that $\{((M^3_{\alpha},\tilde
g^{\alpha}),x_{\alpha})\}$ is convergent to a pointed Alexandrov
space $(Y^s_{\infty}, y_{\infty})$ of non-negative curvature, by
replacing $x_\alpha$ with $x'_\alpha$ in \autoref{prop2.2} if
needed. By \autoref{prop2.3}, we see that the limiting space
$Y^s_\infty$ is a non-compact and complete space of $\dim Y^3_\infty
=3$. Furthermore, $Y^3_\infty$ has no boundary. By Perelman's
stability theorem (cf. \cite{Kap2007}), we see that the limit space
$Y^3_\infty$ is a topological $3$-manifold.

By \autoref{thm1.17}, we see that $\partial B_{(M^3_\alpha, \tilde
g^{\alpha})} (x_\alpha, r)$ is homeomorphic to a quotient of the
$2$-torus $T^2$. The notion of ideal boundary $Y^3_\infty(\infty)$
of $Y^3_\infty$ can be found in \cite{Shio94}. In our case, the
ideal boundary $Y^3_\infty(\infty)$ at infinity of $Y^3_\infty$ is
homeomorphic to a circle $\partial B_{X^2}(x_\infty, r)$. We will
verify that $N^k \sim S^1$ as follows. Let $U_\varepsilon(N^k)$ be a
$\varepsilon$-tubular neighborhood of the soul $N^k$ in $Y^3_\infty
$. By Perelman's stability theorem and our assumption that
$M^3_\alpha $ is oriented, we observe that, for sufficiently large
$\alpha$, the boundary $\partial U_r(N^k)$ is homoeomorphic to
$\partial B_{(M^3_\alpha, \tilde g^{\alpha} )}(x_\alpha, r') \sim
T^2$. Thus, the soul $N^k$ of $Y^3$ must be a circle $S^1$. In this
case, it follows from \autoref{thm2.11} that the metric on $Y^3$ (or
on its universal cover) splits. Therefore, it follows from
Perelman's stability theorem (cf. \cite{Kap2007}) that
$B_{(M^3_\alpha, \tilde g^{\alpha} )}(x_\alpha, r')$ is homeomorphic
to a solid tori, which is foliated by orbits of a free circle action
for sufficiently large $\alpha$. We discuss more about gluing and
perturbing our local circle actions in \autoref{section6}.
\end{proof}
We will use \autoref{thm2.1}-\ref{thm2.11} to derive more refined
results for collapsing $3$-manifolds with curvature $\ge -1$ in
upcoming sections.

\section{Admissible decompositions for collapsed
$3$-manifolds}\label{section3}

Let $S^2(\varepsilon)$ be a round sphere of constant curvature
$\frac{1} {\varepsilon^2}$. It is clear that $S^2(\varepsilon)\times
[a,b]$ is convergent to $[a,b]$ with non-negative curvature as
$\varepsilon\to 0$. The product space $W_{\varepsilon} =
S^2(\varepsilon) \times [a,b]$ is not a graph manifold. However, if
$W_{\varepsilon}$ is contained in the interior of collapsed
$3$-manifold $M^3_{\alpha}$ with boundary, then for topological
reasons, $W_{\alpha}$ still has a chance to become a part of
graph-manifold $M^3_{\alpha}$.

Let us now use the language of Cheeger-Gromov's $F$-structure theory
to describe $3$-dimensional graph-manifold. It is known that a
$3$-manifold $M^3$ is a graph-manifold if and only if $M^3$ admits
an $F$-structure of positive rank, which we now describe.

An F-structure, $\mathscr{F}$, is a topological structure which
extends the notion of torus action on a manifold, (see \cite{CG1986}
and \cite{CG1990}). In fact, the more significant concept is that of
atlas (of charts) for an F-structure.

An atlas for an F-structure on a manifold $M^n$ is defined by a
collection of triples $\{(U_i, V_i, T^{k_i})\}$, called charts,
where $\{U_i\}$ is an open cover of $M^n$ and the torus, $T^{k_i}$,
acts effectively on a finite normal covering, $\pi_i: V_i\to U_i$,
such that the following conditions hold:
\begin{enumerate}[(3.1)]
\item There is a homomorphism, $\rho_i: \Gamma_i=\pi_1(U_i)\to
Aut(T^{k_i})$, such that the action of $T^{k_i}$ extends to an
action of the semi-direct product $T^{k_i}\ltimes _{\rho_i}
\Gamma_i$, where $\pi_1(U)$ is the fundamental group of $U$;

\item If $U_{i_1}\cap U_{i_2}\ne\varnothing$, then $U_{i_1}\cap
U_{i_2}$ is connected. If $k_{i_1} \le k_{i_2}$, then on a suitable
finite covering of $U_{i_1}\cap U_{i_2}$, their lifted tori-actions
commute after appropriate re-parametrization.
\end{enumerate}

The compatibility condition (3.2) on lifted actions implies that
$M^n$ decomposes as a disjoint union of orbits, $\mathcal {O}$, each
of which carries a natural flat affine structure. The orbit
containing $x\in M^n$ is denoted by $\mathcal {O}_x$. The dimension
of an orbit of minimal dimension is called the rank of the
structure.

\begin{prop}[\cite{CG1986}, \cite{CG1990}]\label{prop3.1}
A $3$-dimensional manifold $M^3$ with possible non-empty boundary is
a graph-manifold if and only if $M^3$ admits an F-structure of
positive rank.
\end{prop}

For $3$-dimensional manifolds, we will see that 7 out of 8
geometries admits F-structure. Therefore, there seven types of
locally homogeneous spaces are graph-manifolds.

\begin{example}\label{ex3.2}
Let $M^3$ be a closed locally homogeneous space of dimension 3, such
that its universal covering spaces $\tilde M^3$ is isometric to
seven geometries: $\mathbb R^3, S^3, \mathbb H^2\times\mathbb R,
S^2\times \mathbb R, \tilde{SL}(2,\mathbb R), Nil$ and $Sol$. Then
$M^3$ admits an F-structure and hence it is a graph-manifold.

Let us elaborate this issue in detail as follows.
\begin{enumerate}[{\rm (i)}]
\item
If $M^3=\mathbb R^3/\Gamma$ is a flat $3$-manifold, then it is
covered by $3$-dimensional torus. Hence it is a graph-manifold.

\item
If $M^3=S^3/\Gamma$ is a lens space, then its universal cover $S^3$
admits the classical Hopf fibration:
$$S^1\to S^3\to S^2.$$
It follows that $M^3$ is a graph-manifold.

\item If $M^3= (\mathbb H^2\times\mathbb R)/\Gamma$ is
a closed $3$-manifold, then a theorem of Eberlein implies that a
finite normal cover $\hat M^3$ of $M^3$ is diffeomorphic to
$N^2\times S^1$, where $N^2$ is a closed surface of genus $\ge 1$,
(see Proposition 5.11 of \cite{CCR2001}, \cite{CCR2004}).

\item
If $M^3= (S^2\times\mathbb R)/\Gamma$, then a finite cover is
isometric to $S^2\times S^1$. Clearly, $M^3$ is a graph-manifold. We
should point out that a quotient space $(S^2\times S^1)/\mathbb Z_2$
may be homeomorphic to $\mathbb {RP}^3 \# \mathbb {RP}^3$.

\item
If $M^3= \tilde{SL}(2,\mathbb R)/\Gamma$, then a finite cover $\hat
M^3$ of $M^3$ is diffeomorphic to the unit tangent bundle of a
closed surface $N_k^2$ of genus $k\ge 2$. Thus, we may assume that
$\hat M^3=SN^2_k=\{(x,\vec v)| x\in N_k^2, \vec v\in T_x(N_k^2),
|\vec v|=1\}$. Clearly, there is a circle fibration
$$S^1\to \hat M^3\to N_k^2.$$
It follows that $M^3$ is a graph-manifold.

\item If $M^3= Nil/\Gamma$, then the universal cover
$$\tilde
M^3=Nil=\left\{\left(\left.
            \begin{array}{ccc}
              1 & x & z \\
              0 & 1 & y \\
              0 & 0 & 1 \\
            \end{array}
          \right) \right| x,y,z \in \mathbb R\right\}
$$

is a $3$-dimensional Heisenberg group. Let

$$\hat \Gamma=\left\{\left(\left.
                                \begin{array}{ccc}
                                  1 & m & k \\
                                  0 & 1 & n \\
                                  0 & 0 & 1 \\
                                \end{array}
                              \right)
          \right|k,m,n\in \mathbb Z\right\}
$$
be the integer lattice
group of $Nil$. A finite cover $\hat M^3$ of $M^3$ is a circle
bundle over a $2$-torus. Therefore $M^3$ is a graph-manifold, which
can be a diameter-collapsing manifold.

\item If $M^3= Sol/\Gamma$, then $M^3$ is foliated
by tori, M\"{o}bius bands or Klein bottles, which is a
graph-manifold.
\end{enumerate}
\end{example}

Let us consider a graph-manifold which is not a compact quotient of
a homogeneous space.
\begin{example}
Let $N^2_i$ be a surface of genus $\ge 2$ and with a boundary
circle, for $i=1,2$. Clearly , $\partial (N^2_i\times
S^1)=S^1\times S^1$. We glue $N^2_1\times S^1$ to $S^1\times
N^2_2$ along their boundaries with $S^1$-factor switched. The
resulting manifold $M^3=(N^2_1\times S^1)\cup(S^1\times
N^2_2)$ does not admit a global circle fibration, but $M^3$ is
a graph-manifold.
\end{example}

As we pointed out above, $S^2\times [a,b]$ can be collapsed to an
interval $[a,b]$ with non-negative curvature. Suppose that $W$ is a
portion of collapsed $3$-manifold $M^3$ such that $W$ is
diffeomorphic to $S^2\times[a,b]$. We need to glue extra solid
handles to $W$ so that our collapsed $3$-manifold under
consideration becomes a graph-manifold. For this purpose, we divided
$S^2\times [a,b]$ into three parts. In fact, $S^2$ with two disks
removed, $S^2-(D^2_1\sqcup D^2_2)$, is diffeomorphic to an annulus
$A$. Thus, $S^2\times [a,b]$ has a decomposition
$$
S^2\times [a,b]=\big(D^2_1\times [a,b]\big) \sqcup (D^2_2\times
[a,b]) \sqcup (A\times [a,b]).
$$

The product space $A\times [a,b]$ clearly admits a free circle
action, and hence is a graph-manifold. For solid cylinder part
$D^2_i\times [a,b]$, if one can glue two solid cylinders together,
then one might end up with a solid torus $D^2\times S^1$ which is
again a graph-manifold.

We will decompose a collapsed $3$-manifold $M^3_{\alpha}$ with
curvature $\ge -1$ into four major parts according to the dimension
$k$ of limiting set $X^k$:
$$
M^3_{\alpha}=V_{\alpha,X^0}\cup V_{\alpha, \interior (X^1)}\cup
V_{\alpha,\interior (X^2)}\cup W_{\alpha}
$$
where $\interior (X^s)$ denotes the interior of the space $X^s$.

The portion $V_{\alpha,X^0}$ of $M^3_{\alpha}$ consists of union of
closed, connected components of $M^3_{\alpha}$ which admit
Riemannian metric of non-negative sectional curvature.

\begin{prop}\label{prop3.4}
Let $V_{\alpha,X^0}$ be a union of one of following:
\begin{enumerate}[{\rm (1)}]
\item a spherical $3$-dimensional space form;
\item a manifold double covered by $S^2\times S^1$;
\item a closed flat $3$-manifold.
\end{enumerate}
Then $V_{\alpha,X^0}$ can be collapsed to a $0$-dimensional manifold
with non-negative curvature. Moreover, $V_{\alpha,X^0}$ is a
graph-manifold.
\end{prop}

We denote the regular part of $X^2$ by $X^2_{reg}$. Let us now
recall a reduction for the proof of Perelman's collapsing theorem
due to Morgan-Tian. Earlier related work on 3-dimensional collapsing
theory was done by Xiaochun Rong in his thesis, (cf. \cite{R1993}).

\begin{theorem}[Morgan-Tian \cite{MT2008}, Compare \cite{R1993}]\label{thm3.5}
Let $\{M^3_{\alpha}\}$ be a sequence of compact $3$-manifolds
satisfying the hypothesis of Theorem 0.1' and $V_{\alpha,X^0}$ be as
in \autoref{prop3.4} above. If, for sufficiently large $\alpha$,
there exist compact, co-dimension 0 submanifolds
$V_{\alpha,X^1}\subset M^3_{\alpha}$ and $V_{\alpha, X^2_{reg}}
\subset M^3_{\alpha}$ with $\partial M^3_{\alpha}\subset
V_{\alpha,X^1}$ satisfying six conditions listed below, then
\autoref{thm0.1} holds, where six conditions are:
\begin{enumerate}[{\rm (1)}]
\item Each connected component of $V_{\alpha,X^1}$ is diffeomorphic
to the following.
\begin{enumerate}[{\rm (a)}]
\item a $T^2$-bundle over $S^1$ or a union of two twisted
I-bundle over the Klein bottle along their common boundary;
\item $T^2\times I$ or $S^2\times I$, where $I=[a,b]$ is a closed
interval;

\item a compact $3$-ball or the complement of an open $3$-ball in
$\mathbb{RP}^3$ which is homeomorphic to $\mathbb {RP}^2 \ltimes
[0, \frac 12]$;
\item
a twisted I-bundle over the Klein bottle, or a solid torus.
\end{enumerate}
In particular, every boundary component of $V_{\alpha,X^1}$ is
either a $2$-sphere or a $2$-torus.

\item
$V_{\alpha,X^1} \cap V_{\alpha,X^2_{reg}}=(\partial
V_{\alpha,X^1})\cap(\partial V_{\alpha,X^2_{reg}})$;

\item If $N^2_0$ is a $2$-torus component of $\partial
V_{\alpha,X^1}$, then $N^2_0 \subset \partial
V_{\alpha,X^2_{reg}}$ if and only if $N_0^2$ is not boundary of
$\partial M^3_{\alpha}$;

\item If $N^2_0$ is a $2$-sphere component of $\partial
V_{\alpha,X^1}$, then $N^2_0 \cap \partial
V_{\alpha,X^2_{reg}}$ is diffeomorphic to an annulus;

\item $V_{\alpha,X^2_{reg}}$ is the total space of a locally trivial
$S^1$-bundle and the intersection $V_{\alpha,X^1}\cap
V_{\alpha,X^2_{reg}}$ is saturated under this fibration;

\item The complement $W_{\alpha}=[M^3_{\alpha}
-(V_{\alpha,X^0}\cup V_{\alpha,X^1}\cup V_{\alpha,X^2_{reg}})]$
is a disjoint union of solid tori and solid cylinders. The boundary
of each solid torus is a boundary component of
$V_{\alpha,X^2_{reg}}$, and each solid cylinder $D^2\times I$ in
$W_\alpha$ meets $V_{\alpha,X^1}$ exactly in $D^2\times \partial I$.
\end{enumerate}
\end{theorem}
\begin{proof}
(\cite{MT2008}) The proof of \autoref{thm3.5} is purely topological,
which has noting to do with the collapsing theory.

Morgan and Tian \cite{MT2008} first verified \autoref{thm3.5} for
special cases under additional assumption on $V_{\alpha,X^1}$:
\begin{enumerate}[(i)]
\item $V_{\alpha,X^1}$ has no closed components;
\item Each $2$-sphere component of $\partial V_{\alpha,X^1}$ bounds
a $3$-ball component of $V_{\alpha,X^1}$;
\item Each $2$-torus component of $\partial V_{\alpha,X^1}$ that is
compressible in $M^3_{\alpha}$ bounds a solid torus component of
$V_{\alpha,X^1}$.
\end{enumerate}

The general case can be reduced to a special case by a purely
topological argument.
\end{proof}

\begin{definition}\label{def3.6}
If a collapsed $3$-manifold $M^3_{\alpha}$ has a decomposition
$$M^3_{\alpha}= V_{\alpha,X^0}\cup V_{\alpha,X^1}\cup
V_{\alpha,X^2_{reg}}\cup W_{\alpha}$$ satisfying six properties
listed in \autoref{thm3.5} and if $V_{\alpha,X^0}$ is a union of
closed $3$-manifolds which admit smooth Riemannian metrics of
non-negative sectional curvature, then such a decomposition is
called an admissible decomposition of $M^3_{\alpha}$.
\end{definition}

In \autoref{section1}-\ref{section2}, we already discussed the part
$V_{\alpha, \interior (X^2)}$ and a portion of $W_{\alpha}$. In next
section, we discuss the collapsing part $V_{\alpha, \interior
(X^1)}$ with spherical or toral fibers for our $3$-manifold
$M^3_{\alpha}$, where $\interior (X^s)$ is the interior of $X^s$.

\section{Collapsing with spherical and toral fibers }\label{section4}

In this section, we discuss the case when a sequence of metric balls
$\{(B_{(M^3_{\alpha}, \tilde g^{\alpha}_{ij})}(x_\alpha, r),
x_{\alpha})\}$ collapse to $1$-dimensional space $(X^1,x_{\infty})$.
There are only two choices of $X^1$, either diffeomorphic to a
circle or an interval $[0,l]$.

By Perelman's fibration theorem \cite{Per1994} or Yamaguchi's
fibration theorem, we can find an open neighborhood $U_{\alpha}$ of
$x_{\alpha}$ such that, for sufficiently large $\alpha$, there is a
fibration
$$ N^2_{\alpha}\to U_{\alpha}\to \interior (X^1)$$
where $\interior (X^1)$ is isometric to a circle $S^1$ or an open
interval $(0,l)$.

We will use the soul theory (e.g., \autoref{thm2.11}) to verify that
a finite cover of the collapsing fiber $N^2_\alpha$ must be
homeomorphic to either a $2$-sphere $S^2$ or a $2$-dimensional torus
$T^2$, (see \autoref{fig:4.1} in \S 0)

Let us begin with two examples of collapsing $3$-manifold with toral
fibers

\begin{example}\label{ex4.1}
Let $M^3=\mathbb H^3/\Gamma$ be an oriented and non-compact quotient
of $\mathbb H^3$ such that $M^3$ has finite volume and $M^3$ has
exactly one end. Suppose that $\sigma: [0,\infty)\to \mathbb
H^3/\Gamma$ be a geodesic ray. We consider the corresponding
Busemann function $h_{\sigma}(x)=\lim_{t\to
+\infty}[t-\dis(x,\sigma(t))]$. For sufficiently large $c$, the
sup-level subset $V_{c, X^1}=h^{-1}_{\sigma}([c,+\infty))$ has
special properties. It is well known that, in this case, the cusp
end $V_{c,X^1}$ is diffeomorphic to $T^2\times [c,+\infty)$. Of
course, the component $V_{c, X^1}\cong T^2\times [c,+\infty)$ admits
a collapsing family of metric $\{g_{\varepsilon}\}$, such that
$(V_{c, X^1},g_{\varepsilon})$ is convergent to half line
$[c,+\infty)$. \qed
\end{example}

We would like to point out that $2$-dimensional collapsing fibers
can be collapsed at two different speeds.

\begin{example}
Let $M^3_{\varepsilon}=(\mathbb R/\varepsilon\mathbb
Z)\times(\mathbb R/\varepsilon^2\mathbb Z)\times[0,1]$ be the
product of rectangle torus $T_{\varepsilon,\varepsilon^2}$ and an
interval. Let us fix a parametrization of $M^3_{1}=\{(e^{2\pi
si},e^{2\pi t i}, u)|u\in [0,1], s,t\in \mathbb R\}$ and
$g_{\varepsilon}=\varepsilon^2ds^2+\varepsilon^4dt^2+du^2$. Then the
re-scaled pointed spaces $\{
((M^3_{1},\frac{1}{\varepsilon^2}g_{\varepsilon}),(1,1,\frac12))\}$
are convergent to the limiting space $(Y^2_{\infty}, y_{\infty})$,
where $Y^2_{\infty}$ is isometric to $S^1\times (-\infty,+\infty)$.
\end{example}

Similarly, when the collapsing fiber is homeomorphic to a 2-sphere
$S^2$, the collapsing speeds could be different along longitudes and
latitudes. We may assume that latitudes shrink at speed
$\varepsilon^2$ and longitudes shrink at speed $\varepsilon$. For
the same reason, after re-scaling, the limit space $Y^2_\infty$
could be isometric to $ [0,1] \times (-\infty, +\infty)$. Thus, the
non-compact limiting space $ Y^2_\infty$ could have boundary. \qed

\medskip

Let us return to the proof of Perelman's collapsing theorem.
According to the second condition of \autoref{thm0.1}, we consider a
boundary component $N^2_{\alpha, X^1} \subset
\partial M^2_{\alpha} $, where the diameter of $N^2_{\alpha,
X^1}$ is at most $\omega_\alpha \to 0$ as $\alpha \to \infty$.
Moreover, there exists a topologically trivial collar $\hat
V_{\alpha, X^1}$ of length one and sectional curvatures of
$M^3_{\alpha}$ are between $(-\frac14-\varepsilon)$ and
$(-\frac14+\varepsilon)$. In this case, we have a trivial fibration:
\begin{equation}\label{eq4.1}
T^2_{\alpha}\to \hat V_{\alpha, X^1}\xrightarrow{\pi_{\alpha}}[0,1]
\end{equation}
such that the diameter of each fiber $\pi^{-1}_{\alpha}(t)$ is at
most $[\frac{e+e^{-1}}{2}w_{\alpha}]$ by standard comparison
theorem. As $\alpha \to +\infty$, the sequence $\{\hat V_{\alpha,
X^1}\}$ converge to an interval $X^1=[0,1]$.

We are ready to work on the main result of this subsection.
\begin{theorem}\label{thm4.3}
Let
$\{((M^3_{\alpha},\rho^{-2}_{\alpha}g^{\alpha}_{ij}),x_{\alpha})\}$
be as in Theorem 0.1'. Suppose that an $1$-dimensional space $X^1$
is contained in the $1$-dimensional limiting space and $x_{\alpha}\to x_{\infty}$ is
an interior point of $X^1$. Then, for sufficiently large $\alpha$,
there exists a sequence of subsets $\hat V_{\alpha, X^1}\subset
M^3_{\alpha}$ such that $\hat V_{\alpha, X^1}$ is fibering over
$\interior (X^1)$ with spherical or toral fibers.
$$N^2_\alpha\to \hat V_{\alpha, \interior (X^1)}\to \interior (X^1)$$
where $N^2_\alpha$ is homeomorphic to a quotient of a $2$-sphere
$S^2$ or a $2$-torus $T^2$.

When $M^3_\alpha$ is oriented, $N^2_\alpha$ is either $S^2$ or
$T^2$.
\end{theorem}
\begin{proof}
As we pointed out above, since $ \interior (X^1)$ is a
$1$-dimensional space, there exists a fibration
$$N^2_\alpha\to \hat V_{\alpha, \interior (X^1)}\to \interior (X^1)$$
for sufficiently large $\alpha$. It remains to verify that the fiber
is homeomorphic to $S^2, \mathbb{RP}^2, T^2$ or Klein bottle
$T^2/\mathbb Z_2$. For this purpose, we use soul theory for possibly
singular space $Y^k_{\infty}$ with non-negative curvature.

By our discussion, $B_{\rho^{-2}_{\alpha}g^{\alpha}}(x_{\alpha},1)$
is homeomorphic to $N^2_\alpha\times (0,l)$, where $N^2_\alpha$ is a
closed $2$-dimensional manifold. Thus, the distance function
$r_{x_{\alpha}}(x)=\dis_{\rho^{-2}_{\alpha}g^{\alpha}}(x_{\alpha},x)$
has at least one critical point $y_{\alpha}\ne x_{\alpha}$ in
$B_{\rho^{-2}_ {\alpha}g^{\alpha}}(x_{\alpha},1)$, because
$B_{g^{\alpha}}(x_{\alpha},\rho_{\alpha})$ is not contractible. Let
$\lambda_\alpha$ be
$$
\max\{\dis_{\rho^{-2}_{\alpha}g^{\alpha}}(x_{\alpha},y_{\alpha})|
y_{\alpha}\ne x_{\alpha} \text { is a critical point of }
r_{x_{\alpha}} \text { in } B_{\rho^{-2}_ {\alpha}
g^{\alpha}}(x_{\alpha},1)\}.
$$
We claim that $0<\lambda_{\alpha}<1$ and $\lambda_{\alpha}\to 0$ as
$\alpha\to +\infty$. To verify this assertion, we observe that the
distance functions $\{r_{x_\alpha}\}$ are convergent to
$r_{x_{\infty}} : X^1\to \mathbb R$. By Perelman's convergent
theorem, the trajectory of gradient semi-flow
$$
\frac{d^+\varphi}{dt}= \frac{\nabla r_{x_\alpha}}{|\nabla
r_{x_\alpha}|^2}
$$
is convergent to the trajectory in the limit space $X^1$:
\begin{equation}\label{eq4.2}
\frac{d^+\varphi}{dt}= \nabla r_{x_\infty}
\end{equation}
(see \cite{Petr2007} or \cite{KPT2009}).

Clearly, $r_{x_\infty}: X^1\to \mathbb R$ has no critical value in
$(0,\delta_{\infty})$ for $\delta_{\infty}>0$. Thus, for
sufficiently large $\alpha$, the distance function $r_{x_{\alpha}}$
has no critical value in $(\lambda_{\alpha},
\delta_{\infty}-\varepsilon_{\alpha})$, where $\lambda_{\alpha}\to
0$ and $\varepsilon_{\alpha}\to 0$ as $\alpha\to +\infty$.

Let us now re-scale our metrics
$\{\rho^{-2}_{\alpha}g^{\alpha}_{ij}\}$ by $\{\lambda^{-2}_ {\alpha}
\}$ again. Suppose $\tilde g^{\alpha}_{ij}= \frac{1}
{\lambda^2_{\alpha} \rho^2_{\alpha}}g^{\alpha}_{ij}$. Then a
subsequence of the sequence of the pointed spaces
$\{(M^3_{\alpha},\tilde g^{\alpha}_{ij} ), x_{\alpha}\}$ will
converge to $(Y^k_{\infty}, y_{\infty})$. The curvature of
$Y^k_{\infty}$ is greater than or equal to $0$, because
$\curv_{\frac{1} {\lambda^2_{\alpha} \rho^2_{\alpha}}
g^{\alpha}_{ij}} \ge -{\lambda^2_{\alpha}}\to 0$ as $\alpha\to
+\infty$.

Since the distance $r_{y_{\infty}}(y)=\dis(y, y_{\infty})$ has a
critical point $z_{\infty}\ne y_{\infty}$ with $\dis(z_{\infty},
y_{\infty})\le 1$. Thus $\dim(Y^k_{\infty})>1$.

When $\dim(Y^k_{\infty})=3$, we observe that $Y^3_{\infty}$ has
exactly two ends in our case. Thus $Y^3_{\infty}$ admits a line and
hence its metric splits. A soul $N^2_\infty$ of $Y^3_{\infty}$ must
be of $2$-dimensional. Thus $Y^3_{\infty}$ is isometric to
$N^2_\infty \times \mathbb R$. It follows that the soul $N^2_\infty$
of $Y^3_{\infty}$ has non-negative curvature. By Perelman's
stability theorem, $N^2_\alpha$ is homeomorphic to $N^2_\infty$ for
sufficiently large $\alpha$. A closed possibly singular surface
$N^2_\infty$ of non-negative curvature has been classified in
\autoref{section2}: $S^2, \mathbb{RP}^2, T^2$ or Klein bottle
$T^2/\mathbb Z_2$.

Let us consider the case of $\dim (Y^k_{\infty})=2$. We may assume
that the limiting space $Y^2_{\infty}$ has exactly two ends. When
$Y^2_\infty$ has no boundary, then the limit space is isometric to
$S^1\times \mathbb R$, because $Y^2_{\infty}$ has exactly two ends.
By our discussion in \S 1-2, we have fibration structure
\begin{equation}\label{eq4.3}
S^1\to M^3_{\alpha}\to Y^2_{\infty}
\end{equation}
for sufficiently large $\alpha$.

When $Y^2_\infty$ has non-empty boundary (i.e., $\partial
Y^2_{\infty} \neq \varnothing$), there are two subcases. If
$y_\infty \in \interior(Y^2_\infty)$, by our discussion in
\autoref{section2}, we still have $ N^2_\alpha \sim T^2/\Gamma$. If
$y_\infty \in \partial Y^2_\infty$, using \autoref{thm5.4} below, we
see that $N^2_\alpha \sim [\partial B_{M^3_\alpha}(x_\alpha, r)]
\sim S^2$. This completes the proof of our theorem for all cases.
\end{proof}

In next section, we will discuss the end points or $\partial X^1$
when $X^1\cong [0, l]$ is an interval. In addition, we also discuss
the $2$-dimensional boundary $\partial X^2$ when $X^2$ is a surface
with convex boundary.

\section{Collapsing to boundary points or endpoints of the limiting
space}\label{section5}

Suppose that a sequence of $3$-manifolds is collapsing to a lower
dimensional space $X^s$. In previous sections we showed that there
is (possibly singular) fibration
$$
N^{3-s}_\alpha \to B_{M^3_{\alpha}} (x_\alpha, r) \xrightarrow{G-H}
\interior (X^s).
$$
In this section we will consider the points on the boundary of $X^s$,
we will divided our discussion into two cases: namely (1) when $s=1$
and $X^1$ is a closed interval; and (2) when $s=2$ and $X^2$ is a
surface with boundary.

\subsection{Collapsing to a surface with boundary} \

\smallskip

Since $X^2$ is a topological manifold with boundary, without loss of
generality, we can assume that $\partial X^2=S^1$. First we provide
two examples to demonstrate that the collapsing could occur in many
different ways.

\begin{example}\label{ex 5.1}
Let $M^3_1 =D^2 \times S^1$ be a solid torus, where $D^2 = \{ (x, y)
| x^2 + y^2 < 1 \}$. We will construct a family of metrics
$g_\varepsilon$ on the disk so that $D^2_\varepsilon = (D^2,
g_\varepsilon) $ is converging to the interval $[0, 1)$. It follows
that if $M^3_{\frac{1}{\varepsilon}} = D^2_\varepsilon \times S^1$
then the sequence of $3$-spaces is converging to a finite cylinder:
i.e., $M^3_{\frac{1}{\varepsilon}} \to \big([0, 1) \times S^1 \big)$
as $\varepsilon \to 0$.

For this purpose, we let
$$
D^2_\varepsilon=\{(x, y, z)\in \mathbb R^3|z^2=\tan ( \frac{1}{
\varepsilon })[ x^2 + y^2],z\ge 0, x^2 + y^2 + z^2 < 1 \}
$$
We further make a smooth perturbation around the vertex $(0, 0,
0)$, while keeping curvature non-negative. Clearly, as $\varepsilon
\to 0$, the family of conic surfaces $\{ D^2_\varepsilon \}$
collapses to an interval $[0, 1)$. Let $X^2 = [0, 1) \times S^1$ be
the limiting space. As $\varepsilon \to 0$, our $3$-manifolds $
M^3_{\frac{1}{\varepsilon}}= D^2_\varepsilon \times S^1$ collapsed
to $X^2$ with $\partial X^2 \neq \varnothing$.

In this example, for every point $p\in
\partial X^2$, the space of direction $N_p(X^2)$ is a closed
half circle with diameter $\pi$.
\end{example}

\begin{example}\label{ex5.2}
In this example, we will construct the limiting surface with
boundary corner points. Let us consider the unit circle $S^1 =
\mathbb R/\mathbb Z$ and let $\varphi: \mathbb R^1 \to \mathbb R^1$
be an involution given by $\varphi(s) = -s $. Then its quotient
$S^1/\langle \varphi \rangle$ is isometric to $[0, \frac 12]$. Our
target limiting surface will be $ X^2 = [0, 1) \times [0, \frac
12]$. Clearly, $X^2$ has boundary corner point with total angle
$\frac{\pi}{2}$.

To see that $X^2$ is a limit of smooth Riemannian $3$-manifolds $\{
\bar{M}^3_{ \frac{1}{ \varepsilon } } \}$ while keeping curvatures
non-negative, we proceed as follows. We first viewed a 2-sheet cover
$M^3$ as an orbit space of $N^4$ by a circle action. Let $N^4 = D^2
\times \big (\mathbb R^2/(\mathbb Z \oplus \mathbb Z ) \big) $ and
let $\{(re^{i\theta},s,t)|0\le\theta\le 2\pi,0\le s\le 1,0\le t\le 1
\}$ be a parametrization of $D^2 \times \mathbb R^2$ and hence for
its quotient $N^4$. There is a circle action $\psi_\lambda: N^4 \to
N^4$ given by $\psi_\lambda ( re^{i\theta}, s, t) = ( re^{i(\theta +
2\pi \lambda)}, s, t+ \lambda)$ for each $e^{i2\pi\lambda} \in S^1$.
We also define an involution $ \tau: N^4 \to N^4$ by
$\tau(re^{i\theta}, s, t) = (re^{i\theta}, -s, t + \frac 12)$. It
follows that $ \tau \circ \tau = id $. Let $S^1 \ltimes \mathbb Z_2$
be subgroup generated by $\{\tau, \psi_\lambda\}_{\lambda \in S^1
}$. We introduce a family of metrics:
$$
g_\varepsilon = dr^2 + rd\theta^2 + ds^2 + \varepsilon^2 dt^2.
$$
The transformations $\{\tau, \psi_\lambda\}_{\lambda \in S^1 }$
remain isometries for Riemannian manifolds $(N^4, g_\varepsilon)$.
Thus, $\bar{M}^3_{ \frac{1}{\varepsilon} } = (N^4, g_\varepsilon)/ (
S^1 \ltimes \mathbb Z_2)$ is a smooth Riemannian $3$-manifold with
non-negative curvature. It is easy to see that $\bar{M}^3_{
\frac{1}{\varepsilon} }$ is homeomorphic to a solid torus. As
$\varepsilon \to 0$, our Riemannian manifolds $\{ \bar{M}^3_{
\frac{1}{\varepsilon} } \}$ collapse to a lower dimensional space
$X^2 = [0, 1) \times [0, \frac 12] = \{ (r, s) | 0 \le r < 1, 0 \le
s \le \frac 12 \}$. The surface $X^2$ has a corner point with total
angle $\frac{\pi}{2}$.
\end{example}

In above two examples, if we set $\alpha = \frac{1}{\varepsilon}$,
then there exist an open subset $U_\alpha \subset M^3_\alpha$ and a
continue map $G_\alpha: U_\alpha \to X^2$ such that $G_\alpha^{-1}(
A ) = U_\alpha $ is homeomorphic to a solid torus, where $A$ is an
annular neighborhood of $\partial X^2$ in $X^2$.

Let us recall an observation of Perelman on the distance function
$r_{\partial X^2} $ from the boundary $\partial X^2$.

\begin{lemma}\label{lem5.3}
Let $X^2$ be a compact Alexandrov surface of curvature $\ge c$ and
with non-empty boundary $\partial X^2$, $\Omega_{- \varepsilon}=\{ p
\in X^2 | \dis(p,
\partial X^2) \ge \varepsilon\}$ and $A_\varepsilon = [X^2 -
\Omega_{- \varepsilon}]$. Then, for sufficiently small
$\varepsilon$, the distance function $r_{\partial X^2}
=\dis_{X^2}(\cdot, \partial X^2)$ from the boundary has no critical
point in $A_\varepsilon$.
\end{lemma}
\begin{proof}
We will use a calculation of Perelman and a result of Petrunin to
complete the proof. By the definition, if $X^2$ has curvature $\ge
c$, then $\partial X^2$ must be convex. For any $x \in \partial
X^2$, we let $\theta(x) = \diam[\Sigma_x(X^2)]$ be the total tangent
angle of the convex domain $X^2$. It follows from the convexity $0 <
\theta(x) \le \pi$.

We consider the distance function $f(y) = \dis_{X^2}(y, \partial
X^2)$ for all $y \in X^2$. Perelman (cf. \cite{Per1991} page 33,
line 1) calculated that
$$
|\nabla f(x)| = \sin \left(\frac{\theta(x)}{2}\right) > 0,
$$
for all $x \in \partial X^2$. (Perelman stated his formula for
spaces with non-negative curvature, but his proof using the first
variational formula is applicable to our surface $X^2$ with
curvature $\ge -1$).

Corollary 1.3.5 of \cite{Petr2007} asserts that if there is a
converging sequence $\{ x_n\} \subset X^2 $ with $x_n \to x \in
\partial X^2$ as $n \to \infty$ then
$$\liminf_{n \to \infty}|\nabla f (x_n) |\ge
|\nabla f (x) |.
$$
Since $X^2$ is compact, it follows from the above discussion that
there is an $\varepsilon >0$ such that $|\nabla f (y) | > 0 $ for $y
\in A_\varepsilon$.
\end{proof}

We now recall a theorem of Shioya-Yamaguchi with our own proof. The
proof of Shioya-Yamaguchi used a version of Margulis Lemma, which we
will use our \autoref{thm1.17} in \S1 instead.

\begin{theorem}[\cite{SY2000} page 2]\label{thm5.4}
Let $\{(M^3_\alpha, p_\alpha) \}$ be a sequence of collapsing
$3$-manifolds as in Theorem 0.1'. Suppose that $\{ (B_{M^3_\alpha}(
p_\alpha, r), p_\alpha) \} \to (X^2, p_\infty) $ with $p_\infty \in
\partial X^2$ and $\partial X^2 $ is homeomorphic to $S^1$. Then
there is $\delta_1 > 0$ such that $B_{M^3_\alpha}(p_\alpha,
\delta_1) $ is homeomorphic to $D^2 \times I \cong D^3$ for all
sufficiently large $\alpha$.

Moreover, there exist an $\varepsilon > 0$ and a sequence of closed
curves $\varphi_\alpha: S^1 \to M^3_\alpha $ with $\{\varphi_\alpha
(S^1)\} \to \partial X^2$ as $ B_{M^3_\alpha}( p_\alpha, r)
\stackrel{G-H}{\longrightarrow} X^2$ such that a
$\varepsilon$-tubular neighborhood $ U_{\varepsilon}[\varphi_\alpha
(S^1)] $ of $\varphi_\alpha (S^1)$ in $M^3_\alpha$ is homeomorphic
to a solid tori $D^2 \times S^1$ for sufficiently large $\alpha$.
\end{theorem}
\begin{proof}
By conic lemma (cf. \autoref{prop1.7}), the number of points $p\in
\partial X^2$ with $\diam (\Sigma_p)\le \frac{\pi}{2}$ is finite,
denoted by $\{b_1, \cdots, b_s\}$. Since $\partial X^2 \sim S^1$ is
compact, it follows from \autoref{prop1.15} that there is a common
$\ell >0$ such that (1) the distance function $r_p(y) = \dis(p, y)$
has not critical point in $[B_{X^2}(p, \ell) -\{p\}]$; and (2)
$B_{X^2}(p, \ell)$ is homeomorphic to the upper half disk $D^2_+$,
for all $p \in \partial X^2$.

Since $\partial X^2 $ is homeomorphic to $S^1$, we can approximate
$\partial X^2$ by a sequence of closed broken geodesics $\{
\sigma_\infty^{(j)} \}$ with vertices $\{ q_1, \cdots, q_j\} \subset
\partial X^2$. We may assume that $\{b_1, \cdots, b_s\}$ is a
subset of $\{ q_1, \cdots, q_j\}$ for all $j \ge s$. We also require
that the distance between two consecutive vertices is less than
$c/j$ (i.e., $\dis_{\partial X^2}(q_i, q_{i+1}) \le \frac{c}{j} <
\frac{\ell}{8}$), for sufficiently large $j$.

We now choose a sequence of finite sets $\{ p_1^\alpha, \cdots,
p_j^\alpha\}$ such that $p^\alpha_i \to q_i$ as $\alpha \to \infty$
and $\{ p_1^\alpha, \cdots, p_j^\alpha\}$ span an embedded broken
geodesic $\sigma_\alpha^{(j)}$ in $M^3_\alpha$. It follows that
\begin{equation}\label{eq5.1}
\sigma_\alpha^{(j)} \to \sigma_\infty^{(j)}
\end{equation}
as $\alpha \to \infty$. Therefor, we may assume that there is a
sequence of smooth embedded curves $\varphi_\alpha: S^1 \to
M^3_\alpha$ such that
\begin{equation}\label{eq5.2}
\varphi_\alpha(S^1) \to \partial X^2
\end{equation}
as $M^3_\alpha\to X^2$. Thus, when $\{z(B_{M^3_\alpha}( p_\alpha,
r), p_\alpha) \} \to (X^2, p_\infty) $, the corresponding closed
curves $\varphi_\alpha(S^1) \stackrel{G-H}{\longrightarrow}
\partial X^2$ in the Gromov-Hausdorff topology, as $M^3_\alpha \to X^2$.

We now choose a sufficiently large $j_0 $ and divide our closed
curve $\varphi_\alpha: S^1 \to M^3_\alpha$ into $j_0$-many arcs of
constant speed:
$$
\varphi_{\alpha, i}: [a_i, a_{i+1}] \to M^3_\alpha
$$
for $i = 0, 1,\cdots, j_0$, where $\varphi_\alpha(a_{j_0+1}) =
\varphi_\alpha(a_0)$.

For a fixed $i$, we let
$$
A^i_{\alpha, s} = \varphi_\alpha([a_i + s, a_{i+1} -s])
$$
and $q_{\alpha, i} = \varphi_\alpha(a_i)$. We will show that there
is an $\varepsilon_i $ such that
\begin{equation}\label{eq5.3}
U_{\varepsilon_i}( A^i_{\alpha, 0}) = B_{M^3_\alpha}(q_{\alpha, i},
\varepsilon_i ) \cup U_{\varepsilon_i}( A^i_{\alpha, s} ) \cup
B_{M^3_\alpha}(q_{\alpha, i+1}, \varepsilon_i )
\end{equation}
is homeomorphic to the unit $3$-ball $D^3$. Our theorem will follow
from \eqref{eq5.3}. It remains to establish \eqref{eq5.3}.

We will show that each of $\{ B_{M^3_\alpha}(q_{\alpha, i},
\varepsilon_i ), U_{\varepsilon_i}( A^i_{\alpha, s} ),
B_{M^3_\alpha}(q_{\alpha, i+1}, \varepsilon_i )\}$ is homeomorphic
to $D^3$.

For $U_{\varepsilon_i}( A^i_{\alpha, s} )$, we first observe that
$\partial X^2$ is convex. Thus, the boundary arc $\varphi_\infty
([a_i + s, a_{i+1} -s]) \subset \partial X^2$ is a
Perelman-Sharafutdinov semi-gradient curve of the distance function
$r_{q_{\infty, i}}(y) = \dis_{X^2}(y, q_{\infty, i})$. We already
choose $\ell$ sufficiently small and $j_0$ sufficiently large so
that $r_{q_{\infty, i}}$ has no critical point on $[B_{X^2}(
q_{\infty, i}, \ell) -\{q_{\infty, i} \}]$. By our construction, we
have $r_{q_{\alpha, i}}(\cdot) \to r_{q_{\infty, i}}(\cdot)$, as
$\alpha \to \infty$. It follows from \autoref{prop1.14} that the
distance function has no critical points on $U_{\varepsilon_i}(
A^i_{\alpha, s} )$. Using Perelman's fibration theorem, we obtain a
fibration
$$
N^2_{\alpha, i} \to U_{\varepsilon_i}( A^i_{\alpha, s} )
\stackrel{r_{q_{\alpha, i}}}{\longrightarrow} (s, \ell_{\alpha,
i}-s).
$$
It follows that $ U_{\varepsilon_i}( A^i_{\alpha, s} ) \sim
[N^2_{\alpha, i} \times (s, \ell_{\alpha, i}-s)]$. We will first use
\autoref{thm1.17} and its proof to show that $\partial N^2_{\alpha,
i} \sim S^1$.

Let $\varepsilon_0$ be given by \autoref{lem5.3} and $r_{\alpha}( y)
=\dis_{M^3_\alpha}(y, \varphi_\alpha(S^1))$. By the proof of
\autoref{thm1.17}, the map $F_\alpha( z) = (r_{q_{\alpha, i}}(z),
r_{\alpha}(z)) $ is regular on $Z^3_{\alpha, i}(s/2, \varepsilon_0)
= [ U_{\varepsilon_0}( A^i_{\alpha, s} ) - U_{s/2}( A^i_{\alpha, s}
) - B_{M^3_\alpha}(q_{\alpha, i}, \varepsilon_i ) -
B_{M^3_\alpha}(q_{\alpha, i+1}, \varepsilon_i )]$. There is a circle
fibration $ S^1 \to Z^3_{\alpha, i}(\frac s2, \varepsilon_0)
\stackrel{F_\alpha}{\longrightarrow} \mathbb R$. This proves that
$\partial N^2_{\alpha, i} \sim S^1$. It follows that
$H^2_{\alpha, i} = \{\partial [ U_{\varepsilon_i}( A^i_{\alpha, s})]
- B_{M^3_\alpha}(q_{\alpha, i}, \varepsilon_i ) -
B_{M^3_\alpha}(q_{\alpha, i+1}, \varepsilon_i )\} $ is homeomorphic
to a cylinder $S^1 \times (s, \ell_{\alpha, i}-s)$.
 Using two points
suspension of the cylinder $H^2_{\alpha, i}$, we can further show
that
$$
\partial [ U_{\varepsilon_i}( A^i_{\alpha, s})] \sim S^2
$$
It remains to show that $ U_{\varepsilon_i}( A^i_{\alpha, s} ) \sim
D^3$ for sufficiently small $\varepsilon_i >0$. Suppose contrary, we
argue as follows. Using \autoref{prop1.14}, we see that $r_\alpha$
has no critical points in $ [U_{\varepsilon_0}( A^i_{\alpha, s} )-
U_{\varepsilon_0/2}( A^i_{\alpha, s} )]$. Let $\lambda_\alpha$ be
the largest critical value of $r_{\alpha}$ in $U_{\varepsilon_0}(
A^i_{\alpha, s} )$. By our assumption, $\lambda_\alpha \to 0$ as
$\alpha \to \infty$. We now consider a sequence of re-scaled spaces
$\{(\frac{1}{\lambda_\alpha}M^3_\alpha, q_{\alpha, i} ) \}$. Its
sub-sequence converges to a limiting space $(Y^s_\infty, \bar
q_{\infty, i})$. Recall that $\dim(X^2)= 2$. For the same reason as
in the proof of \autoref{prop2.3}, we can show that
$\dim(Y^s_\infty) = 3$ and that $Y^3_\infty$ has no boundary. Let
$\bar N^s_\infty$ be the soul of $Y^3_\infty$. There are three
possibilities.

(1) If the soul $\bar N^s_\infty$ is a point, then by
\autoref{thm2.11} we obtain that $Y^3_\infty \sim \mathbb R^3$. It
follows that $U_{\varepsilon_i}( A^i_{\alpha, s} ) \sim D^3$ by
Perelman's stability theorem, we are done.

(2) If the soul $\bar N^s_\infty$ is a circle, then $Y^3_\infty$ (or
its double cover) is isometric to $S^1 \times Z^2$, where $Z^2 $ is
homeomorphic to $\mathbb R^2$. Let $\bar \varphi_\infty ( \mathbb R
) $ be the limit curve in the re-scaled limit space $Y^3_\infty$.
Since $Y^3_\infty$ (or its double cover) is isometric to $S^1 \times
Z^2$, we have $U_r(\bar \varphi_\infty ([-R, R]))$ is homeomorphic
to $ S^1 \times D^2$. By Perelman's stability theorem, we would have
$\partial [ U_{\varepsilon_i}(A^i_{\alpha, s})]\sim \partial [
S^1\times D^2 ]\sim T^2$, which contradicts to the assertion $
\partial [U_{\varepsilon_i}(A^i_{\alpha,s})]\sim S^2$.

(3) If the soul $\bar N^s_\infty$ of $Y^3_\infty$ has dimension $2$,
then it follows from that the infinity of $Y^3_\infty$ would have at
most two points. However, since $\dim(X^2) =2$, the infinity of
$Y^3_\infty$ has an arc, a contradiction. This completes the proof
of $U_{\varepsilon_i}( A^i_{\alpha, s} ) \sim D^3$ for sufficiently
large $\alpha$.

With extra efforts, we can also show that $B_{M^3_\alpha}(q_{\alpha,
i}, \varepsilon) \sim D^3$ for sufficiently large $\alpha$. Hence, $
U_{\varepsilon}[\varphi_\alpha (S^1)] \sim \cup_i D^3_i \sim [D^2
\times S^1]$. This completes our proof. \end{proof}

\subsection{Collapsing to a closed interval}\label{section5.2} \

\smallskip

Since all our discussions in this sub-section are semi-local, we may
have the following setup:
\begin{equation}\label{eq5.4}
\lim_{\alpha\to+\infty}(M^3_\alpha, p_\alpha) = (I, O)
\end{equation}
in the pointed Gromov-Hausdorff distance, and $I = [0, \ell]$ is an
interval, $O\in I$ is an endpoint of $I$. We will study the topology
of $B_{M^3_\alpha}(p_\alpha, r)$ for a given small $r$.

We begin with four examples to illustrate how smooth Riemannian
$3$-manifolds $M^3_\alpha$ collapse to an interval $[0, \ell]$ with
curvature bounded from below. Collapsing manifolds in these
manifolds are homeomorphic to one of the following: $\{ D^3,
[\mathbb {RP}^3 -D^3], S^1 \times D^2, K^2 \ltimes [0,\frac 12]\}$,
where $D^3 = \{ x \in \mathbb R^3 | |x | < 1\}$ and $K^2$ is the
Klein bottle.

\begin{example}\label{ex5.5}
($ M^3_\alpha$ is homeomorphic to $D^3$). For each $\varepsilon >
0$, we consider a convex hypersurface in $\mathbb R^4$ as follows.
We glue a lower half of the $3$-sphere
$$
S^{3}_{\varepsilon, -} = \{ (x_1, x_2, x_3, x_4) \in \mathbb R^4 | 0
\le x_4 \le \varepsilon, x_1^2 + x_2^2 + x_3^2 + (x_4 -
\varepsilon)^2 = \varepsilon^2\}
$$
to a finite cylinder
$$
S^2_\varepsilon \times [\varepsilon, 1] = \{(x_1, x_2, x_3, x_4) \in
\mathbb R^4 |\varepsilon \le x_4 \le 1, x_1^2 + x_2^2 + x_3^2 =
\varepsilon^2 \}.
$$
Our $3$-manifold $M^3_\varepsilon = S^{3}_{\varepsilon, -} \cup
\big( S^2_\varepsilon \times [\varepsilon, 1] \big)$ collapse to the
unit interval as $\varepsilon \to 0$.
\end{example}

For other cases, we consider the following example.

\begin{example}\label{ex5.6}
($ M^3_\alpha$ homeomorphic to $S^1 \times D^2 $). Let us glue a
lower half of $2$-sphere
$$
S^{2}_{\varepsilon, -} = \{ (x_1, x_2, x_3) \in \mathbb R^3 | 0 \le
x_3 \le \varepsilon, x_1^2 + x_2^2 + (x_3 - \varepsilon)^2 =
\varepsilon^2\}
$$
to a finite cylinder $S^1_\varepsilon \times [\varepsilon, 1]$. The
resulting disk
$$
D^2_\varepsilon = S^{2}_{\varepsilon, -} \cup \big( S^1_\varepsilon
\times [\varepsilon, 1] \big)
$$
is converge to unit interval, as $\varepsilon \to 0$. We could
choose $M^3_{\frac{1}{\varepsilon}} = S^1_{ \varepsilon^2 } \times
D^2_\varepsilon$. It is clear that $M^3_{\frac{1}{\varepsilon}} \to
[0, 1]$ as $\varepsilon \to 0$.
\end{example}

\medskip

We now would like to consider the remaining cases. Of course, two
un-oriented surfaces $\mathbb {RP}^2_\varepsilon = S^2_\varepsilon
/\mathbb Z^2$ and $K^2 = T^2/\mathbb Z_2$ would converge to a point,
as $\varepsilon \to 0$. However, the {\it twisted } $I$-bundle over
$\mathbb {RP}^2$ (or $K^2$) is homeomorphic to an oriented manifold
$ \mathbb {RP}^2 \ltimes [0, \frac 12] = [\mathbb {RP}^3 -D^3]$ (or
$K^2 \ltimes [0, \frac 12] = M$\"o$\ltimes S^1$), where $M$\"o is
the M\"obius band.

\begin{example}\label{ex5.7}($ M^3_\alpha$ homeomorphic to
$\mathbb {RP}^2 \ltimes [0, \frac 12] $ or $M$\"o$ \ltimes S^1$).

Let us first consider round sphere $S^2_\varepsilon = \{ (x_1, x_2,
x_3) \mathbb R^3 | |x| = \varepsilon\}$. There is an orientation
preserving involution $\tau: \mathbb R^3 \times [-\frac 12, \frac
12]$ given by $\tau (x, t) = (-x, -t)$. Suppose that $\langle\tau
\rangle = \mathbb Z_2$ is the subgroup generated by $\tau$. Thus,
the quotient of $S^2_\varepsilon \times [-\frac 12, \frac 12]$ is an
orientable manifold $\mathbb {RP}^2 \ltimes [0, \frac 12] = [\mathbb
{RP}^3 -D^3]$.

Similarly, we can consider the case of $K^2_\varepsilon \ltimes [0,
\frac 12] \to [0, \frac 12]$, where $K^2_\varepsilon$ is a Klein
bottle.
\end{example}

Shioya and Yamaguchi showed that the above examples exhausted all
cases up to homeomorphisms.

\begin{theorem}[Theorem 0.5 of \cite{SY2000}]\label{thm5.8}
Suppose that $\lim_{\alpha\to+\infty}(M^3_\alpha, p_\alpha)= (I, O)$
with curvature $\ge -1$ and $I = [0, \ell]$. Then $M^3_\alpha$ is
homeomorphic to a gluing of $C_0$ and $C_\ell$ to $N^2 \times (0,
1)$, where $C_0$ and $C_\ell$ are homeomorphic to one of $\{ D^3,
[\mathbb {RP}^3 - D^3], S^1 \times D^2, K^2 \ltimes [0, \frac 12]\}$
and $N^2$ is a quotient of $T^2$ or $S^2$.
\end{theorem}

For the proof of \autoref{thm5.8}, we need to establish two
preliminary results (see \autoref{thm5.9} - \ref{thm5.10} below).
Let us consider possible exceptional orbits in the Seifert fibration
\begin{equation}\label{eq5.5}
S^1 \to M^3_\alpha \to \interior (Y^2)
\end{equation}
for sufficiently large $\alpha$. We emphasize that the topological
structure of $M^3_\alpha$ depends on the number of extremal points
(or called essential singularities) of $Y^2$ in this case. Moreover,
the topological structure of $M^3_\alpha$ also depends on the type
of essential singularity of $Y^2$, when we glue a pair of solid tori
together, (see \eqref{eq5.14} below). Therefore, we need the
following theorem with a new proof.

\begin{theorem}[Compare with Corollary 14.4 of \cite{SY2000}]\label{thm5.9}
Let $Y^2$ be a connected, non-compact and complete surface with
non-negative curvature and with possible boundary. Then the
following is true.

\begin{enumerate}[{\rm (i)}]
\item If $Y^2$ has no boundary, then $Y^2$ has at most two extremal points
(or essential singularities). When $Y^2$ has exactly two extremal
points, $Y^2$ is isometric to the double ${\rm dbl}( [0, \ell]
\times [0, \infty))$ of the flat half strip.
\item If $Y^2$ has non-empty boundary $\partial Y^2 \neq 0$, then $Y^2$
has at most one interior essential singularity.
\end{enumerate}
\end{theorem}
\begin{proof}
For (i), we will use the multi-step Perelman-Sharafutdinov
semi-flows to carry out the proof. For the assertion (i), we
consider the Cheeger-Gromoll type Busemann function
\begin{equation}\label{eq5.6}
f(y) = \lim_{t \to \infty}[t - \dis(y, \partial B_{Y^2}(y_0, t))].
\end{equation}
In \autoref{section2}, we already showed that $\Omega_c =
f^{-1}((-\infty, c])$ is compact for any finite $c$. Let $a_0 = \inf
\{f(y)| y \in Y^2\}$. Then the level set $\Omega_{a_0} = F^{-1}(
a_0)$ has dimension at most $1$. Recall that $\Omega_{a_0}$ is
convex by the soul theory. Thus, $\Omega_{a_0}$ is either a point or
isometric to a length-minimizing geodesic segment: $\sigma_0: [0,
\ell] \to Y^2$.

{\it Case a.} If $\Omega_{a_0} = \{ z_0 \}$ is a point soul of
$Y^2$, by an observation of Grove \cite{Grv1993} we see that the
distance function $r(y) = \dis(y, z_0)$ has no critical point in
$[Y^2-\{ z_0\}]$. It follows that $z_0$ is only possible extremal
point of $Y^2$.

{\it Case b.} If $\Omega_{a_0} = \sigma_0 ([0, \ell])$, then for
$s\in (0, \ell)$ Petrunin \cite{Petr2007}] showed that
$T^-_{\sigma_0(s)}(Y^2) = \mathbb R^2$. Hence, two endpoints
$\sigma_0(0)$ and $\sigma_0(\ell)$ are only two possible extremal
points of $Y^2$.

Suppose that $Y^2$ has exactly two extremal points $\sigma_0(0)$ and
$\sigma_0(\ell)$. We choose two geodesic rays $\psi_0: [0, \infty)
\to Y^2$ and $\psi_\ell: [0, \infty) \to Y^2$ from two extremal
points $\sigma_0(0)$ and $\sigma_0(\ell)$ respectively. The broken
geodesic $\Gamma = \sigma_0 \cup \psi_0 \cup \psi_\ell$ divides our
space $Y^2$ into two connected components:
\begin{equation}\label{eq5.7}
[Y^2 - \Gamma] = \Omega_- \cup \Omega_+.
\end{equation}
We now consider the distance function
\begin{equation}\label{eq5.8}
w_\pm(u) = \dis_{\Omega_\pm}(\psi_\ell(u), \psi_0(u)).
\end{equation}

Let us consider the Perelman-Sharafutdinov semi-gradient flow
$\frac{d^+w}{dt} = \nabla (-f)|_w$ for Busemann function $-f(w)$.
Recall that the semi-flow is distance non-increasing, since the
curvature is non-negative. Hence, we see that $t \to w_\pm(u-t)$ is
a non-increasing function of $t \in [0, u]$. It follows that
\begin{equation}\label{eq5.9}
w_\pm(u) \ge w_\pm(0) = \ell
\end{equation}
for $u \ge 0$. On other hand, we could use multi-step geodesic
triangle comparison theorem as in \cite{CMD2009} to verify that
\begin{equation}\label{eq5.10}
w_\pm(u) \le w_\pm(0) = \ell
\end{equation}
with equality holds if and only if four points $\{ \sigma_0(0),
\sigma_0(\ell), \psi_\ell(u), \psi_0(u)\}$ span a flat rectangle in
$\Omega_\pm$. Therefore, $\Omega_\pm $ is isometric to $[0, \ell]
\times [0, \infty)$. It follows that $Y^2$ is isometric to the
double ${\rm dbl}( [0, \ell] \times [0, \infty))$.

The second assertion (ii) follows from (i) by using $\text{dbl}
(Y^2)$.
\end{proof}

For compact surfaces $X^2$ of non-negative curvature, we use an
observation of Perelman (\cite{Per1991} page 31) together with
multi-step Perelman-Sharafutdinov semi-flows, in order to estimate
the number of extremal points in $X^2$.

\begin{theorem}\label{thm5.10}
{\rm (1)} Suppose that $X^2$ is a compact and oriented surface with
non-negative curvature and with non-empty boundary $\partial X^2
\neq \varnothing$. Then $X^2$ has at most two interior extremal
points. When $X^2$ has two interior extremal points, then $X^2$ is
isometric to a gluing of two copies of flat rectangle along their
three corresponding sides.

{\rm (2)} Suppose that $X^2$ is a closed and oriented surface with
non-negative curvature. Then $X^2$ has at most four extremal points.
\end{theorem}
\begin{proof}
(1) We consider the double ${\rm dbl}(X^2)$ of $X^2$. If ${\rm
dbl}(X^2)$ is not simply-connected and if it is an oriented surface
with non-negative curvature, then we already showed that ${\rm
dbl}(X^2)$ is a flat torus.

We may assume that $X^2$ is homeomorphic to a disk: $X^2 \cong D^2$
and $\partial X^2=S^1$. Let $X_{-\varepsilon}=\{x\in X^2 | \dis(x,
\partial X^2)\ge\varepsilon\}$. Perelman \cite{Per1991} already
showed that $X_{-\varepsilon}$ remains to be convex, (see also
\cite{CMD2009}). If $a_0=\max\{\dis(x, \partial X^2) \}$, then $X_{-
a_0}$ is either a geodesic segment or a single point set. Thus,
$X^2$ has at most two interior extremal points. The rest of the
proof is the same as that of \autoref{thm5.9} with minor
modifications.

(2) Suppose that $X^2$ has 2 distinct extremal points $p$ and $q$.
Let $N$ be a geodesic segment connecting $p$ and $q$. We need to
show for the function $f(x)=\dis_N(x)=\dis(N, x)$ is concave for all
$x\in X^2\setminus N$. Clearly for $x\in X^2\setminus N$ there
exists $x_N\in N$ such that $|xx_N|=\dis(x,N)$. There are two
possibilities:
\begin{enumerate}[(i)]
\item
$x_N$ is in the interior of $N$, then proof of the concavity of $f$
is exactly same as the one of Theorem 6.1 in \cite{Per1991} (see
also \cite{Petr2007} p156 and \cite{CMD2009}).

\item
$x_N$ is one of the endpoints, say $p$, then by first variational
formula we have
\begin{equation}\label{eq:5.10.1}
\dis_{\Sigma_p}(\Uparrow_p^x, \Uparrow_p^q)\ge \frac{\pi}{2}
\end{equation}
where $\Uparrow_p^x$ denotes the set of directions of geodesics from
$p$ to $x$ in $\Sigma_p$. On the other hand, by our assumption $p$
is an interior extremal point so
\begin{equation}\label{eq:5.10.2}
\diam \Sigma_p\le \frac{\pi}{2}
\end{equation}
\end{enumerate}
Combine \eqref{eq:5.10.1} and \eqref{eq:5.10.2} we have
$$\dis_{\Sigma_p}(\Uparrow_p^x, \Uparrow_p^q) = \frac{\pi}{2}$$
The rest of the proof is same as the one of Theorem 6.1 in
\cite{Per1991} (or \cite{Petr2007}, \cite{CMD2009}).

\begin{figure*}[ht]
\includegraphics[width=250pt]{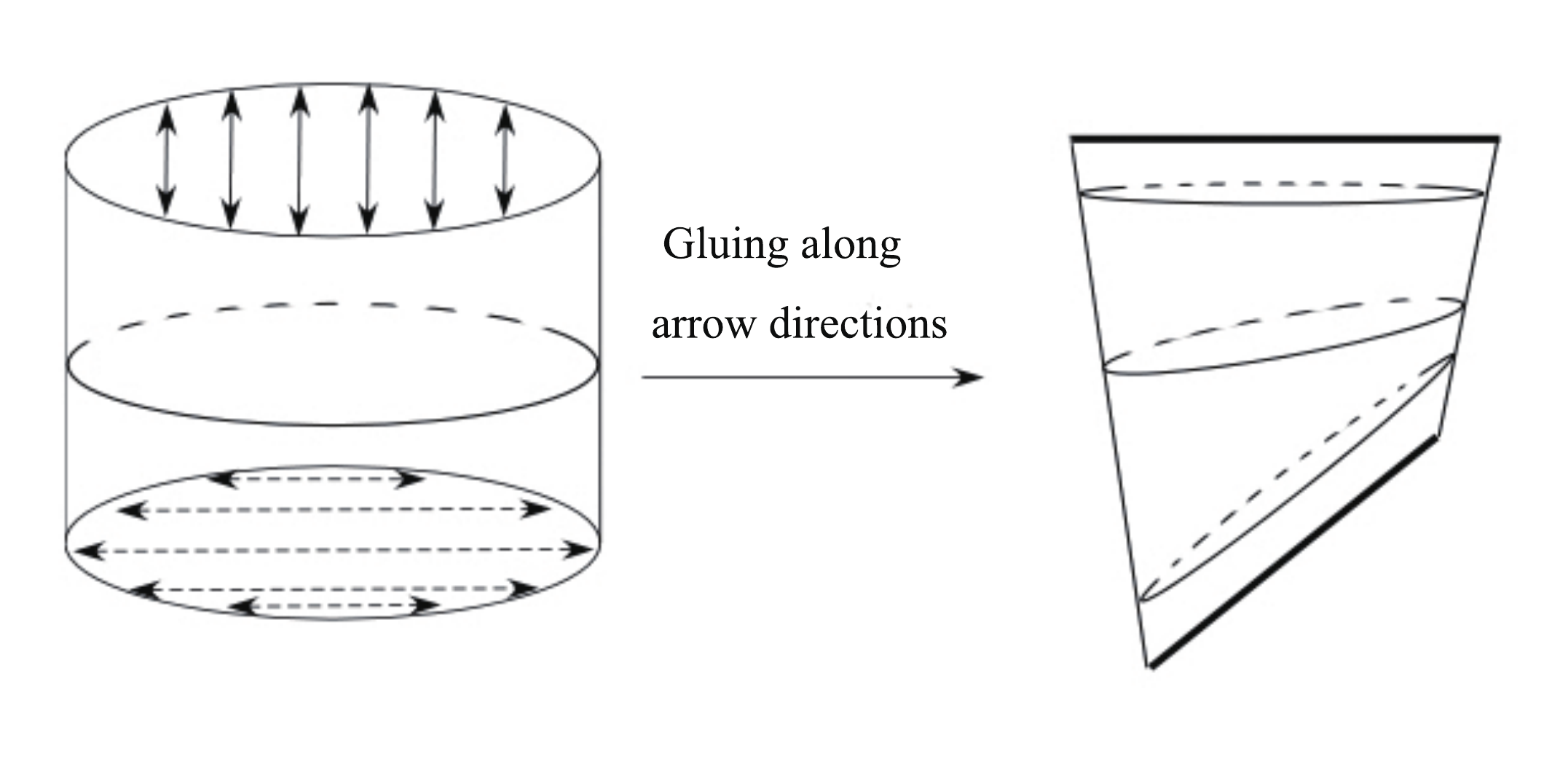}\\
\caption{A $2$-sphere $S^2$ with $4$ essential singularities}\label{fig:5.0}
\end{figure*}

Hence, we have shown that $f$ is concave function on $X^2\setminus
N$. Let $A$ be the maximum set. Then $\Omega_1$ is either a geodesic
segment or one point. Thus, $X^2\setminus N$ contains at most 2
extremal points by our proof of (1). Therefore, the number of total
extremal points is at most 4.
\end{proof}

The case of exactly 4 extremal points on a topological $2$-sphere
can be illustrated in \autoref{fig:5.0}. Let us now complete the
proof of \autoref{thm5.8}.

\bigskip
\noindent \textbf{Proof of \autoref{thm5.8}:}
\bigskip

The proof is due to Shioya-Yamaguchi \cite{SY2000}. Our new
contribution is the simplified proof of \autoref{thm5.9}, which will
be used in the study of subcases (1.b) and (2.b) below. For the
convenience of readers, we provide a detailed argument here.

Recall that in \autoref{section4}, we already constructed a
fibration
$$
N^2_\alpha \to U_\alpha \to \interior (X^1)
$$
with shrinking fiber $N^2_\alpha$ is homeomorphic to either $S^2$ or
$T^2$ for sufficiently large $\alpha$.

When $\{ p_\alpha \} \to p_\infty \in \partial X^1$, since the
fibers $\{N^2_\alpha\}$ are shrinking, we may assume that
\begin{equation}\label{eq5.11}
\partial B_{M^3_\alpha}(p_{\alpha},\delta) = S^2 \quad \quad \text{or}\quad \quad
\partial B_{M^3_\alpha}(p_{\alpha},\delta) = T^2
\end{equation}
for sufficiently large $\alpha$, where $X^1 = [0, \ell]$ and $\delta
= \ell/2$.

If there exists an $\varepsilon_0$ such that $r_\alpha(x) =
\dis_{M^3_\alpha}(x, p_\alpha)$ has no critical points on
the punctured ball $[B_{M^3_\alpha}(p_{\alpha},\varepsilon_0) -
\{p_{\alpha}\}]$ for sufficiently large $\alpha$, then, using the
proof of \autoref{prop1.7}, we can show that
$B_{M^3_\alpha}(p_{\alpha},\varepsilon_0) \sim D^3$ for sufficiently
large $\alpha$. Thus, conclusion of \autoref{thm5.8} holds for this
case.

Otherwise, there exists a subsequence $\{ \varepsilon_\alpha \} \to
0$ such that $r_\alpha(x) = \dis_{M^3_\alpha}(x, p_\alpha)$ has a
critical point $q_\alpha$ with $\dis(p_\alpha, q_\alpha) =
\varepsilon_\alpha$. It follows that $
\left(\frac{1}{\varepsilon_{\alpha}}M^3_{\alpha},
\bar{p}_{\alpha}\right)\to (Y^k, \bar{p}_{\infty})$. Using a similar
argument as in the proof of \autoref{prop2.3}, we can show that the
limiting space $Y^k$ has non-negative curvature and $\dim(Y^k) = k
\ge 2$, where $\bar{p}_{\alpha}$ is the image of $p_\alpha$ under
the re-scaling. There are two cases described in \eqref{eq5.11}
above.

\medskip
{\bf Case 1.} $\partial B_{M^3_\alpha}(p_{\alpha},\delta) = S^2$.

If $B_{M^3_\alpha}(p_{\alpha},\delta)$ is homeomorphic to $ D^3$,
then no further verification is needed. Otherwise, we may assume
$B_{M^3_\alpha}(p_{\alpha},\delta)\ne D^3$ and $\partial
B_{M^3_\alpha}(p_{\alpha},\delta)=S^2$.

Under these two assumptions, we would like to verify that
$B_{M^3_\alpha}(p_{\alpha},\delta)$ is homeomorphic to $\mathbb
{RP}^2\ltimes I$. There are two sub-cases:

\medskip

{\it Subcase 1.a. } If $\dim Y=3$, then $Y^3$ is non-negatively
curved, open and complete. Let $N^k$ be a soul of $Y^3$.

Using Perelman's stability theorem, we claim that the soul $N^k$
cannot be $S^1$. Otherwise the boundary of
$B_{M^3_\alpha}(p_{\alpha},\delta)$ would be $T^2$, a contradiction.
For the same reason, the soul $N^k$ cannot be $T^2$ or $K^2$.
Otherwise the boundary $\partial B_{M^3_\alpha}(p_{\alpha},\delta)$
would be homeomorphic to $T^2$, contradicts to our assumption.
Because the boundary of $B_{M^3_\alpha}(p_{\alpha},\delta)$ only
consists of one component, the soul of $N^k$ of $Y$ cannot be $S^2$;
otherwise we would have $Y^3 \sim S^2 \times \mathbb R$ and $
\partial B_{M^3_\alpha}(p_{\alpha},\delta) \sim S^2 \cup S^2$, a
contradiction.

Therefore, we have demonstrated that the soul $N^k$ must be either a
point or $\mathbb {RP}^2$. It follows that either $B_Y(N^k,R)\cong
D^3\text{ or } \mathbb {RP}^2\ltimes I$ by soul theorem. Hence, we
conclude that $B_Y(N^k,R)\cong \mathbb {RP}^2\ltimes I$ for
sufficiently large $R$. Using Perelman's stability theorem, we
conclude that $B_{M^3_\alpha} (p_{\alpha},\delta)$ is homeomorphic
to $\mathbb {RP}^2\ltimes I$ for this sub-case.
\medskip

{\it Subcase 1.b. } If $\dim Y=2$, i.e. $Y$ is a surface with
possibly non-empty boundary. First we claim $\partial Y\ne
\varnothing$. For any fixed $r>0$, we have
\begin{equation}\label{eq5.12}
\partial B_{M^3_\alpha}(p_\alpha, r)\cong S^2
\end{equation}
by our assumption that the regular fiber is $S^2$. Suppose contrary,
$\partial Y^2 = \varnothing$. We would have $Y\cong \mathbb R^2$ (or
a M\"obius band) by \autoref{thm2.6}, because $Y^2$ has one end.
Thus, for sufficiently large $R$ we have $\partial B_Y(\bar
p_\infty, R)= S^1$. Applying fibration theorem for the collapsing to
the surface case, we would further have
$$
\partial B_{M^3_\alpha}(\bar{p}_{\alpha}, R)= T^2
$$
which contradicts to our boundary condition \eqref{eq5.12}. Hence,
$Y$ has non-empty boundary: $\partial Y^2 \neq \varnothing$.

Since $Y^2$ has one end and $\partial Y^2 \neq \varnothing$, by
\autoref{cor2.9} we know that $Y^2$ is homeomorphic to either
upper-half plane $[0, \infty) \times \mathbb R$ or isometric a half
cylinder. By previous argument, $Y^2$ can not be isometric to a half
cylinder. Hence, we have
\begin{equation}\label{eq5.13}
Y^2 \cong [0, \infty) \times \mathbb R
\end{equation}
and that $\partial Y$ is a non-compact set.

We further observe that if $\bar{p}_\infty$ is a boundary point of
$Y^2$, then, by \autoref{thm5.4}, we would have
$$
B_{M^3_\alpha}(p_\alpha, r) \cong [D^2_+ \times I] \cong D^3
$$
where $D^2_+$ is closed half disk, a contradiction.

Thus we may assume that $\bar{p}_\infty$ is an interior point of
$Y^2$. Let us consider the Seifert fiber projection
$B_{\frac{1}{\lambda_\alpha}M^3_{\alpha}}(\bar{p}_{\alpha},R)\to
B_Y(\bar{p}_\infty,R)$ for some large $R$. If
$B_Y(\bar{p}_\infty,R)$ contains no interior extremal point, then by
the proof of \autoref{thm5.4} one would have
$B_{\frac{1}{\varepsilon_\alpha}M^3_{\alpha}}(\bar{p}_{\alpha},R)\cong
D^2\times I\cong D^3$, a contradiction. Therefore, there exists an
extremal point inside $B(\bar{p}_\infty, R)$. Without loss of
generality, we may assume this extremal point is $O$.

By \autoref{thm5.9} and the fact that $\partial Y^2 \neq
\varnothing$, we observe that $Y^2$ has at most one interior
extremal point. In our case, $Y^2$ is isometric to the following
flat surface with a singularity. Let $\Omega = [0, h] \times [0,
\infty)$ be a half flat strip and $\Gamma = \{ (x, y) \in \Omega |
xy = 0\} \subset \partial \Omega $. Our singular flat surface $Y^2$
is isometric to a gluing two copies of $\Omega$ along the curve
$\Gamma$.

The picture of $B_Y(O,R)$ for large $R$ will look like
\autoref{figure5.1}, where the bold line denotes $\partial Y\cap
B_Y(O,R)$.
\begin{figure}[ht]
\includegraphics[width=150pt]{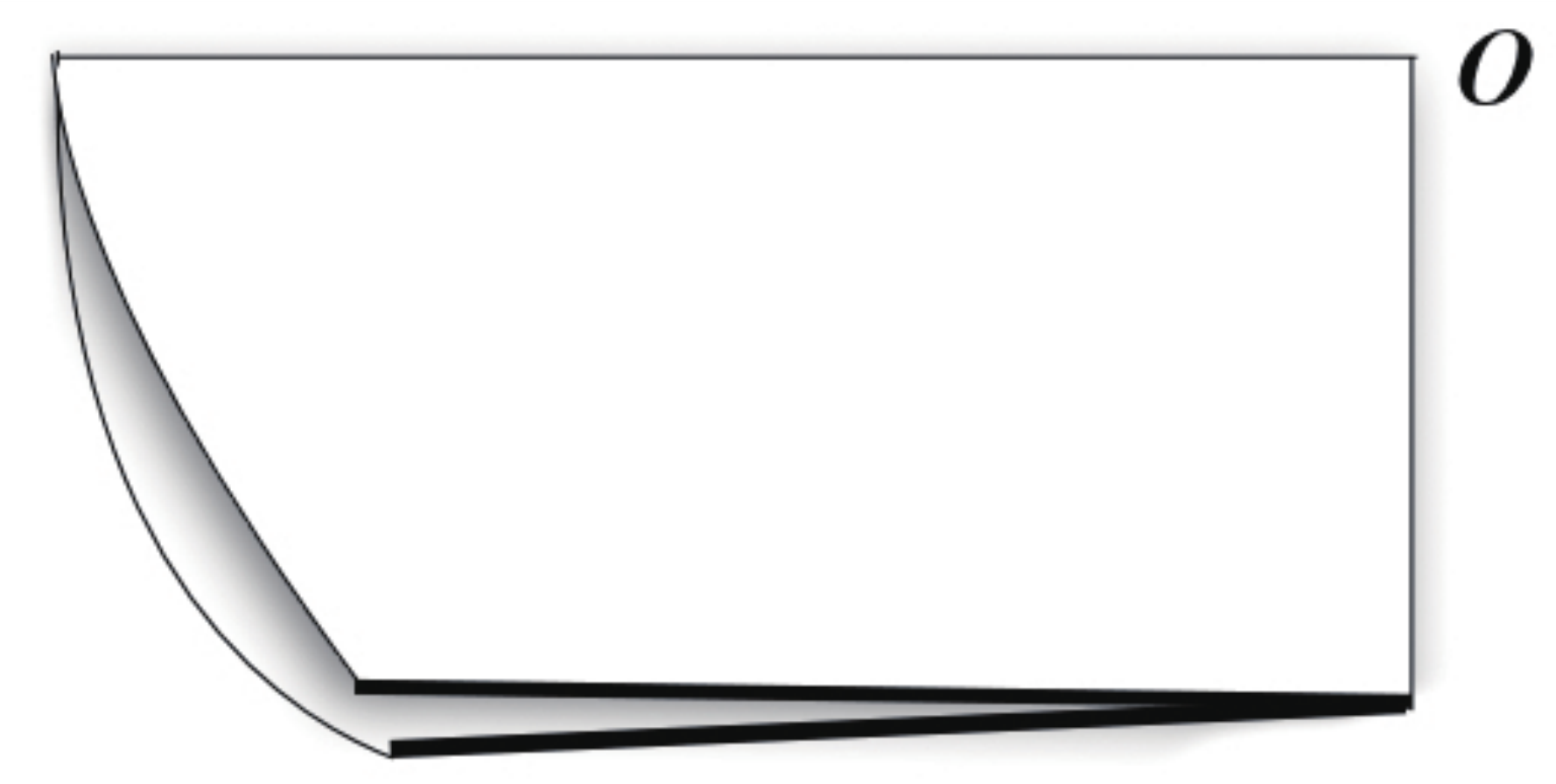}\\
\caption{A metric disk $B_{Y^2}(O,R)$ in $Y^2$ for large $R$}\label{figure5.1}
\end{figure}
By our assumption, we have $\partial B(\bar{p}_{\alpha}, R) \cong
S^2$. Thus we can glue a $3$-ball $D^3$ to $B(\bar{p}_{\alpha}, R)$
along $\partial B(\bar{p}_{\alpha}, R) \cong S^2$ to get a new
closed $3$-manifold $\hat M^3_\alpha$.

Recall that we have a (possibly singular) fibration
$$
S^1 \to B(\bar{p}_{\alpha}, R)
\stackrel{G_\alpha}{\longrightarrow}\interior (Y^2).
$$
By \autoref{thm2.1}, we have $ G_\alpha^{-1}(B_{Y^2}(O, h/2)) \cong
D^2 \times S^1. $ With some efforts, we can further show that $
[\hat M^3_\alpha - G_\alpha^{-1}(B_{Y^2}(O, h/4))] \cong S^1 \times
D^2. $

The exceptional orbit $G_\alpha^{-1}(O)$ is corresponding to the
case of $m_0=2$ in \autoref{ex2.0}. Finally we conclude that $ \hat
M^3_\alpha $ is homeomorphic to real projective $3$-space:
\begin{equation}\label{eq5.14}
\hat M^3_\alpha \cong (D^2\times S^1)\cup_{\psi_{\mathbb Z_2}}
(S^1\times D^2)\cong \mathbb {RP}^3.
\end{equation}

Therefore, by our construction, we have
$$B_{M^3_\alpha}(p_{\alpha},\delta) = [\hat M^3_\alpha \setminus D^3 ]
\cong [\mathbb{RP}^3\setminus D^3] \cong \mathbb{RP}^2\ltimes I.$$

This completes the first part of our proof for the case of $\partial
B_{M^3_\alpha}(p_{\alpha},\delta)= S^2$.

\medskip
{\bf Case 2. } If $\partial B_{M^3_\alpha}(p_{\alpha},\delta)=T^2$,
our discussion will be similar to the previous case with some
modifications. In fact, there are still two sub-cases:

\medskip

{\it Subcase 2.a.} If $\dim Y=3$, then the $Y$ has a soul $N^k$. We
assert that $N^k$ cannot be a point nor $S^2$ nor $\mathbb {RP}^2$
because the boundary $\partial B_{M^3_\alpha}(p_{\alpha},\delta) =
T^2$. In addition, we observe that $N^k$ cannot be $T^2$ since the
boundary of $B_{M^3_\alpha}(p_{\alpha},\delta)$ consists only one
component. Therefore, there are two remaining cases: $N^k$ is
homeomorphic to either $S^1$ or Klein bottle $K^2$. If the soul is
$S^1$ then $B_{M^3_\alpha} (p_{\alpha},\delta)\cong D^2\times S^1$.
Similarly, if soul is $K^2$, then
$B_{M^3_\alpha}(p_{\alpha},\delta)\cong K^2\ltimes I$.

\medskip

{\it Subcase 2.b.} $\dim Y=2$. It is clear that our limiting space
$Y^2$ is non-compact.

If $\partial Y^2 = \varnothing$, we proceed as follows. (i) When the
soul of $Y^2$ is $S^1$, by the connectedness of $\partial
B_{M^3_\alpha}(p_{\alpha}, \delta)$ we know that $Y$ is a open
M\"{o}bius band. Therefore, $B_{M^3_\alpha}(p_{\alpha},\delta)$ is
homeomorphic to product of M\"{o}bius band and $S^1$, i.e. a twist
$I$-bundle over $K^2$. (ii) When the soul of $Y^2$ is a single
point, $Y=\mathbb R^2$, which is non-compact. By \autoref{thm5.9},
we see that the number $k$ of extremal points in $Y^2$ is at most
$2$. Recall that there is (possibly singular) fibration: $S^1 \to
B_{M^3_\alpha}(p_{\alpha},\delta)
\stackrel{G_\alpha}{\longrightarrow} \interior (Y^2)$. If $k \le 1$,
we have $B_{M^3_\alpha}(p_{\alpha},\delta)\cong S^1\times D^2$. If
$k =2$, by \autoref{thm5.9}, we can further show that
$B_{M^3_\alpha}(p_{\alpha},\delta)\cong K^2\ltimes I$.

Let us now handle the remaining subcase when $\partial Y^2 \neq
\varnothing$ is not empty. By a similar argument as in subcase 1.b
above, we conclude that $B_{M^3_\alpha}(p_{\alpha},\delta)$ is
homeomorphic to either $\mathbb{RP}^2\ltimes I$ or $D^3$, which
contradicts to our assumption $\partial
B_{M^3_\alpha}(p_{\alpha},\delta) \cong T^2$.

This completes the proof of \autoref{thm5.8} for all cases. \qed

\section{Gluing local fibration structures and Cheeger-Gromov's compatibility
condition}\label{section6}
\setcounter{theorem}{-1}

In this section, we complete the proof of Perelman's collapsing
theorem for $3$-manifolds (Theorem 0.1'). In previous five sections,
we made progress in decomposing each collapsing $3$-manifold
$M^3_{\alpha}$ into several parts:
\begin{equation}\label{eq6.1}
M^3_{\alpha}= V_{\alpha,X^0}\cup V_{\alpha,X^1}\cup V_{\alpha,
\interior (X^2)}\cup W_{\alpha}
\end{equation}
where $V_{\alpha,X^0}$ is a union of closed smooth $3$-manifolds of
non-negative sectional curvature, $V_{\alpha,X^1}$ is a union of
fibrations over $1$-dimensional spaces with spherical or toral
fibers and $V_{\alpha, \interior (X^2)}$ admits locally defined
almost free circle actions.

Extra cares are needed to specify the definition of $ V_{\alpha,
X^i}$ for each $\alpha$. For example, we need to choose specific
parameters for collapsing $3$-manifolds $\{M^3_\alpha \}$ so that
the decomposition in \eqref{eq6.1} becomes well-defined. The choices
of parameters can be made in a similar way as in \autoref{section3}
of \cite{SY2005}.

\begin{theorem}\label{thm6.0}
Suppose that $\{ (M^3_\alpha, g^\alpha)\}$ be as in
\autoref{thm0.1} and that $M^3_\alpha$ has no connected components
which admit metrics of non-negative curvature. Then there are two
small constants $\{c_1, \varepsilon_1\}$ such that
$$
V_{\alpha,X^1} = \{x \in M^3_\alpha \quad | \quad \dis_{GH}
(B_{(M^3_\alpha, \rho^{-2}_\alpha g_{ij}^\alpha)}(x, 1) ,
[0,\ell])\le \varepsilon_1 \}
$$
and
$$
V_{\alpha, X^2} = \left\{x \in M^3_\alpha | \vol(B_{g^\alpha} (x,
\rho_\alpha(x))) \le c_1 \varepsilon_1 [\rho_\alpha(x)]^3 \right\}
$$
which are described as in \eqref{eq6.1} above and $ \{
\rho_\alpha(x) \}$ satisfy inequality \eqref{eq6.3} below.
\end{theorem}
\begin{proof}

Recall that if the pointed Riemannian manifolds $\{((M^3_\alpha,
\rho^{-2}_\alpha g_{ij}^\alpha), x_\alpha) \}$ converge to $(X^k,
x_\infty)$, then by our assumption we have $\diam(X^k) \ge 1$ and
hence
$$
1 \le \dim[X^k] \le 2.
$$

Therefore, there are two cases: (1) $\dim [X^k] = 1$ and (2)
$\dim[X^k] = 2$.

\medskip
\noindent {\it Case 1.} $\dim[X^k] = 1$. By the proof of
\autoref{thm5.8}, there exists $\varepsilon_1>0$ such that if
\begin{equation}\label{eq:6.0.1}
\dis_{GH}(B_{(M^3_\alpha, \rho^{-2}_\alpha g_{ij}^\alpha)}(x, 1) ,
[0, \ell])\le \varepsilon_1
\end{equation}

then $B_{(M^3_\alpha, \rho^{-2}_\alpha g_{ij}^\alpha)}(x, 1) $
admits a possibly singular fibration over $[0,1]$, where the
curvature of the $ \rho^{-2}_\alpha g_{ij}^\alpha$ is bounded below
by $-1$. Clearly, $\ell$ must satisfy $1 - \varepsilon_1 \le \ell
\le 2 + \varepsilon_1$.

If the inequality \eqref{eq:6.0.1} holds, then the unit metric ball
is very thin and can be covered by at most
$\frac{4\ell}{\varepsilon_1 }$ many small metric balls $\{
B_{(M^3_\alpha, \rho^{-2}_\alpha g_{ij}^\alpha)}(y_j,
2\varepsilon_1) \}$. By Bishop volume comparison theorem, we have
a volume estimate:
\begin{equation}
\vol[B_{(M^3_\alpha, \rho^{-2}_\alpha g_{ij}^\alpha)}(x, 1)] \le
\frac{4\ell}{\varepsilon_1 } c_0 [\sinh( 2\varepsilon_1)]^3 \le
c_0^* \ell \varepsilon_1^2
\end{equation}

In this case, we set $x \in V_{\alpha, X^1}$.

\noindent {\it Case 2.} $\dim[X^k] = 2$. There are two subcases for

\begin{equation}
\dis_{GH}(B_{(M^3_\alpha, \rho^{-2}_\alpha g_{ij}^\alpha)}(x, 1),
B_{X^2}(x_\infty, 1))\le \varepsilon_2
\end{equation}

\smallskip
\noindent
{\it Subcase 2a.} The metric disk $B_{X^2}(x_\infty, 1)$ has small area.

In this subcase, one can choose constant $\varepsilon_2>0$ such that
if a metric disk in $X^2$ with radius $r$ satisfies $1/2\le r\le 1$
and area $\le \varepsilon_2$ then by comparison theorems one can
prove
\begin{equation}
\dis_{GH}(B_{X^2}(x_\infty, 1), [0,\ell])\le \varepsilon_1/3
\end{equation}
for some $\ell>0$. It follows that
\begin{equation*}
\begin{split}
\dis_{GH}(B_{(M^3_\alpha, \rho^{-2}_\alpha g_{ij}^\alpha)
}(x_\alpha, 1),[0,\ell])& \le \dis_{GH}(B_{(M^3_\alpha,
\rho^{-2}_\alpha g_{ij}^\alpha)}(x_\alpha,
1),B_{X^2}(x_\infty, 1) )\\
&\quad+ \dis_{GH}(B_{X^2}(x_\infty, 1), [0,\ell])\\
&\le\frac{\varepsilon_1}{3}+\frac{\varepsilon_1}{3}<\varepsilon_1
\end{split}
\end{equation*}

We can now view the ball $B_{\rho^{-2}_\alpha g_{ij}^\alpha
}(x_\alpha, 1)$ is fibred over $[0,\ell]$, instead of fibred over
the disk $B_{X^2}(x_\infty, 1)$.

In this subcase, we let $x_\alpha$ be in both $ V_{\alpha, X^1}$ and
$V_{\alpha, X^2}$.

\smallskip
\noindent {\it Subcase 2b.} The metric disk $B_{X^2}(x_\infty, 1)$
is very fat (thick).

In this subcase, we may assume that the length of metric circle
$\partial B_{X^2}(x_\infty, r)$ is 100 times greater than
$\varepsilon_1$ for $r \in [\frac 12, 1]$. Hence, the length of
collapsing fiber $F^{-1}_\alpha(z)$ for $ z \in \partial
B_{X^2}(x_\infty, r)$. Thus, the metric spheres $\{ \partial
B_{(M^3_\alpha, \rho^{-2}_\alpha g_{ij}^\alpha) }(x_\alpha, r)\}$
collapse in only one direction. Because $G_\alpha$ is an almost
metric submersion due to Perelman's semi-flow convergence theorem,
(cf. \autoref{prop1.14} above), volumes of metric balls collapse at
an order $o( \varepsilon_1)$:
\begin{equation}
\vol[ B_{(M^3_\alpha, \rho^{-2}_\alpha g_{ij}^\alpha)}(x, r) ] \le
c_1 \varepsilon_1 r^3
\end{equation}
for $r \in [\frac 12, 1]$. Moreover, the free homotopy class of
collapsing fibers is {\it unique} in the annular region
$A_{(M^3_\alpha, \rho^{-2}_\alpha g_{ij}^\alpha) }(x_\alpha, \frac
14, 1)$.

By the discussion above, we have the following decomposition of the
manifold $M^3_\alpha$ for $\alpha$ large. More precisely, for
sufficiently large $\alpha$, $M^3_\alpha$ has decomposition
$$
M^3_\alpha = V_{\alpha,X^1}\cup V_{\alpha, \interior (X^2)} =
V_{\alpha, \interior (X^1) }\cup V_{\alpha, \interior (X^2)}\cup
W_{\alpha}
$$
where $W_\alpha$ contains collapsing parts near $\partial X^1$ and
$\partial X^2$ described in \autoref{section5} above. This completes
the proof.
\end{proof}

Let us now recall an observation of Morgan-Tian about the choices of
functions $\{\rho_{\alpha}(x)\}$ and volume collapsing factors
$\{w_{\alpha}\}$.

\begin{prop}[Morgan-Tian \cite{MT2008}]\label{prop6.1}
Let $M^3_{\alpha}, w_{\alpha}$ and $\rho_{\alpha}(x)$ be as in
Theorem 0.1'. Suppose that none of connected components of
$M^3_{\alpha}$ admits a smooth Riemannian metric of non-negative
sectional curvature. Then, by changing $w_{\alpha}$ by a factor
$\hat c$ independent of $\alpha$, we can choose $\rho_{\alpha}$ so
that
\begin{equation}\label{eq6.2}
\rho_{\alpha}(x)\le \dim(M^3_{\alpha})
\end{equation}
and
\begin{equation}\label{eq6.3}
\rho_{\alpha}(y)\le \rho_{\alpha}(x)\le 2\rho_{\alpha}(y)
\end{equation}
for all $y\in B_{g_{\alpha}}(x,\frac{\rho_{\alpha}}{2})$
\end{prop}
\begin{proof}
(cf. Proof of Lemma 1.5 of \cite{MT2008}) Since there is a typo in
Morgan-Tian's argument, for convenience of readers, we reproduce
their argument with minor corrections.

Let us choose
$$C_0=\max_{0\le s\le 1}\left\{\frac{s^3}{4 \pi\int_0^s(\sinh
u)^2du}\right\}$$
For each connected component $N^3_{\alpha,j}$ of
$M^3_{\alpha}$, the Riemannian sectional curvature of
$N^3_{\alpha,j}$ can not be everywhere non-negative. Thus, for each
$x\in N^3_{\alpha,j}$, there is a maximum $r=r_{\alpha}(x)\ge
\rho_{\alpha}(x)$ such that sectional curvature of $g_{\alpha}$ on
$B_{g_{\alpha}}(x,r)$ is greater than or equal to $-\frac{1}{r^2}$.
Consequently, curvature of $\frac{1}{r^2}g_{\alpha}$ on
$B_{\frac{1}{r^2}g_{\alpha}}(x,1)$ is $\ge -1$. By Bishop-Gromov
relative comparison theorem and our assumption
$\vol(B_{g_{\alpha}}(x,\rho_{\alpha}))\le
w_{\alpha}\rho_{\alpha}^3$, we have
\begin{equation} \label{eq6.4}
\begin{split}
& \vol(B_{g_{\alpha}}(x,r))=r^3\vol [B_{\frac{1}{r^2}g_{\alpha}}(x,1)]
\le r^3\frac{V_{\rm hyp}(1)}{V_{\rm hyp}(\frac{\rho_\alpha}{r})}
\vol(B_{\frac{1}{r^2}g_{\alpha}}(x,\frac{\rho_\alpha}{r})) \\
& = \frac{V_{\rm hyp}(1)}{V_{\rm hyp}(\frac{\rho_\alpha}{r})}
\vol(B_{g_{\alpha}}(x, \rho_\alpha ))
\le \frac{V_{\rm hyp}(1)}{V_{\rm hyp}(\frac{\rho_\alpha}{r})} w_{\alpha}\rho_{\alpha}^3 \\
& = \frac{V_{\rm hyp}(1)}{V_{\rm hyp}(\frac{\rho_\alpha}{r})}
(\frac{\rho_{\alpha}}{r})^3\frac{1}{(\frac{\rho_{\alpha}}{r})^3} w_{\alpha}\rho_{\alpha}^3
\le C_0 V_{\rm hyp}(1) w_{\alpha}r^3
\end{split}
\end{equation}
where $V_{\rm hyp}(s)$ is the volume of the ball $B_{\mathbb
H^3}(p_0,s)$ of radius $s$ in $3$-dimensional hyperbolic space with
constant negative curvature $-1$.

Let $
\hat C=C_0 V_{\rm hyp}$
be a constant number independent of $\alpha$, and $Rm_g$ be the
sectional curvature of the metric $g$. We now replace
$\rho_\alpha(x)$ by
\begin{equation}\label{eq6.6}
\rho_\alpha^*(x)=\max\{r|
Rm_{g_{\alpha}}|_{B_{g_{\alpha}}(x,r)}\ge-\frac{1}{r^2}\}
\end{equation}
Our new choice $\rho_\alpha^*(x)$ clearly satisfies

\begin{equation} \label{eq6.7}
\left\{ \begin{aligned}
       \rho_\alpha^*(x) &\le \diam(M^3_{\alpha}) \\
                  \frac{1}{2}\rho_\alpha^*(x)\le&\rho_\alpha^*(y)\le 2\rho_\alpha^*(x)
                          \end{aligned} \right.
                          \end{equation}
for all $y\in B_{g_{\alpha}}(x,\frac{\rho_{\alpha}^*}{2})$ and
$x\in M^3_{\alpha}.$
\end{proof}

Our next goal in this section is to show that we can perturb our
decomposition above along their boundaries so that the new
decomposition admits an F-structure in the sense of Cheeger-Gromov,
and hence $M^3_{\alpha}$ is a graph manifold for sufficiently large
$\alpha$. Let us begin with a special case.

\subsection{Perelman's collapsing theorem for a special case} \

\medskip
\
In this subsection, we prove Theorem 0.1' for a special case when
Perelman's fibrations are assumed to be circle fibrations or toral
fibrations. In next sub-section, we reduce the general case to the
special case, (i.e. the case when no spherical fibration occurred).

As we pointed out earlier, for the proof of \autoref{thm0.1} it is
sufficient to verify that $M^3_{\alpha}$ admits an F-structure of
positive rank in the sense of Cheeger-Gromov for sufficiently large
$\alpha$, (cf. \cite{CG1986}, \cite{CG1990}, \cite{R1993}).

Recall that an F-structure on a $3$-manifold $M^3$ is a collection
of charts $\{(U_i,V_i,T^{k_i})\}$ such that $T^{k_i}$ acts on a
finite normal cover $V_i$ of $U_i$ and the $T^{k_i}$-action on $V_i$
commutes with deck transformation on $V_i$. Moreover the tori-group
actions satisfy a compatibility condition on any possible overlaps.

Let us state compatibility condition as follows.
\begin{definition}[Cheeger-Gromov's compatibility condition]\label{def6.2}
$\quad$\\ Let $\{(U_i,V_i,T^{k_i})\}$ be a collection of charts as
above. If, for any two charts $(U_i,V_i,T^{k_i})$ and $ (U_j,V_j,
T^{k_j})$ with non-empty intersection $U_i\cap U_j\ne \varnothing$
and with $k_i\le k_j$, the $T^{k_i}$ actions commutes with the
$T^{k_j}$-actions on a finite normal cover of $U_i\cap U_j$ after
re-parametrization if needed, then the collection $\{(U_i,V_i,
T^{k_i}) \}$ is said to satisfy Cheeger-Gromov's compatibility
condition.
\end{definition}

For the purpose of this paper, since our manifolds under
consideration are $3$-dimension, the choice of free tori $T^{k_i}$
actions must be either circle action or $2$-dimensional torus
action. Thus we only have to consider following three possibilities:
\begin{enumerate}[(i)]

\item Both overlapping charts $ (U_i,V_i, S^1)$ and $ (U_j,V_j, S^1)$
admit almost free circle actions;

\item
Both overlapping charts $ (U_i,V_i, T^2)$ and $ (U_j,V_j, T^2)$
admit almost free torus actions;
\item
There are a circle action $(U_i,V_i, S^1)$ and a torus-action $
(U_j,V_j, T^2)$ with non-empty intersection $U_i\cap U_j\ne
\varnothing$.
\end{enumerate}

Let us begin with the sub-case (ii).

\begin{prop}\label{prop6.3}
Let $U_{\alpha,i_1}$ and $U_{\alpha,i_2}$ be two overlapping open
subsets in $U_{\alpha, X^1}$ with toral fibers $F^2_{i_1}\cong F^2
_{i_2} \cong T^2$ described in \autoref{section4}. Suppose that $
U_{\alpha,i_1}\cap U_{\alpha,i_2}\ne\varnothing$. Then we can modify
charts $(U_{\alpha,i_1},V_{\alpha,i_1},T^2)$ so that the perturbed
torus-actions on modified charts satisfy the Cheeger-Gromov's
compatibility condition.
\end{prop}
\begin{proof}
Recall that in \autoref{section4} we worked on the following
diagrams
\begin{diagram}
U_{\alpha,i_1}&\rTo^{f_{i_1}}&\mathbb{R}&&U_{\alpha,i_2}&\rTo^{f_{i_2}}&\mathbb{R}\\
\dTo_{\text{G-H}}&&\dCorresponds&{\quad {\rm and} \quad}&\dTo_{\text{G-H}}&&\dCorresponds\\
X^1_{i_1}&\rTo^{r_{\rho_{i_1}}}&\mathbb{R}&&X^1_{i_2}&\rTo^{r_{\rho_{i_2}}}&\mathbb{R}
\end{diagram}

Without loss of generality, we may assume that
$X_{i_1}=(a_{i_1},b_{i_1})$ and $X_{i_2}=(a_{i_2},b_{i_2})$ with
non-empty intersection $X_{i_1}\cap X_{i_2}=(a_{i_2},b_{i_1})\ne
\varnothing$, where $a_{i_1}< a_{i_2}< b_{i_1}< b_{i_2}$. Let
$\{\lambda_1,\lambda_2\}$ be a partition of unity of
$[a_{i_1},b_{i_2}]$ corresponding to the open cover
$\{(a_{i_1},b_{i_1}),(a_{i_2},b_{i_2}))\}$. After choosing $\pm$
sign and $\lambda_1, \lambda_2$ carefully, we may assume that
$$
\hat f(t)=\lambda_1(t)f_{i_1}(t)\pm \lambda_2(t)f_{i_2}(t)
$$
will not have any critical point $t\in [a_{i_1},b_{i_2}]$. We also
require that $\hat f: X_{i_1}\cup X_{i_2}\to \mathbb R$ is a {\it
regular} function in the sense of Perelman (i.e. $\hat f$ is a
composition of distance function given by
\autoref{def1.9}-\ref{def1.10}). Thus, we can lift the admissible
function $\hat f$ to a function
$$f_{\alpha}: U_{\alpha,i_1}\cup U_{\alpha,i_2}\to \mathbb R $$
such that
$$\nabla f_{\alpha}|_{y}\ne 0$$
for $y\in [U_{\alpha,i_1}\cup U_{\alpha,i_2}]$, $f_{\alpha}\cong
f_{\alpha,i_1}$ on $[U_{\alpha, i_1}^*-U_{\alpha,i_2}^*]$ and
$f_{\alpha}\cong f_{\alpha, i_2}$ on
$[U_{\alpha,i_2}^*-U_{\alpha,i_1}^*]$, where $U_{\alpha, i_1}^*$ is
a perturbation of $U_{\alpha, i_1}$ and $U_{\alpha, i_2}^*$ is a
perturbation of $U_{\alpha, i_2}$. Thus there is a new perturbed
torus fibration.
$$T^2\to U_{\alpha,i_1}\cup U_{\alpha,i_2}\to Y^1. $$
This completes the proof.
\end{proof}

Similarly, we can modify admissible maps to glue two circle actions
together.

\begin{prop}\label{prop6.4}
Let $U_{\alpha,i_1}$ and $U_{\alpha,i_2}$ be two open sets contained
in $U_{\alpha,X^2_{reg}}$ corresponding to a decomposition of
$M^3_\alpha$ described in \autoref{section1}. Suppose that two
charts $\{(U_{\alpha,i_1}, V_{\alpha,i_1}, S^1),(U_{\alpha,i_2},
V_{\alpha,i_2}, S^1)\}$ have non-empty overlap $U_{\alpha,i_1}\cap
U_{\alpha,i_2}\ne \varnothing$. Then the union $U_{\alpha,i_1}\cup
U_{\alpha,i_2}$ admits a global circle action after some
modifications when needed.
\end{prop}
\begin{proof}
We may assume that, in the limiting processes of $U_{\alpha,i_1} \to
X^2_1 $ and $U_{\alpha,i_1} \to X^2_1 $, both limiting surfaces
$X^2_1$ and $X^2_2$ are fat (having relatively large area growth).
For the remaining cases when either $X^2_1$ and $X^2_2$ are
relatively thin, by the proof of \autoref{thm6.0} we can view either
$ U_{\alpha,i_1} $ or $ U_{\alpha,i_1}$ is a portion of $V_{\alpha,
X^1}$ instead. These remaining situations can be handled by either
\autoref{prop6.3} above or \autoref{prop6.5} below respectively.

As we pointed out in \autoref{section1}-\ref{section2} the
exceptional orbits are isolated. Without loss of generality, we may
assume that there is no exceptional circle orbits occurs in the
overlap $U_{\alpha,i_1}\cap U_{\alpha,i_2}$.

Since both limiting surfaces $X^2_1$ and $X^2_2$ are very {\it fat},
the lengths of metric circles $\partial B_{X^2_i}(x_{\infty, i}, r)$
is at least 100 times larger than $\varepsilon_1$, which is great
than the length of collapsing fibers, where $ r \in [\frac 12, 1] $.
Hence, on the overlapping region $U_{\alpha,i_1}\cap
U_{\alpha,i_2}$, two collapsing fibres are freely homotopic each
other in the shell-type region $ A_{(M^3_\alpha, \rho^{-2}_\alpha
g_{ij}^\alpha) }(x_\alpha,\frac 14, 1)$.

Let us consider the corresponding diagrams again.
\begin{diagram}
U_{\alpha,i_1} &\rTo^{\psi_{i_1}}&\mathbb{R}^2&&U_{\alpha,i_2} &\rTo^{\psi_{i_2}}&\mathbb{R}^2\\
\dTo_{\text{G-H}}&&\dCorresponds&{\quad {\rm and} \quad}&\dTo_{\text{G-H}}&&\dCorresponds\\
X^2_{i_1}&\rTo&\mathbb{R}^2&&X^2_{i_2}&\rTo&\mathbb{R}^2
\end{diagram}

We can use $2\times 2$ matrix-valued function $A(x)$ to construct a
new {\it regular} map

$$\psi_{\alpha}: U_{\alpha,i_1}\cup U_{\alpha,i_2} \to \mathbb R^2$$
from $\{\psi_{\alpha,i_1},\psi_{\alpha,i_2}\}$ as follows. Let
$$\psi_{\alpha}=\lambda_1 A_1\psi_{i_1}+\lambda_2 A_2\psi_{i_2}$$
where $\{\lambda_1,\lambda_2\}$ is a partition of unity for
$U_{i_1}, U_{i_2}$, $A_{i_1}(x)$ and $A_{i_2}(x)$ are $2\times 2$
matrix-valued functions such that
\begin{enumerate}[(i)]
\item $A_{i_1}(x)\cong \left(
                         \begin{array}{cc}
                           1 & 0 \\
                           0 & 1 \\
                         \end{array}
                       \right)$ is close to the identity on
                       $[U^*_{i_1}-U^*_{i_2}]$.

\item Similarly, $A_{i_2}(x)\cong \left(
                         \begin{array}{cc}
                           1 & 0 \\
                           0 & 1 \\
                         \end{array}
                       \right)$ is close to the identity on
                       $[U^*_{i_2}-U^*_{i_1}]$.

\end{enumerate}
where $U^*_{i_1}$ is a perturbation of $U_{i_1}$ and $U^*_{i_2}$ is
a perturbation of $U_{i_2}$. We leave the details to readers.
\end{proof}

We now discuss relations between circle actions and torus action.

\begin{prop}\label{prop6.5}
Let $U_{\alpha,i_1}\subset U_{\alpha,X^1}$ and
$U_{\alpha,i_2}\subset U_{\alpha,X^2_{reg}}$, where
$\{U_{\alpha,X^1}$, $U_{\alpha,X^2_{reg}}, U_{\alpha,X^0},
W_{\alpha}\}$ is a decomposition of $M^3_{\alpha}$ as in
\autoref{section1}-\ref{section5}. Suppose that $(U_{\alpha,i_1},
V_{\alpha,i_1}, T^2)$ and $(U_{\alpha,i_2}, V_{\alpha,i_2}, S^1)$
have an interface $U_{\alpha,i_1}\cap U_{\alpha,i_2}\ne
\varnothing$. Then, after a perturbation if need, the circle orbits
are contained in torus orbits on the overlap.
\end{prop}
\begin{proof}
Let $f_{\alpha,i_1}: U_{\alpha,i_1}\to \mathbb R$ be the regular
function which induces the chart $(U_{\alpha,i_1}, V_{\alpha,i_1},
T^2)$. Suppose that
$\psi_{\alpha,i_1}=(f_{\alpha,i_2},g_{\alpha,i_2}):
U_{\alpha,i_2}\to \mathbb R^2$ is the corresponding regular map.
Since any component of a regular map must be regular, after
necessary modifications, we may assume that $f^*_{\alpha,i_1}:
U^*_{\alpha,i_1}\to \mathbb R$ is equal to $f^*_{\alpha,i_2}$ in the
modified regular map
$\psi^*_{\alpha,i_2}=(f^*_{\alpha,i_2},g^*_{\alpha,i_2}):
U^*_{\alpha,i_2}\to \mathbb R^2$ on the overlap
$U^*_{\alpha,i_1}\cap U^*_{\alpha,i_2}$. It follows that the
modified $S^1$-orbits are contained in the newly perturbed
$T^2$-orbits. Thus, the modified charts $(U^*_{\alpha,i_1},
V^*_{\alpha,i_1}, T^2)$ and $(U^*_{\alpha,i_2}, V^*_{\alpha,i_2},
S^1)$ satisfy the Cheeger-Gromov's compatibility condition.
\end{proof}

We now conclude this sub-section by a special case of Perelman's
collapsing theorem.

\begin{theorem}\label{thm6.6}
Suppose that $\{(M^3_{\alpha}, g_{\alpha})\}$ satisfies all
conditions stated in \autoref{thm0.1}. Suppose that all Perelman
fibrations are either (possible singular) circle fibrations or toral
fibrations. Then $M^3_\alpha$ must admits a F-structure of positive
rank in the sense of Cheeger-Gromov for sufficiently large $\alpha$.
Consequently, $M^3_\alpha$ is a graph manifold for sufficiently
large $\alpha$.
\end{theorem}
\begin{proof}
By our assumption and \autoref{prop6.1}, there is a collection of
charts $\{(U_{\alpha,i}, V_{\alpha,i}, T^{k_i})\}_{i=1}^n$ such that
\begin{enumerate}[(i)]
\item $1\le k_i\le 2$;
\item $\{U_{\alpha,i}\}_{i=1}^n$ is an open cover of $M^3_\alpha$;
\item $(U_{\alpha,i}, V_{\alpha,i}, T^{k_i})$ satisfies condition
(3.1).
\end{enumerate}

It remains to verify that our collection
$$
\{(U_{\alpha,i}, V_{\alpha,i}, T^{k_i})\}_{i=1}^n
$$
satisfies the Cheeger-Gromov compatibility condition, after some
modifications. Since exceptional circle orbits with non-zero Euler
number are isolated, we may assume that on any possible overlap
$$U_{\alpha,i_1}\cap\cdots \cap U_{\alpha,i_j}\ne \varnothing$$
there is no exceptional circular orbits. Applying
\autoref{prop6.3}-\ref{prop6.5} when needed, we can perturb our
charts so that the modified collection of charts $\{(U^*_{\alpha,i},
V^*_{\alpha,i}, T^{k_i})\}_{i=1}^n$ satisfy the Cheeger-Gromov
compatibility condition. It follows that $M^3_\alpha$ admits an
F-structure $\mathscr{F}^*$ of positive rank. Therefore,
$M^3_\alpha$ is a graph manifold for sufficiently large $\alpha$.
\end{proof}

\subsection{Perelman's collapsing theorem for general case} \

\medskip

\

Our main difficulty is to handle a Perelman fibration with possible
spherical fibers:
$$S^2\to U_{\alpha,i}\to \interior (X^1) $$
because the Euler number of $S^2$ is non-zero and $S^2$ does not
admit any free circle actions.

\begin{prop}\label{prop6.7}
Let $\{(M^3_{\alpha}, g_{\alpha})\}$ be a sequence of Riemannian
$3$-manifolds as in \autoref{thm0.1}. If there is a Perelman
fibration
$$N^2\to U_{\alpha,X^1}\to (a,b)$$
with spherical fiber $N^2\cong S^2$, then $U_{\alpha, X^1}$
must be contained in the interior of $M^3_\alpha$ when $M^3_\alpha$
has non-empty boundary $\partial M^3_\alpha\ne \varnothing$.
\end{prop}
\begin{proof}
According to condition (2) of \autoref{thm0.1}, for each boundary
component $N^2_{\alpha,j}\subset \partial M^3_{\alpha}$, there
is a topologically trivial collar $V_{\alpha,j}$ of length one such
that $V_{\alpha,j}$ is diffeomorphic to $T^2\times [0,1)$. Thus, we
have $U_{\alpha,j}\cap
\partial M^3_\alpha=\varnothing$
\end{proof}

We need to make the following elementary but useful observation.
\begin{prop}\label{cor6.8}
(1) Let $A=\{(x_1,x_2,x_3)\in \mathbb R^3\big |\ \ |\vec x|=1,
|x_3|\le \frac12\}$ be an annulus. The product space $S^2\times
[0,1]$ has a decomposition
$$\big(D^2_+\times [0,1]\big) \cup \big( A\times [0,1] \big) \cup
\big( D^2_-\times [0,1]\big), $$
 where $D^2_{\pm}=\{(x_1,x_2,x_3)\in S^2(1)|
\pm x_3\ge \frac12\}$.

(2) If $\{(M^3_{\alpha}, g_{\alpha})\}$ satisfies conditions of
\autoref{thm0.1}, then, for sufficiently large $\alpha$,
$M^3_\alpha$ has a decomposition $\{U_{\alpha,j}\}_{j=1}^{m_\alpha}$
such that

(2.a) either a finite normal cover $V_{\alpha,j}$ of
$U_{\alpha,j}$ admits a free $T^{k_i}$-action with $k_i=1,2$.

(2.b) or $U_{\alpha,j}$ is homeomorphic to a finite solid cylinder
$D^2\times [0,1]$ with $U_{\alpha,j}\cap \partial M^3_\alpha=
\varnothing$.

(3) If $U_{\alpha,j}$ is a finite cylinder as in (2.b), then it must
be contained in a chain $\{U_{\alpha,j_1},\cdots, U_{\alpha,j_m}\}$
of finite solid cylinders such that their union
$$\hat W_{\alpha,h_j}=\bigcup_{i=1}^m U_{\alpha,j_i}$$
is homeomorphic to a solid torus $D^2\times S^1$.
\end{prop}
\begin{proof}
The first two assertions are trivial. It remains to verify the third
assertion. By our construction, if $U_{\alpha,j}$ is a finite
cylinder homeomorphic to $D^2\times I$, then $U_{\alpha,j}$ meets
$V_{\alpha,X^1}$ exactly in $D^2\times \partial I$. Moreover, such a
finite solid cylinder $U_{\alpha,j}$ is contained in a chain
$\{U_{\alpha,j_1},\cdots,U_{\alpha,j_m}\}$ of solid cylinders. It is
easy to see that the union $\hat W_{\alpha, h_j}$ is homotopic to
its core $S^1$. i.e. $\hat W_{\alpha,h_j}$ is homeomorphic to a
solid torus $D^2\times S^1$.
\end{proof}

We are ready to complete the proof of Perelman's collapsing theorem.

\begin{proof}[{\bf Proof of \autoref{thm0.1}}]

Using \autoref{prop6.1} we can choose $\rho_{\alpha}$ so that
$$
\rho_{\alpha}(x)\le \dim(M^3_{\alpha})
$$
and
$$
\rho_{\alpha}(y)\le \rho_{\alpha}(x)\le 2\rho_{\alpha}(y)
$$
for all $y\in B_{g_{\alpha}}(x,\frac{\rho_{\alpha}}{2})$.
Therefore, it is sufficient to verify Theorem 0.1' instead.

By our discussion above, for sufficiently large $\alpha$, our
$3$-manifold $M^3_\alpha$ admits a collection of (possibly singular)
Perelman fibration. Thus $M^3_\alpha$ has a decomposition
$$
M^3_{\alpha}= V_{\alpha,X^0}\cup V_{\alpha,X^1}\cup V_{\alpha,
\interior (X^2)} \cup W_\alpha
$$
For each chart in $ V_{\alpha,X^0}\cup V_{\alpha, \interior (X^2)}$,
it admits a Seifert fibration. However, remaining charts could be
homeomorphic to $S^2\times I$, $[\mathbb {RP}^3-D^3]$ or a solid
cylinder $D^2\times I$. It follows from \autoref{cor6.8}(3) that
$M^3_\alpha$ has a more refined decomposition
$$M^3_\alpha=\bigcup_{i=1}^m U_{\alpha,j}$$
such that
\begin{enumerate}[(i)]
\item either a finite normal cover $V_{\alpha,j}$ of $U_{\alpha,j}$
admits a free $T^{k_j}$-action with $k_j=1,2$;
\item or $U_{\alpha,j}$ is homeomorphic to a solid torus $D^2\times
S^1$, which is obtained by a chain of solid cylinders.
\end{enumerate}

Observe that possibly exceptional circular orbits are isolated.
Moreover if $\{U_{\alpha,j}\}$ are of type (2.b) in
\autoref{cor6.8}, these cores $\{0\}\times S^1$ are isolated as
well.

It remains to verify that our collection of charts
$\{(U_{\alpha,j},V_{\alpha,j}, T^{k_j})\}$ satisfy Cheeger-Gromov
compatibility condition. By observations on exceptional orbits and
cores of solid cylinders, we may assume that on possibly overlap
$$U_{\alpha,j_1}\cap\cdots\cap U_{\alpha,j_k}\ne \varnothing$$
there is no exceptional orbits nor cores of solid cylinders.

We require that $V_{\alpha, \interior (X^2)}$ meets $S^2$-factors as
annular type $A_\alpha$. Thus, if $U_{\alpha, j} \cong S^1 \times
D^2$, we can introduce $T^2$-actions on $\partial U_{\alpha, j}
\cong S^1 \times S^1$ which are compatible with $S^1$-actions on
$A_\alpha$. After modifying our charts as in proofs of
\autoref{prop6.3}-\ref{prop6.5}, we can obtain a new collection of
charts $\{(U^*_{\alpha,j},V^*_{\alpha,j}, T^{k_j})\}$ satisfying the
Cheeger-Gromov's compatibility condition. Therefore, $M^3_\alpha$
admits an F-structure $\mathscr{F}_{\alpha}$ of positive rank. It
follows that $M^3_\alpha$ is a $3$-dimensional graph manifold, (cf.
\cite{R1993}).
\end{proof}

\noindent {\bf Acknowledgement:} Both authors are grateful to
Professor Karsten Grove for teaching us the modern version of
critical point theory for distance functions. Professor Takashi
Shioya carefully proofread every sub-step of our simple proof in the
entire paper, pointing out several inaccurate statements in an
earlier version. He also generously provided us a corrected proof of
\autoref{prop2.2}. We are very much indebted to Professor Xiaochun
Rong for supplying \autoref{thm6.0} and its proof. Hao Fang and
Christina Sormani made useful comments on an earlier version.
Finally, we also appreciate Professor Brian Smyth's generous help
 the exposition in our paper. We thank the referee for his (or her)
 suggestions, which led many improvements.

\bibliographystyle{amsalpha}

\begin{thebibliography}{A}

\bibitem[BBBMP10]{BBBMP2010}
L. Bessi\`{e}res, G. Besson, M. Boileau, S. Maillot, J. Porti,
\textit{Collapsing irreducible 3-manifolds with nontrivial
fundamental group } Invent. Math \textbf{179} (2010) 434-460.

\bibitem[BBI01]{BBI2001}
Dima Burago, Yuri Burago; S. Ivanov, \textit{ A course in metric
geometry}. Graduate Studies in Mathematics, \textbf{33}. American
Mathematical Society, Providence, RI, 2001. xiv+415 pp. ISBN:
0-8218-2129-6.

\bibitem[BGP92]{BGP1992}
Yu Burago; M. Gromov; G. Perelman, \textit{A. D. Aleksandrov spaces
with curvatures bounded below.} (Russian) Uspekhi Mat. Nauk
\textbf{47} (1992), no. 2(284), 3--51, 222; translation in Russian
Math. Surveys \textbf{47} (1992), no. 2, 1--58


\bibitem[CalC92]{CalC92}
E. Calabi and J. Cao, \textit{ Simple closed geodesics on convex
surfaces.} Journal of Differential Geometry \textbf{36} (1992),
517-549.

\bibitem[CaZ06]{CZ2006}
H. Cao, X. Zhu, \textit{A complete proof of the Poincar\'{e} and
geometrization conjectures--- application of the Hamilton-Perelman
theory of the Ricci flow} Asian J. Math. \textbf{10} (2006), no. 2,
165--492.

\bibitem[CCR01]{CCR2001}
J. Cao, J. Cheeger and X. Rong, \textit{Splittings and Cr-structures
for manifolds with nonpositive sectional curvature.} Invent. Math.
\textbf{144} (2001), no. 1, 139--167.

\bibitem[CCR04]{CCR2004}
J. Cao, J. Cheeger and X. Rong, \textit{Local splitting structures
on nonpositively curved manifolds and semirigidity in dimension 3.}
Comm. Anal. Geom. \textbf{12} (2004), no. 1-2, 389--415.

\bibitem[CDM09]{CMD2009}
J. Cao, Bo Dai and Jiaqiang Mei, \textit{An optimal extension of
Perelman's comparison theorem for quadrangles and its applications.}
Advanced Lectures in Mathematics, volume 11 (2009), page 39-59, In
book ``Recent Advances in Geometric Analysis", edited by Y. Lee, C-S
Lin, M-P Tsui. ISBN: 978-7-04-027602-2, 229 pages, Higher
Educational Press and International Press.








\bibitem[CG72]{CG72}
J. Cheeger, D. Gromoll, \textit{On the structure of complete
manifolds of nonnegative curvature, } Annals of Math., vol \textbf{
96}, (1972), 413--443.

\bibitem[CG86]{CG1986}
J. Cheeger, M. Gromov, \textit{Collapsing Riemannian manifolds while
keeping their curvature bounded. I.} J. Differential Geom.
\textbf{23} (1986), no. 3, 309--346.

\bibitem[CG90]{CG1990}
J. Cheeger, M. Gromov, \textit{Collapsing Riemannian manifolds while
keeping their curvature bounded. II}. J. Differential Geom.
\textbf{32}(1990), 269-298.




\bibitem[CM05]{CM2005}
T. Colding, W. Minicozzi II, \textit{Estimates for the extinction
time for the Ricci flow on certain $3$-manifolds and a question of
Perelman}, J. Amer. Math. Soc. \textbf{18} (2005), no. 3, 561--569.

\bibitem[CM08a]{CM2008a}
T. Colding, W. Minicozzi II, \textit{Width and mean curvature flow,}
Geom. Topol. 12 (2008), no. 5, 2517--2535.

\bibitem[CM08b]{CM2008b}
T. Colding, W. Minicozzi II, \textit{Width and finite extinction
time of Ricci flow,} Geom. Topol. 12 (2008), no. 5, 2537--2586.

\bibitem[Gro81]{Gro1981}
M. Gromov, \textit{Curvature, diameter and Betti numbers,} Comment.
Math. Helv., vol \textbf{56} (1981), no. 2, 179--195.

\bibitem[Grv93]{Grv1993} K. Grove,
\textit{ Critical point theory for distance functions. Differential
geometry}, In book ``Riemannian geometry" (Los Angeles, CA, 1990),
357--385, Proc. Sympos. Pure Math., vol 54, Part 3, Amer. Math.
Soc., Providence, RI, 1993.

\bibitem[GS77]{GS1977}
K. Grove and K. Shiohama, \textit{A generalized sphere theorem,} Ann
of Math, \textbf{106} (1977), 201-211.

\bibitem[GW97]{GW1997}
K. Grove, F. Wilhelm, \textit{Metric constraints on exotic spheres
via Alexandrov geometry}, J. Reine Angew. Math. \textbf{487} (1997),
201--217.

\bibitem[Ham86]{H1986}
R. S. Hamilton, \textit{Four-manifolds with positive curvature
operator.} J. Differential Geom. \textbf{24} (1986), no. 2,
153--179.

\bibitem[Ham99]{H1999}
R. S. Hamilton, \textit{Non-singular solutions of the Ricci flow on
three-manifolds.}
 Comm. Anal. Geom. \textbf{7} (1999), no. 4, 695--729.

\bibitem[Kap02]{Kap2002}
V. Kapovitch, \textit{Regularity of limits of noncollapsing
sequences of manifolds}, Geom. Funct. Anal. \textbf{12} (2002), no.
1, 121--137.

\bibitem[Kap05]{Kap2005}
V. Kapovitch, \textit{Restrictions on collapsing with a lower
sectional curvature bound}, Math. Z. \textbf{249}(2005), no. 3,
519--539

\bibitem[Kap07]{Kap2007}
V. Kapovitch, \textit{Perelman's Stability Theorem}, in ``Surveys in
differential geometry", Surveys in Differential Geometry, volume
\textbf{11}. International Press, Somerville, MA, 2007. xii+347 pp.
ISBN: 978-1-57146-117-9, pages 103-136.

\bibitem[KPT09]{KPT2009}
V. Kapovitch, A. Petrunin, W. Tuschmann, \textit{Nilpotency, Almost
Nonnegative Curvature and Gradient Flow on Alexandrov spaces},
arXiv:math/0506273v4 [math.DG], to be appear in Annals of
Mathematics.

\bibitem[KL08]{KL2008}
B. Kleiner, J. Lott, \textit{Notes on Perelman's papers}, Geometry
and Topology, \textbf{ 12} (2008) 2587-2855.

\bibitem[KL10]{KL2010}
B. Kleiner, J. Lott, \textit{Locally Collapsed $3$-Manifolds}, in
preparation.


\bibitem[MT07]{MT2007}
J. Morgan , G. Tian, \textit{Ricci Flow and the Poincar\'{e}
Conjecture,} Clay Mathematics Monographs, vol.\textbf{3}, American
Math. Society, Providence, 521 pages, 2007.

\bibitem[MT08]{MT2008}
J. Morgan, G. Tian, \textit{Completion of the Proof of the
Geometrization Conjecture, } arXiv:0809.4040v1 [math.DG]

\bibitem[MT10]{MT2010}
J. Morgan, G. Tian, \textit{Ricci flow and the Geometrization
Conjecture}, in preparation.



\bibitem[Per91]{Per1991}
G. Perelman, \textit{A. D. Alexandrov's spaces with curvatures
bounded from below. II}, Preprint.

\bibitem[Per94]{Per1994}
G. Perelman, \textit{Elements of Morse theory on Aleksandrov
spaces}, (Russian. Russian summary) Algebra i Analiz \textbf{5}
(1993), no. 1, 232--241; translation in St. Petersburg Math. J.
\textbf{5} (1994), no. 1, 205--213.


\bibitem[Per97]{Per1997}
G. Perelman, \textit{Collapsing with no proper extremal subsets},
Comparison geometry (Berkeley, CA, 1993--94), 149--155, Math. Sci.
Res. Inst. Publ., \textbf{30}, Cambridge Univ. Press, Cambridge,
1997.

\bibitem[Per02]{Per2002}
G. Perelman, \textit{The entropy formula for the Ricci flow and its
geometric applications}, arXiv:math/0211159v1 [math.DG].

\bibitem[Per03a]{Per2003a}
G. Perelman, \textit{Ricci flow with surgery on three-manifolds},
arXiv:math/0303109v1 [math.DG].

\bibitem[Per03b]{Per2003b}
G. Perelman, \textit{Finite extinction time for the solutions to the
Ricci flow on certain three-manifolds}, arXiv:math/0307245v1
[math.DG].

\bibitem[PP94]{PP1994}
G. Perelman, A. Petrunin, \textit{Extremal subsets in Aleksandrov
spaces and the generalized Liberman theorem}, Algebra i Analiz
\textbf{5} (1993), no. 1, 242--256; translation in St. Petersburg
Math. J. \textbf{5} (1994), no. 1, 215--227.

\bibitem[Petr07]{Petr2007}
A. Petrunin, \textit{Semiconcave functions in Alexandrov's
geometry.} Surveys in differential geometry. Surveys in Differential
Geometry, volume \textbf{11}. International Press, Somerville, MA,
2007. xii+347 pp. ISBN: 978-1-57146-117-9, pages 137--201.


\bibitem[Rong93]{R1993} X. Rong,
\textit{The limiting eta invariants of collapsed three-manifolds}.
J. Differential Geom. \textbf{37} (1993), no. 3, 535--568.


\bibitem[Shen96]{Shen1996}
Z. Shen, \textit{Complete manifolds with nonnegative Ricci curvature
and large volume growth}, Invent. Math. 125 (1996), no. 3, 393--404.


\bibitem[Shio94]{Shio94}
T. Shioya, \textit{Mass of rays in Alexandrov spaces of nonnegative
curvature}, Comment. Math. Helv. 69 (1994), no. 2, 208--228.

\bibitem[SY00]{SY2000}
T. Shioya, T. Yamaguchi, \textit{Collasping three-manifolds under a
lower curvature bound}, J. Differential Geom.
\textbf{56}(2000),1-66.

\bibitem[SY05]{SY2005}
T. Shioya, T. Yamaguchi, \textit{Volume collapsed three-manifolds
with a lower curvature bound}, Math. Ann. \textbf{333}(2005),
131-155.



\bibitem[Wu79]{Wu79}
H. Wu, An elementary method in the study of non-negative curvature,
Acta Math. vol. \textbf{142} (1979), 57-78.


\bibitem[Yama09]{Yama2009}
T. Yamaguchi, \textit{Collapsing and essential covering}, preprint
2009.

\end{thebibliography}

\bigskip

\noindent
{\small Jianguo Cao and Jian Ge}\\
{\small Department of Mathematics} \\
{\small University of Notre Dame}\\
{\small Notre Dame, IN 46556, USA} \\
{\small jcao@nd.edu and jge@nd.edu} \\

\end{document}